%% file: ModuliColored.tex
\documentclass[11pt]{amsart}
\usepackage{geometry}                
\usepackage{graphicx}
\usepackage{amsmath}%
\usepackage{amsfonts}%
\usepackage{amssymb}
\usepackage{amsthm}
\usepackage{subcaption}
\usepackage{cite}
\usepackage{hyperref}
\usepackage[capitalize]{cleveref}
\usepackage{tikz}
\usetikzlibrary{matrix,arrows}
\usepackage{url}

\theoremstyle{plain}
\newtheorem{theorem}{Theorem}[section]
\newtheorem{lemma}{Lemma}[section]
\newtheorem{corollary}{Corollary}[section]
\newtheorem{proposition}{Proposition}[section]

\theoremstyle{definition}
\newtheorem{definition}{Definition}[section]

\newtheorem{example}{Example}[section]

\theoremstyle{remark}
\newtheorem{remark}{Remark}[section]

\numberwithin{equation}{section}

\crefname{pluralequation}{Eqs.}{Eqs.}
\Crefname{pluralequation}{Eqs.}{Eqs.}

\newcommand{\lb}{[\![}
\newcommand{\rb}{]\!]}
\newcommand{\Cour}[1]{\lb #1\rb}
\newcommand{\la}{\langle}
\newcommand{\ra}{\rangle}
\newcommand{\mf}{\mathfrak}
\newcommand{\mbb}{\mathbb}
\newcommand{\mbf}{\mathbf}
\newcommand{\mc}{\mathcal}
\newcommand{\on}{\mathsf}
\newcommand{\ol}{\overline}
\newcommand{\Lied}{\mathcal{L}}
\newcommand{\ad}{\mathbf{ad}}
\newcommand{\Ad}{\mathbf{Ad}}

\newcommand{\g}{\mathfrak{g}}
\newcommand{\h}{\mathfrak{h}}
\newcommand{\gr}{\on{gr}}
\newcommand{\ann}{\on{ann}}
\newcommand{\sect}{\mbf{\Gamma}}
\renewcommand{\d}{{\mbox{d}}}
\newcommand{\Hom}{\operatorname{Hom}}
\newcommand{\mmat}[2][3em]{\matrix (#2) [matrix of math nodes, row sep=#1,
  column sep=#1, text height=1.5ex, text depth=0.25ex]}
\tikzset{node distance=2cm, auto}

\title[Symplectic and Poisson Geometry of Moduli Spaces]{Symplectic and Poisson geometry of the moduli spaces of flat connections over quilted surfaces.}

\author{David Li-Bland}
\author{Pavol \v{S}evera}

\address{Department of Mathematics, University of California, Berkeley}
\email{libland@math.berkeley.edu}
\address{Department of Mathematics, Universit\'{e} de Gen\`{e}ve, Geneva, Switzerland, on leave
from Dept. of Theoretical Physics, FMFI UK, Bratislava, Slovakia}
\email{pavol.severa@gmail.com}
\thanks{D.L-B. was supported by the National Science Foundation under Award No. DMS-1204779.
\\
P.\v S. was partially supported by the Swiss National Science Foundation (grants 140985 and
141329).}

\begin{document}
\begin{abstract}
In this paper we study the symplectic and Poisson geometry of moduli spaces of flat connections over \emph{quilted surfaces}. These are surfaces where the structure group varies from region to region in the surface, and where a reduction (or relation) of structure occurs along the boundaries of the regions. Our main theoretical tool is a new form moment-map reduction in the context of Dirac geometry. This reduction framework allows us to extend the results of \cite{LiBland:2012vo,Severa:2011ug} to allow more general relations of structure groups, and to investigate both the symplectic and Poisson geometry of the resulting moduli spaces from a unified perspective.

The moduli spaces we construct in this way include a number of important examples, including Poisson Lie groups and their Homogeneous spaces, moduli spaces for meromorphic connections over Riemann surfaces (following the work of Philip Boalch), and various symplectic groupoids. Realizing these examples as moduli spaces for \emph{quilted surfaces} provides new insights into their geometry.
 
\end{abstract}
\maketitle

\tableofcontents

\section{Introduction and Summary of Results}
Suppose that $G$ is a Lie group whose Lie algebra, $\g$, is endowed with a $G$-invariant inner product, $\la\cdot,\cdot\ra$. Suppose that $\Sigma$ is a closed oriented surface, and $P\to \Sigma$ is a principal $G$-bundle. Let $\mc{A}_{flat}(P\to\Sigma)$ denote the space of flat connections on $P$.
Atiyah and Bott \cite{Atiyah:1983dt} showed that the moduli space
$$\mc{M}(P\to\Sigma):=\mc{A}_{flat}(P\to\Sigma)/\text{Aut}(P)$$
 of flat connections on $P$ carries a symplectic structure. Their construction involves infinite dimensional symplectic reduction. 
 Somewhat later, Alekseev, Malkin, and Meinrenken introduced quasi-Hamiltonian geometry \cite{Alekseev97}, equipping it with a toolkit of fusion and reduction operations, in order to provide a finite dimensional construction of this moduli space. 
 Boalch \cite{Boalch:2009tn} enlarged the quasi-Hamiltonian toolkit, introducing the fission operation, which enables a finite dimensional construction of the moduli space of flat connections with prescribed irregular singularities. Interestingly, this new fission operation also allowed Boalch to associate Poisson/sympletic/quasi-Hamiltonian spaces of connections to surfaces with different structure groups in different regions. Moreover, these techniques enabled Boalch to interpret additional Poisson spaces, including examples of Poisson Lie groups \cite{Boalch:2001fw,Boalch:2001gw,Boalch:2009tn,Boalch:2011vt,Boalch:2007ty,Boalch:2001gw} and Lu-Weinstein double symplectic groupoids \cite{Boalch:2009tn,Boalch:2011vt,Boalch:2007ty}, as moduli spaces for connections.
 
 In this paper we expand the quasi-Hamiltonian toolkit further. First we introduce a slight generalization of group-valued moment maps, so that the moduli space on a surface with several marked points on every boundary component is equipped with such a moment map.
 
 Next, we subsume the quasi-Hamiltonian toolkit, consisting of reduction, fusion, and fission, into a single broad generalization of reduction. In particular, the moduli space for a triangulated surface is obtained via reduction from the moduli spaces for the triangles.
 
 Consequently, we are able to construct symplectic structures on moduli spaces for: 
 \begin{itemize}
 \item surfaces with boundary, where segments of the boundary are labelled by coisotropic subalgebras of $\g$ (generalizing some results found in \cite{Severa:2011ug}),
 \item surfaces with domains labelled by distinct structure groups and domain walls labelled by coisotropic relations between the structure groups - also called \emph{quilted surfaces} (generalizing some results found in \cite{Boalch:2009tn,Boalch:2007ty,Boalch:2011vt}),
 \item branched surfaces, where the branch locus is labelled by a coisotropic interaction between the branches (generalizing some results found in \cite{Boalch:2009tn,Boalch:2007ty,Boalch:2011vt}).
 \end{itemize}
 Even more generally, our techniques may be used to produce Poisson structures, and a natural generalization of quasi-Hamiltonian and quasi-Poisson structures.
 
 
 As a result, we are able to construct of a number of well known spaces including: Lu's symplectic double groupoid integrating a Poisson Lie group \cite{thesis-3}, Boalch's Fission spaces \cite{Boalch:2009tn,Boalch:2011vt}, Poisson Lie groups \cite{Drinfeld83,Semenov-Tian-Shansky85}, and Poisson homogeneous spaces \cite{Lu06}, among others. Our approach builds upon the results and ideas of various authors including Fock and Rosly, Boalch, and the second author \cite{Fock:1999wz,Boalch:1999wk,Boalch:2001gw,Boalch:2001fw,Boalch:2011vt,Boalch:2007ty,Boalch:2009tn,Severa:2011ug,Severa98,Severa:2005vla}.

Some of these results appeared in \cite{LiBland:2012vo}, where the (quasi-)Poisson structures on moduli spaces are constructed in terms of an intersection pairing. Here we present the reduction theorems in full generality (unifying both the twists and reductions found in (quasi-)Poisson geometry)  and with an emphasis on symplectic structures. We also formulate the results in more natural way, as morphisms of Manin pairs. Among the morphisms of Manin pairs, we introduce the class of \emph{exact morphisms}, corresponding to (quasi-)symplectic structures.

 \subsection{Notation and terminology}\label{sec:Notation}
 At this point, we would like to introduce some notation. 
Suppose $V_i$ is a family of vector spaces (or manifolds) indexed by a set $I$ and $f:J\to I$ is a map. We use the notation 
\begin{align*}
f^!:\prod_{i\in I}V_i&\to \prod_{j\in J}V_{f(j)}\\
\{v_i\}_{i\in I}&\mapsto \{v_{f(j)}\}_{j\in J}
\end{align*} for the induced pull-back map.

 
 
 For any oriented graph $\Gamma$, we let $E_\Gamma$ denote the set of edges, $V_\Gamma$ the set of vertices and $\on{in},\on{out}:E_\Gamma\to V_\Gamma$ the incidence maps.   $\Gamma$ is called a \emph{permutation graph}\footnote{Such graphs are also called \emph{directed cycle graphs} in the literature.} if both $\on{in}$ and $\on{out}$ are bijections.

A \emph{quadratic Lie algebra} is a Lie algebra endowed with an invariant non-degenerate symmetric pairing.

 To simplify our presentation, we will assume that the Lie group $G$ is connected throughout this paper. The generalization to disconnected Lie groups is straightforward. 
 \subsection{The construction}\label{sec:ConstIntro}
\subsubsection{Motivating example: The symplectic form from a triangulation}

Let $\Sigma$ be a closed oriented surface and let
$$\mathcal M_\Sigma(G)=\Hom(\pi_1(\Sigma),G)/\Ad(G)$$
be the moduli space of flat connections. Let us recall how to compute the Atiyah-Bott symplectic form $\omega$ on $\mathcal M_\Sigma(G)$ in terms of a triangulation of $\Sigma$.

Let $\mathcal T$ be a triangulation of $\Sigma$. Let $\mathcal T_0$ denote the set of its vertices, $\mathcal{T}_1$ the set of (unoriented) edges and $\mathcal T_2$ the set of triangles. We let $\tilde{\mc{T}_1}$ denote the set of oriented edges (we thus have a 2-to-1 map $\tilde{\mc T}_1\to\mc T_1$)  we let $$(e\to \bar e):\tilde{\mc{T}_1}\to \tilde{\mc{T}_1}$$ denote the map which reverses the orientation of the edges.

 \begin{figure}
\begin{center}
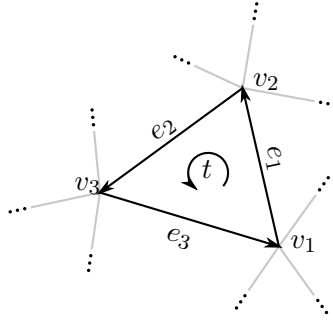
\caption{\label{fig:TriInc} In the figure, $t\in\mc{T}_2$ is a triangle, $e_1,e_2,e_3\in\tilde{\mc{T}_1}$ are oriented edges, and $v_1,v_2,v_3\in\mc{T}_0$ are vertices.
We have $\partial t=\{e_1,e_2,e_3\}$, and $v_2=\on{in}(e_1)$ and $v_1=\on{out}(e_1)$.}
\end{center}
\end{figure}

 Let $\mathcal A_{flat}(\mathcal T)$ be the space of ``combinatorial flat connections'' on $\Sigma$:
$$\mc{A}_{flat}(\mc{T})=\{ g\in G^{\tilde{\mc{T}_1}}\mid g_{\bar e}=g_{e}^{-1}\text{ for all }e\in\tilde{\mc{T}}_1,\text{ and }\prod_{e\in\partial t} g_e=1 \text{ for all }t\in\mc{T}_2\},$$
here $\partial t\subset\tilde{\mc{T}_1}$ denotes the oriented boundary and
 the product is taken in the natural (cyclic) order (cf. \cref{fig:TriInc}). We have an action of $G^{\mc T^0}$ on $\mc{A}_{flat}(\mc{T})$ by ``gauge transformations''
\begin{equation}\label{eq:ResGTransIntro}(g'\cdot g)_e=g'_{\on{in}(e)}g_e(g'_{\on{out}(e)})^{-1},\quad g'\in G^{\mc{T}_0},\quad g\in G^{\tilde{\mc{T}_1}}\end{equation}
and 
$$\mathcal M_\Sigma(G)=\mc{A}_{flat}(\mc{T})/G^{\mc T^0}.$$

If $t$ is an oriented triangle with edges $e_1,e_2,e_3$ (in their cyclic order), let
\begin{equation}\label{eq:MtIntro}
\mc M_t(G)=\{(g_{e_1},g_{e_2},g_{e_3})\in G\times G\times G\mid g_{e_1}g_{e_2}g_{e_3}=1\}.
\end{equation}
We have an inclusion
$$i:\mc{A}_{flat}(\mc{T})\subset \prod_{t\in{\mc T_2}}\mc M_t(G),$$
where the subset $\mc{A}_{flat}(\mc{T})$ is given by the condition $g_e=g_{\bar e}^{-1}$.

Let
$$\omega_t=\frac{1}{2}\langle g_{e_2}^{-1}\d g_{e_2}^{\vphantom{-1}},\d g_{e_1}^{\vphantom{-1}}\,g_{e_1}^{-1}\rangle\in\Omega^2(\mc M_t(G)).$$
The 2-form $\omega_t$ is invariant under cyclic permutations of the edges. 

The symplectic form $\omega$ on $\mc M_\Sigma(G)$ is given by
\begin{equation}\label{eq:omegaFromTriangIntro}
p^*\omega=i^*\sum_{t\in\mc T_2}\omega_t,
\end{equation}
where $p:\mc{A}_{flat}(\mc{T})\to\mathcal M_\Sigma(G)$ is the projection \cite{Weinstein:1995uf}.

We shall interpret Equation \eqref{eq:omegaFromTriangIntro} in the following way: $\mathcal M_\Sigma(G)$ is obtained from $\prod_{t\in{\mc T_2}}\mc M_t(G)$ by a variant of Hamiltonian reduction. The subset $\mc{A}_{flat}(\mc{T})\subset \prod_{t\in{\mc T_2}}\mc M_t(G)$ is given by a moment map condition, and then we need to take the quotient by the residual group $G^{\mc T^0}$ to get a symplectic manifold. To do it, we need to explain this (quasi-)Hamiltonian reduction and the (quasi-)Hamiltonian structure on $\mc M_t(G)$.

\subsubsection{Quasi-Hamiltonian reduction} Let $\mf d$ be a quadratic Lie algebra and $\h\subset\mf d$ a Lagrangian subalgebra (i.e.\ $\h^\perp=\h$). In other words, $(\mf d,\h)$ is a Manin pair.

 Suppose that $\mf d$ acts on a manifold $N$ so that all the stabilizers are coisotropic Lie subalgebras of $\mf d$. We shall recall below the following notions (introduced by Alekseev, Malkin and Meinrenken in \cite{Alekseev97} and by Alekseev, Kosmann-Schwarzbach and Meinrenken in \cite{Alekseev00}, slightly generalized in this paper):
\begin{itemize}
\item A \emph{quasi-Hamiltonian $(\mf d,\h)\times N$-manifold} (or quasi-Hamiltonian $\h$-mani\-fold, if $\mf d$ and $N$ are clear from the context) is a manifold $M$ with an action of $\h$, an $\h$-equivariant map $\mu:M\to N$ (\emph{moment map}), and a bivector field  $\pi$ on $M$, satisfying certain conditions.\footnote{Strictly speaking the bivector field $\pi$ depends in an inessential way on a choice of a vector space complement $\mf{k}\subset\mf{d}$ to $\h$, as in \cite{Alekseev99}. Similarly, in the exact case, the 2-form depends in an inessential way on some other choice. These choices can be made canonically in our cases of interest, and so we will ignore this subtlety until \cref{sec:Q-HamMflds}.}
\item Among the moment maps there are \emph{exact moment maps}. In this case the bivector field $\pi$ can be replaced by a 2-form ($M$ is ``quasi-symplectic'').
\end{itemize}

One of our main results is the following reduction theorem:

\begin{theorem}\label{thm:reductionIntro}
Let $M$ be a quasi-Hamiltonian $(\mf d,\h)\times N$-manifold, $\mf l\subset\mf d$ a Lagrangian Lie subalgebra, and $S\subset N$ an $\mf l$-invariant submanifold.
\begin{enumerate}
\item There is a natural Poisson bracket on the algebra $C^\infty(M)^{\mf l\cap\mf h}\subset C^\infty(M)$ of $\mf l\cap\mf h$-invariant functions. In particular, if $M/(\mf l\cap\mf h)$ is a manifold, it is a Poisson manifold.
\item The ideal $I\subset C^\infty(M)^{\mf l\cap\mf h}$ of functions vanishing on $\mu^{-1}(S)$ is a Poisson ideal. In particular, $\mu^{-1}(S)/(\mf l\cap\mf h)$ is a Poisson manifold, provided it is a manifold.
\item If the moment map $\mu$ is exact and $S$ is an $\mf l$-orbit then the Poisson manifold $\mu^{-1}(S)/(\mf l\cap\mf h)$ is symplectic.
\end{enumerate}
\end{theorem}
More generally, if in place of the Lagrangian subalgebra $\mf l$ we use a coisotropic subalgebra, we have a similar result, where the reduced manifold is still quasi-Hamiltonian. This result is contained in \cref{thm:ExactPartRed,thm:PartRed}, expressed in the more appropriate language of morphisms of Manin pairs.

\subsubsection{The quasi-Hamiltonian structure on moduli spaces}\label{sec:qHamIntro}
Let $e$ be an (abstract) oriented edge, let $N_e=G$ and ${\mf d_e}=\bar\g\oplus\g$, where $\bar \g$ is $\g$ with the inner product negated. The corresponding group $D_e=G\times G$ acts on $N_e=G$ via
\begin{equation}\label{eq:resGaugTransEdgeIntro}
(g_1,g_2)\cdot g=g_1\,g\,g_2^{-1}.
\end{equation}
$N_e=G$ should be imagined as the space of possible holonomies along $e$, and the action of $D_e=G\times G$ as gauge transformations at the endpoints of $e$.

Let $\Sigma$ be a compact oriented surface and $V\subset\partial\Sigma$ a finite subset such that every component of both $\Sigma$ and $\partial\Sigma$ intersects $V$ non-trivially. We shall call $(\Sigma,V)$ a \emph{marked surface}. The boundary circles  of $\Sigma$  are cut into a sequence of oriented edges with endpoints in $V$. Together these edges and vertices form a permutation graph $\Gamma$, the \emph{boundary graph} of $(\Sigma,V)$ (cf. \cref{fig:SurfWBoundInt}).
 Let $\Pi_1(\Sigma,V)$ denote the fundamental groupoid of $\Sigma$ with the base set $V$. Let
$$\mathcal M_{\Sigma,V}(G)=\Hom(\Pi_1(\Sigma,V),G)$$
be the moduli space of flat connections on $G$-bundles over $\Sigma$ trivialized at $V$. This moduli space is quasi-Hamiltonian in the following way:

\begin{figure}
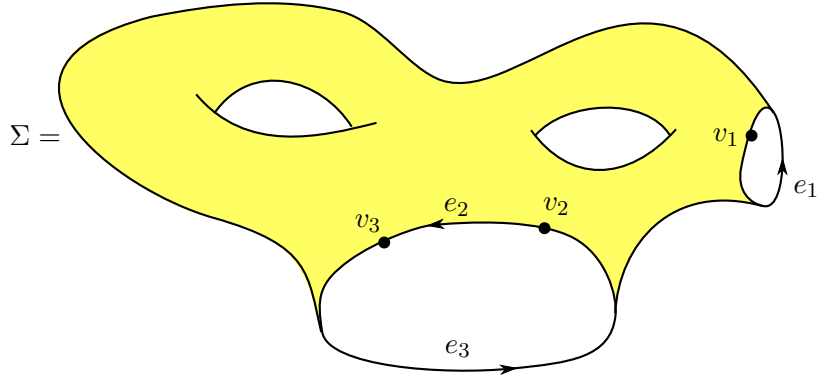
\caption{
\label{fig:SurfWBoundInt} The marked surface $(\Sigma,V)$, with $V=\{v_1,v_2, v_3\}$. The boundary graph $\Gamma$ has edges $E_\Gamma=\{e_1,e_2,e_3\}$ and vertices $V_\Gamma=V$.}
\end{figure}

 We have an  action of the group $H=G^V$ on $\mathcal M_{\Sigma,V}(G)$ by (residual) gauge transformations,
 $$(h\cdot f)(e)=h^{\phantom{-1}}_{\on{in}(e)}f(e)h_{\on{out}(e)}^{-1},$$ for  $h\in G^V,$  $f\in \mc{M}_{\Sigma,V}(G),$ and $ e\in\Pi_1(\Sigma, V).$
  We also have a map
$$\mu:\mathcal M_{\Sigma,V}(G)\to N:=\prod_{e\in E_\Gamma} N_e.$$
where the components of $\mu$ are given by $$\mu(f)_e=f(e)$$ (in other words, $\mu$ is the list of holonomies along the boundary arcs). Notice that the map $\mu$ is $H$-equivariant, where $H=G^V$ embeds as a subgroup
$$G^V\subseteq D:=\prod_{e\in E_\Gamma} D_e=\prod_{e\in E_\Gamma}(G\times G).$$
Here $g\in G^V$ is included as the element $\prod_{e\in E_\Gamma}(g_{\on{in}(e)},g_{\on{out}(e)})$. Letting $\mf{d}$ and $\h$ denote the Lie algebras of $D$ and $H$, we have:
\begin{theorem}\label{thm:qhamIntro}
There is a natural $(\mf d,\mf h)\times N$-quasi-Hamiltonian structure on $\mathcal M_{\Sigma,V}(G)$ with the moment map $\mu$. The moment map is exact and the quasi-symplectic form $\omega$ on $\mathcal M_{\Sigma,V}(G)$ is given by the formula \eqref{eq:omegaFromTriangIntro}, where $\mc T$ is any triangulation of $\Sigma$ such that $\mc T_0\cap\partial\Sigma=V$.
\end{theorem}
We prove this theorem in \cref{sec:SymplStrMMP}.

\begin{remark}
In the case where every boundary component of $\Sigma$ contains exactly one element of $V$, the theorem (except for the triangulation part) was proved by Alekseev, Malkin and Meinrenken in \cite{Alekseev97}, and became the motivation for quasi-Hamiltonian structures.
\end{remark}

\subsubsection{Reduction applied to moduli spaces}
We can combine Theorems \ref{thm:reductionIntro} and \ref{thm:qhamIntro} to produce Poisson and symplectic manifolds: we choose a collection $(\Sigma_i,V_i)$ of marked surfaces with boundary graphs $\Gamma_i$, and a collection $G_i$ of Lie groups with quadratic Lie algebras. The manifold 
$$\mc M:=\prod_i\mc M_{\Sigma_i,V_i}(G_i)$$
is quasi-Hamiltonian, with the moment map $\mu:\mc M\to N=\prod_i N_i$. We choose a Lagrangian Lie subalgebra $\mf l\subset\mf d$ and a $\mf l$-invariant submanifold $S\subset N$. Then by Theorem \ref{thm:reductionIntro}, if the transversality conditions are satisfied, the manifold
$$\mc M_{red}=\mu^{-1}(S)/(\mf l\cap\mf h)$$
is symplectic or Poisson.

The reduced manifold $\mc M_{red}$ can be again seen as a moduli space of flat connections, with certain boundary (or sewing) conditions. Below we shall give various examples for simple choices of $\mf l$ and $S$.

\begin{example}
As the first example, let $\Sigma$ be a closed surface with a triangulation $\mc T$. Let $(\Sigma',V')$ be the disjoint union of the triangles, with $V'$ consisting of the vertices, and let
$$\mc M=\mc M_{\Sigma',V'}(G)=\prod_{t\in\mc T_2}M_t(G).$$
in the notation of Eq.\ \eqref{eq:MtIntro}.

Let us now identify our data on a picture (showing just two triangles, with the parallel edges identified in $\Sigma$):
$$
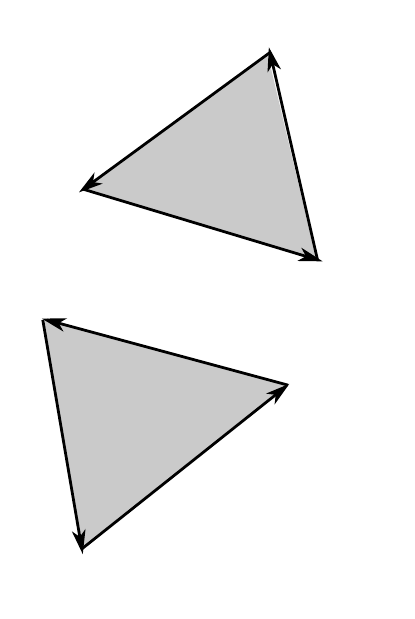
$$
The Lie algebra $\mf d$ is the direct sum of all the $\g$'s and $\bar\g$'s, situated at the half-edges of the triangles. $N$ is the product of all $G$'s. The Lie algebra $\h\subset\mf d$ is the direct sum of all the diagonal Lie subalgebras, $\g_\Delta\subset\g\oplus\bar\g$, situated at the vertices of the triangles. Let the Lie algebra $\mf l\subset\mf d$ be the direct sum of all the diagonals $\g_\Delta\subset\g\oplus\bar\g$ situated at the pairs of half-edges that are identified in $\Sigma$. Notice that $\h\cap\mf l=\g^{\mc T_0}$.
\begin{center}
\begin{tabular}{p{.45\linewidth}p{.45\linewidth}}
\begin{center}
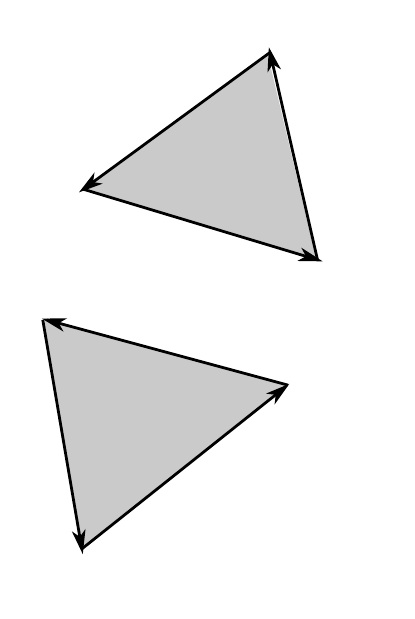
\end{center}

Elements of $\h$ lie in the direct sum of all the diagonals $\g_\Delta\subset\g\oplus\bar\g$ at the vertices of the triangles. 
&
\begin{center}
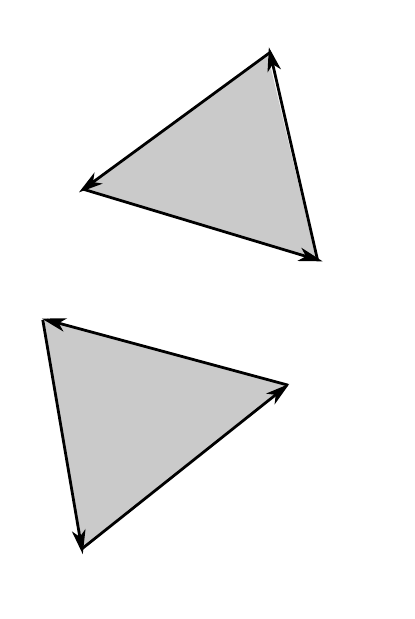
\end{center}

Elements of $\mf{l}$ lie in the direct sum of all the diagonals $\g_\Delta\subset\g\oplus\bar\g$ at the pairs of half-edges that are identified in $\Sigma$. The $\mf{l}$-orbit, $S$, consists of those elements $g\in \prod_{e\in \tilde{\mc{T}_1}}$ such that $g_e=g_{\bar e}^{-1}$.
\end{tabular}
\end{center}

For the $\mf l$-orbit  $S\subset N$ we take the subset given by the conditions $g_{\bar e}=g^{-1}_e$ for any pair of edges $e,\bar e$ that are identified in $\Sigma$. We have
$$\mc M_\Sigma(G)=\mc M_{red}:=\mu^{-1}(S)/(\h\cap\mf l).$$
Thus we are able to obtain $\mc M_\Sigma(G)$ by quasi-Hamiltonian reduction from triangles.
\end{example}

So far we have not explicitly described the symplectic or Poisson structure on $\mc M_{red}$. In a special case it is very simple. Let $$\mu_i:\mc M_{\Sigma_i, V_i}(G_i)\to N_i$$
denote the (exact) moment map, and let $\omega_i$ be the quasi-symplectic 2-form on $M_{\Sigma_i, V_i}(G_i)$ (given explicitly in Theorem \ref{thm:qhamIntro}).
For every boundary arc $e$ of $\Sigma_i$ we have the involution of $\mf d_e=\bar\g_i\oplus\g_i$ given by 
$$(\xi,\eta)\mapsto(\eta,\xi).$$
If we apply the involution simultaneously at all the boundary arcs, we get an involution of
$$\mf d=\bigoplus_i\bigoplus_{e\subset\partial\Sigma_i}\bar\g_i\oplus\g_i.$$
We shall say that a subalgebra $\mf l\subset\mf d$ is \emph{symmetric} if it is invariant with respect to this involution.

\begin{theorem}\label{thm:MainThmIntro}
If $\mf l\subset\mf d$ is a symmetric Lagrangian subalgebra and $S\subset N$ is the $\mf l$-orbit through the identity element $$1\in N=\prod_i G_i^{E_{\Gamma_i}},$$ then the symplectic form $\omega_{red}$ on
$$\mc M_{red}=\mu^{-1}(S)/{\mf l\cap\mf h}$$
is given by
$$p^*\omega_{red}=\left.\sum_{i}\omega_{i}\right\rvert_{\mu^{-1}(S)}$$
where $p:\mu^{-1}(S)\to\mc M_{red}$ is the projection.
\end{theorem}

As explained in \cref{rem:SymmetricOrbits}, \cref{thm:MainThmIntro} will follow as a corollary to \cref{thm:PartRedSplEx}.
\subsection{Colouring Edges}\label{sec:ColBndSurf}

Suppose that $\mf{c}\subseteq\g$ is a coisotropic subalgebra (i.e. $\mf{c}^\perp\subseteq\mf{c}$). Then the subalgebra
$$\mf{l}_{\mf{c}}:=\{(\xi,\eta)\in(\ol{\g}\oplus\g)\mid \xi,\eta\in\mf{c}\text{ and }\xi-\eta\in\mf{c}^\perp\}$$
is both Lagrangian and symmetric. The orbit of $\mf{l}_{\mf{c}}$ through the identity of $G$, with respect to the action \cref{eq:resGaugTransEdgeIntro}, can be identified with the simply connected Lie group $C^\perp$ integrating the Lie algebra $\mf{c}^\perp$.

Let $(\Sigma, V)$ be a marked surface. For every boundary arc  $e$ (i.e.\ for every edge of the permutation graph $\Gamma_{\Sigma,V}$ with the vertex set $V$), let $\mf{c}_e\in\g$ be a coisotropic subalgebra, and consider the Lie subalgebra
$$\mf{l}:=\bigoplus_e \mf l_{\mf c_e}\subset\bigoplus_e\bar\g\oplus\g=\mf d.$$
It is clear that $\mf{l}$ is both Lagrangian and symmetric. Let $S\subset N=\prod_e G$ be the $\mf l$-orbit passing through $1\in\prod_e G$. \cref{thm:MainThmIntro} implies that if the quotient space
$$\mc M_{red}=\mu^{-1}(S)/\mf{l}\cap\g^V$$
is a manifold, it is symplectic. 

Concretely,
\begin{equation}\label{eq:ColBndLevSet}
\mc M_{red}=\{f:\Pi_1(\Sigma,V)\to G \mid f(e)\in C^\perp_e\text{ for every }e\}/\mf{l}\cap\g^V,
\end{equation}
and 
$\mf{l}\cap\g^V\subset\g^V$
is given by the conditions
\begin{subequations}\label{eq:ColBndResGT}
\begin{gather}
\xi_v\in\mf c_{e_1}\cap\mf c_{e_2}\text{ where }v=\on{in}(e_1)=\on{out}(e_2)\\
\xi_{\on{in}(e)}-\xi_{\on{out}(e)}\in\mf c^\perp_e.
\end{gather}
\end{subequations}
Notice that if $\mf c_e$'s are Lagrangian then the first condition implies the second one. If, moreover, $\mf c_{e_1}\cap\mf c_{e_2}=0$ for any pair of consecutive boundary arcs then $\mf{l}\cap\g^V=0$. Under these conditions the moduli space $\mc M_{red}$ was considered in \cite{Severa:2011ug}.

\begin{figure}
\begin{center}
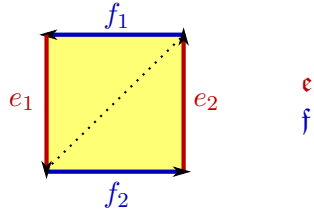
\caption{\label{fig:DoublSympGroupoid} The symplectic double groupoid integrating the Lie-Poisson structures on $E$ and $F$.}
\end{center}
\end{figure}
\begin{example}[\!\!\cite{Severa:2005vla,Severa:2011ug,Boalch:2011vt,Boalch:2009tn,Boalch:2007ty}]\label{ex:LuWeinSympDbl}
Suppose that $\mf{e},\mf{f}\subseteq\g$ are transverse Lagrangian subalgebras, and let $E,F\subset G$ denote the corresponding  connected Lie groups. We may colour alternate edges of a rectangle with $\mf{e}$ and $\mf{f}$, as in \cref{fig:DoublSympGroupoid}. From \cref{eq:ColBndLevSet} we see that
$$\mc M_{red}=\{(e_1,e_2,f_1,f_2)\in E^2\times F^2\mid e_1f_1e_2f_2=1\}.$$
By \cref{thm:MainThmIntro} the moduli space $\mc M_{red}$
carries the symplectic form
$$\omega=\frac{1}{2}\la e_1^{-1}\d e_1,\d f_1 \:f_1^{-1}\ra+\frac{1}{2}\la e_2^{-1}\d e_2,\d f_2\: f_2^{-1}\ra.$$
Here, the upper-left triangle in \cref{fig:DoublSympGroupoid} contributed the term $\frac{1}{2}\la e_1^{-1}\d e_1,\d f_1 f_1^{-1}\ra$ to this expression while the bottom-right triangle in \cref{fig:DoublSympGroupoid} contributed the term $\frac{1}{2}\la e_2^{-1}\d e_2,\d f_2 f_2^{-1}\ra$.

As explained in \cite{Severa:2005vla,Severa:2011ug}, the symplectic manifold $(\mc{M},\omega)$ is the Lu-Weinstein symplectic double groupoid integrating the Lie-Poisson structures on $E$ and $F$ \cite{Lu:1989vb}. 


\end{example}

\begin{figure}
\begin{center}
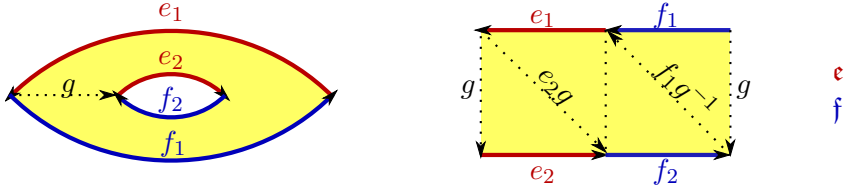
\caption{\label{fig:DoublSympGroupoidDbl} The symplectic double groupoid integrating the Lie-Poisson structure on $G$.}
\end{center}
\end{figure}
\begin{example}[\!\!\cite{Severa:2011ug,Severa98}]
Let  $\mf{e},\mf{f}\subseteq\g$ be as above. Divide each boundary component of the annulus into two segments and colour alternate edges with $\mf{e}$ and $\mf{f}$, as in \cref{fig:DoublSympGroupoidDbl}. From \cref{eq:ColBndLevSet} we see that
$$\mc M_{red}=\{(e_1,e_2,f_1,f_2,g)\in E^2\times F^2\times G\mid ge_1f_1g^{-1}f_2e_2=1\}.$$
The moduli space $\mc M_{red}$
carries the symplectic form
\begin{multline*}\omega=\frac{1}{2}\la e_2^{-1}\d e_2,\d g\: g^{-1}\ra+\frac{1}{2}\la (e_2g)^{-1}\d(e_2g),\d e_1\:e_1^{-1}\ra\\
+\frac{1}{2}\la gf_1^{-1}\d (f_1g^{-1}),gf_2\:f_2^{-1}\ra-\frac{1}{2}\la \d g\:g^{-1},\d f_1 \:f_1^{-1}\ra,\end{multline*}
which can be computed from the triangulation pictured in \cref{fig:DoublSympGroupoidDbl}.

As explained in \cite{Severa:2011ug,Severa98}, the symplectic manifold $(\mc{M},\omega)$ is the symplectic double groupoid integrating the Lie-Poisson structure on $G$.
\end{example}

\begin{figure}
\begin{center}
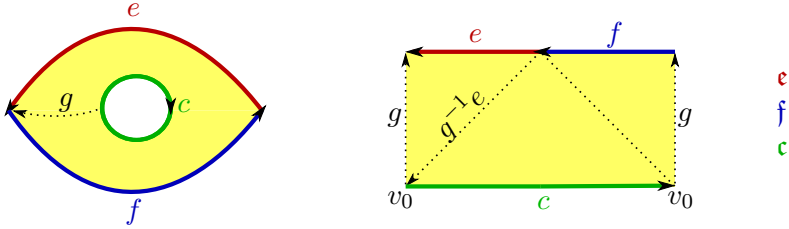
\caption{\label{fig:LuYakimov} The symplectic groupoid integrating the Lu-Yakimov Poisson structure on $G/C$.}
\end{center}
\end{figure}
\begin{example}
Suppose that $\mf{e},\mf{f}\subseteq\g$ are transverse Lagrangian subalgebras and $\mf{c}\subseteq\g$ is a coisotropic subalgebra. Let $E,F,C,C^\perp\subset G$ denote the corresponding  connected Lie subgroups, and suppose that  $C\subset G$  is closed. Consider the annulus whose outer boundary is divided into two segments. Colour the outer boundary by the two Lie subalgebras $\mf{e}$ and $\mf{f}$ and the inner boundary by the inner boundary by the Lie subalgebra $\mf{c}$, as in \cref{fig:LuYakimov}. We have
$$\mu^{-1}(S)=\{(e,f,g,c)\in E\times F\times G\times C^\perp\mid efgcg^{-1}=1\}.$$
Meanwhile, \cref{eq:ColBndResGT} yields 
$$\mf{l}\cap\g^V=\{\xi\in\g^V\mid \xi_{v_0}\in\mf{c}\text{ and }\xi_v=0 \text{ for all }v\neq v_0\},$$
where $v_0$ is the vertex labelled in \cref{fig:LuYakimov}. Thus the Lie group of residual gauge transformations is $C$, acting as 
$$c'\cdot(e,f,g,c)=(e,f,gc'^{-1},c'cc'^{-1}),\quad c'\in C,\quad (e,f,g,c)\in E\times F\times G\times C^\perp.$$
Since, by assumption this acts freely and properly on $\mu^{-1}(N)$, \cref{thm:MainThmIntro} implies that the moduli space
$$\mc{M}_{red}=\{(e,f,g,c)\in E\times F\times G\times C^\perp\mid efgcg^{-1}=1\}/C$$
carries the symplectic form
$$\omega=-\frac{1}{2}\la\d g\:g^{-1},\d e\:e^{-1}\ra+\frac{1}{2}\la c^{-1}\d c,\d(g^{-1} e)\:e^{-1}g\ra+\frac{1}{2}\la f^{-1}\d f,\d g\:g^{-1}\ra.$$

The symplectic manifold $(\mc{M}_{red},\omega)$ is the symplectic groupoid integrating the Lu-Yakimov Poisson structure on the homogeneous space $G/C$ \cite{Lu06}. The source and target maps are 
$$\on{s}(e,f,g,c)=g,\quad \on{t}(e,f,g,c)=fg,$$
and the multiplication is
$$(e',f',g',c')\cdot(e,f,g,c)=(ee',f'f,g,c'c),\quad g'=fg.$$
\end{example}


 
\subsection{Domain walls and branched surfaces}\label{sec:DomWallBrnch}
Let $(\Sigma_i, V_i)$ be a finite collection of marked surfaces with boundary graphs $\Gamma_i$, and $G_i$ a collection of Lie groups with quadratic Lie algebras $\g_i$. As we observed above, the space
$$\mc M=\prod _i\mc M_{\Sigma_i,V_i}(G_i)$$
is a $(\mf d,\h)\times N$-quasi-Hamiltonian for appropriate $(\mf d,\h,N)$, and if we choose a Lagrangian Lie subalgebra $\mf l\subset\mf d$ and a $\mf l$-orbit $S\subset N$, then
$$\mc M_{red}=\mu^{-1}(S)/\mf l\cap\mf h$$
is symplectic. If the subalgebra $\mf l\subset\mf d$ is symmetric then \cref{thm:MainThmIntro} gives us a simple formula for the symplectic form on $\mc M_{red}$.

Let us now choose a symmetric $\mf l\subset\mf d$ in the following way. We first glue the boundary arcs of $(\Sigma_i, V_i)$ in an arbitrary way. More precisely, let $\mathsf W$ be a finite collection of (disjoint) unit intervals called \emph{domain walls}, let
$$\kappa: \sqcup_i E_{\Gamma_i}\to \mathsf W$$
be a surjective map assigning to every edge of every boundary graph $\Gamma_i$ a domain wall, and let 
$$\phi_e:e\to\kappa(e)$$
be a homeomorphism for every boundary edge $e$ (not required to preserve the orientation). Let $\Sigma$ be the topological space obtained from $\Sigma_i$'s and the domain walls after we identify every boundary arc $e$ with $\kappa(e)$ via the map $\phi_e$.

For every boundary arc $e\in E_{\Gamma_i}$ let $i( e)=i$, and
$$
\on{sign}(e)=
\begin{cases}
+1 &\text{if $\phi_e$ is orientation-preserving}\\
-1 &\text{otherwise.}
\end{cases}
$$
For every domain wall $w\in\mathsf W$, let
$$\g_w=\bigoplus_{\substack{e\in \kappa^{-1}(w)\\ \on{sign}(e)=+1}}\g_{i(e)}\ \oplus\bigoplus_{\substack{e\in \kappa^{-1}(w)\\ \on{sign}( e)=-1}}\bar\g_{i(e)}$$
and
$${\mf d_w}=\bar\g_w\oplus\g_w.$$
Notice that
$$\mf d:=\bigoplus_{w\in \mathsf W}\bar\g_w\oplus\g_w=\bigoplus_i(\bar\g_i\oplus\g_i)^{E_{\Gamma_i}}.$$

For every domain wall $w\in \mathsf W$ we now choose a coisotropic Lie subalgebra
$$\mf c_w\subset\g_w.$$
Using $\mf c_w$ we construct the symmetric Lagrangian Lie subalgebra
$\mf l_w\subset\mf d_w,$
$$\mf l_w:=\{(\xi,\eta)\in\bar\g_w\oplus\g_w \mid \xi,\eta\in\mf c_w,\,\xi-\eta\in\mf c_w^\perp\}.$$
Finally we set
$$\mf l=\bigoplus_{w\in\mathsf W}\mf l_w.$$

For every domain wall $w\in \mathsf W$, let $C^\perp_w\subset G_w$ denote the  connected Lie subgroup with Lie algebra $\mf{c}^\perp_w$,
and $$C^\perp=\prod_{w\in \mathsf W}C^\perp_w\subseteq \prod_i G_i^{E_{\Gamma_i}}.$$
 Then $S:=C^\perp\subset N$ is the $\mf l$-orbit passing through $1\in \prod_i G_i^{E_{\Gamma_i}}=N$. 
As before, we have
\begin{equation}\label{eq:BrnchLevSet}
\mu^{-1}(S)=\bigl\{\{f_i:\Pi_1(\Sigma_i,V_i)\to G_i\}_i\mid \prod_i\{f_i(e)^{\on{sign}(e)}\}_{e\in E_{\Gamma_i}}\in C^\perp\bigr
\}
\end{equation}
and
$$\mc M_{red}=\mu^{-1}(S)/\mf{l}\cap\bigoplus_i\g_i^{V_i}.$$

\begin{example}[Oriented surfaces with coloured boundaries]
If we have just one domain and the gluing map $\kappa$ is injective,   then we are in the case described in \cref{sec:ColBndSurf}. 
\end{example}

\subsubsection{Domain walls}
\begin{figure}
\begin{center}
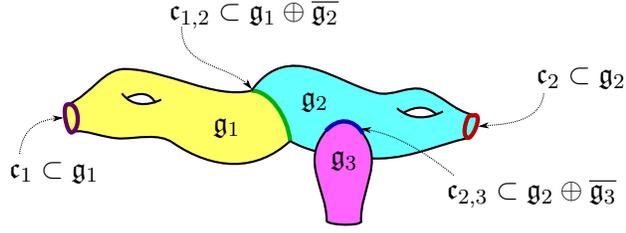
\caption{\label{fig:DomColSurf} Our surface is divided into domains with distinct structure groups, and the domain walls are coloured by coisotropic relations between the structure groups. As before, coisotropic boundary conditions are also chosen.}
\end{center}
\end{figure}

Suppose that the glued topological space $\Sigma$ is still a (not necessarily oriented) surface. Equivalently,  every domain wall $w\in\mathsf W$ borders  either one or two domains (i.e. the preimage $\kappa^{-1}(w)$ has cardinality one or two). The resulting surface $\Sigma$ was called a \emph{quilted surface} in \cite{LiBland:2012vo} (following \cite{Wehrheim:2010fa}).
  \begin{remark}
  Quantizations of these moduli spaces have been studied in the physics community \cite{Kapustin:2010we,Kapustin:2011cn,Fuchs:2012tq} for abelian structure groups and Lagrangian relations on the domain walls.
  \end{remark}

\begin{figure}
\begin{center}
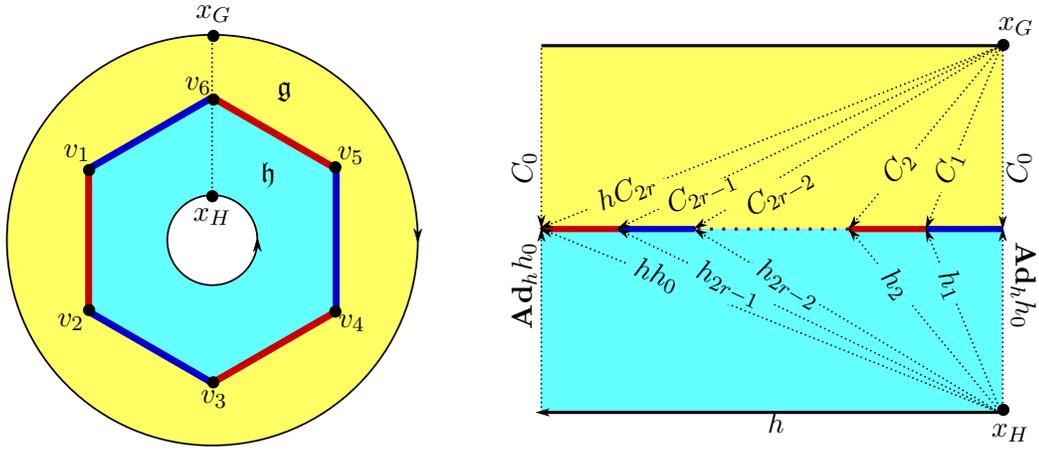
\caption{\label{fig:PBoalchDrawCut} On the surface pictured above, the structure group in the yellow domain is $G$ while the structure group in the blue domain is $H$. Along the boundary of the two domains, blue edges are coloured with $\mf{c}_+$ while the red edges are coloured with $\mf{c}_-$. Cutting along the dotted line in the first picture yields yields the second picture. Acting by $H$ at the vertices $v_1,\dots, v_{2r}$ allows us to set the holonomies $h_0,\dots,h_{2r-1}$ to the identity.}
\end{center}
\end{figure}
\begin{example}[\!\!\cite{Boalch:2011vt}]\label{ex:FissionSpace}
Suppose that $\g=\mf{u}_+\oplus\h\oplus\mf{u}_-$ as a vector space (but not as a Lie algebra), where $\mf{p}_\pm:=\h\oplus\mf{u}_\pm\subseteq \g$ are coisotropic subalgebras satisfying $\mf{p}_\pm^\perp=\mf{u}_\pm$. Suppose further that the Lie subalgebras $\mf{u}_\pm,\mf{p}_\pm,\h$ all integrate to closed subgroups $U_\pm,P_\pm,H\subseteq G$ such that $H=P_+\cap P_-$. The metric on $\g$ descends to a non-degenerate invariant metric on $\h\subseteq\g$, and
\begin{equation}\label{eq:BoalchCois}\mf{c}_\pm:=\{(\xi;\xi+\mu)\in\h\oplus\g\mid \xi\in\h\text{ and }\mu\in \mf{u}_\pm\}\end{equation}
is a coisotropic subalgebra (in fact, it is Lagrangian).

As in \cref{fig:PBoalchDrawCut}, let $\Sigma$ denote the annulus, and let $\gamma\subset(\Sigma\setminus\partial\Sigma)$ be a simple closed curve representing the  generator of the fundamental group. Cutting $\Sigma$ along $\gamma$ yields two annuli, $\Sigma_G,\Sigma_H\subset\Sigma$, which we label with the structure groups $G$ and $H$, respectively. We divide $\gamma$ into $2r$ segments with endpoints labelled $v_1,\dots v_{2r}$, and colour alternating segments with the coisotropic Lie subalgebras $\mf{c}_+$ and $\mf{c}_-$. Finally, we mark the respective components of $\partial \Sigma$ with points $x_G$ and $x_H$. 

 The points $x_G,X_H,v_1,\dots, v_{2r}$ form the vertices of a triangulation of $\Sigma$, as pictured in $\cref{fig:PBoalchDrawCut}$. Now the orbit of $\mf{l}_\pm$ through the identity is $P_\pm$.
 Thus, from \cref{eq:BrnchLevSet}, we see that
 \begin{multline*}\mu^{-1}S=\{(h,h_0,\dots, h_{2r-1};C_0,C_1,\dots, C_{2r})\in H^{2r+1}\times G^{2r+1}\\\mid h_{2i+1}^{-1}C_{2i+1} C_{2i}^{-1}h_{2i}\in U_+\text{ and }h_{2i}^{-1}C_{2i} C_{2i-1}^{-1}h_{2i-1}\in U_-,\},\end{multline*}
 where the elements $h,h_0,\dots, h_{2r-1}\in H$ and $C_0,C_1,\dots, C_{2r}\in G$ denote the appropriately labelled holonomies in \cref{fig:PBoalchDrawCut}.

 On the other hand,
$$\mf{l}\cap\bigoplus_i\g_i^{V_i}\cong \prod_{v_1,\dots, v_{2r}}\h,$$
acting at the appropriate vertices. Thus, up to a gauge transformation, we may assume that $h_0=h_1=\dots=h_{2r-1}=1$. Setting $S_i=C_{i} C_{i-1}^{-1}$, we see that
 that the quotient space, 
$\on{hol}^{-1}(\mf{l}\cdot1)/\big(\mf{l}\cap\prod_{t\in\mc{T}_2}(\g_t)_{\Gamma_{P_3}}\big),$ 
can be identified with 
$$_G\mc{A}_H^r:=\{(h;S_{2r},\dots,S_1;C_0)\in H\times (U_-\times U_+)^r\times G\}.$$

We compute the two form to be
\begin{multline*}\omega=-\frac{1}{2}\big(\la\d(hC_{2r})\:(hC_{2r})^{-1},\d C_0\:C^{-1}_0\ra+\la (h C_{2r})^{-1}\d(h C_{2r}),C_{2r-1}^{-1}\d C_{2r-1}\ra\\+\sum_{i=1}^{2r-1}\la C_i^{-1}\d C_i,C_{i-1}^{-1}\d C_{2i-1}\ra\big).\end{multline*}
Substituting $bC_0=hC_{2r}$ in the first term and simplifying yields
\begin{multline*}\omega=\frac{1}{2}\big(\la \d C_0\:C_0^{-1},\Ad_b\d C_0\:C_0^{-1}\ra+\la \d C_0\:C_0^{-1},\d b\:b^{-1}\ra+\la\d C_{2r}\:C_{2r}^{-1},h^{-1}\d h\ra\\ +\sum_{i=1}^{2r}\la C_i^{-1}\d C_i,C_{i-1}^{-1}\d C_{2i-1}\ra\big),
\end{multline*}
(here we have used the fact that $\la \d S_{2r}\: S_{2r}^{-1},h^{-1}\d h\ra=0$).
Now, we haven't coloured the boundary of $\partial \Sigma$, so \cref{thm:MainThmIntro} does not imply that $\omega$ is symplectic. Nevertheless, as we shall see later, \cref{thm:ExactPartRed} implies that $\omega$ defines a quasi-Hamiltonian $G\times H$ structure on $_G\mc{A}_H^r$, where the moment map $_G\mc{A}_H^r\to G\times H$ is given by the holonomy along the (oriented) boundary components:
$$(h;S_{2r},\dots,S_1;C_0)\to (C_0^{-1}hS_{2r}\cdots S_1C_0,h^{-1}),$$
and the $G$ and  $H$ actions on $_G\mc{A}_H^r$ are precisely the residual gauge transformations. These act by $G$ at $x_G$ and by $H$ at $x_H$:
$$(g,k)\cdot(h;S_{2r},\dots,S_1;C_0)=(khk^{-1},kS_{2r}k^{-1},\dots, kS_{1}k^{-1},kC_0g^{-1}),\quad g\in G,\quad k\in H.$$

  \begin{remark}
This quasi-Hamiltonian $G\times H$-space was first discovered by Boalch \cite{Boalch:2011vt,Boalch:2007ty,Boalch:2009tn}, who used it to study  meromorphic connections on Riemann surfaces. 
  \end{remark}

%
%

\end{example}

%
%
%

\subsubsection{Branched Surfaces}
\begin{figure}
\begin{center}
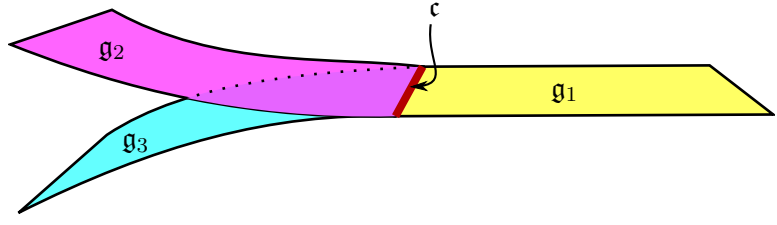
\caption{\label{fig:ColouredBranched} Pictured above are three domains with structure Lie algebras $\g_1$, $\g_2$ and $\g_3$. The three domains intersect at a branch locus, which we must colour by a coisotropic subalgebra $\mf{c}\subseteq \oplus_{i=1}^3\g_i$.}
\end{center}
\end{figure}
We can now consider examples where $\Sigma$ is not a topological surface, i.e.\ where the domain walls may border more than two domains.
\begin{remark}
Since our gauge fields (connections on $\Sigma$) are constrained to lie in $\mf{c}_w\subseteq\bigoplus_{e\in \kappa^{-1}(w)}\g_e$ along the domain wall $w\in W$, 
one may interpret $\mf{c}_w$ as a ``conservation law'' for an interaction between the structure groups of the various domains glued to the domain wall $w$. 
\end{remark}

\begin{figure}
\begin{center}
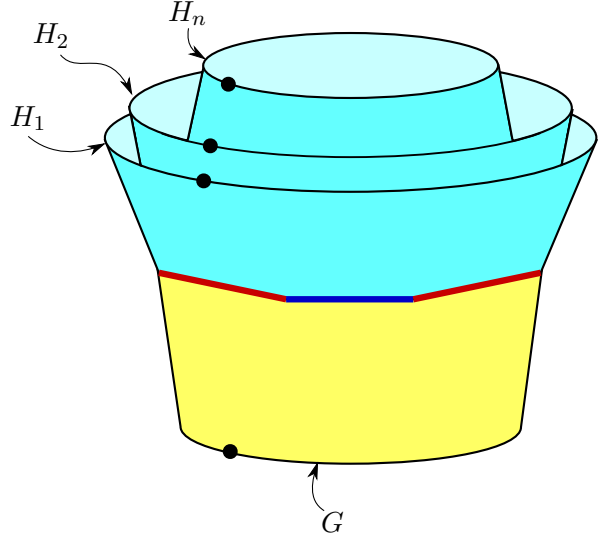
\caption{\label{fig:BoalchFission} A branched surface which arises in the work of Philip Boalch (see \cite[Pg. 2675]{Boalch:2009tn} and \cite[Fig. 2]{Boalch:2011vt}).}
\end{center}
\end{figure}
\begin{example}[\!\!\cite{Boalch:2011vt,Boalch:2007ty,Boalch:2009tn}]
Let $V=\oplus_{i=1}^n V_i$ be a direct sum decomposition of a finite dimensional vector space, $G=\on{Gl}(V)$,
 and $P_+\subseteq G$ the stabilizer for the flag 
 $$F_1\subset F_2\subset\cdots \subset F_n=V,$$
  where $F_k=\oplus_{i=1}^k V_i$. Similarly, let $P_-\subseteq G$ be the stabilizer for the flag
 $$\tilde F_n\subset\cdots \subset \tilde F_2 \subset \tilde F_1=V,$$
  where $\tilde F_k=\oplus_{i=k}^n V_i$. Finally, let $H_i=\on{Gl}(V_i)$ so that $\prod_{i=1}^nH_i=P_+\cap P_-$. 
  Let $H=\prod_{i=1}^nH_i$, let $U_\pm$ denote the unipotent radicals of $P_\pm$, and let $\g,\mf{p}_\pm,\mf{u}_\pm,\h,\h_i$ denote the Lie algebras corresponding to the various Lie groups. 
  
  Now consider the moduli space
  $$_G\mc{A}_H^r:=(\prod_{i=1}^nH_i)\times (U_-\times U_+)^r\times G$$
  described in \cref{ex:FissionSpace}. The coisotropic Lie algebra
  $$\mf{c}_\pm:=\{(\sum_{i=1}^n\xi_i;\sum_{i=1}^n\xi_i+\mu)\mid \xi_i\in\h_i\text{ and }\mu\in \mf{u}_\pm\}\subset(\oplus_{i=1}^n\h_i)\oplus\g$$
  defined in \cref{eq:BoalchCois} can be used to colour the branch locus of $n+1$ domains with the structure groups $H_1,\dots,H_n$ and $G$.
  Thus we may interpret $_G\mc{A}_H^r$ as the moduli space of flat connections for the branched surface pictured in \cref{fig:BoalchFission}. 
  \begin{remark}
The quasi-Hamiltonian space $_G\mc{A}_H^r$ first appeared in the work of Philip Boalch \cite{Boalch:2011vt,Boalch:2007ty,Boalch:2009tn}. 
  \end{remark}

\end{example}

\begin{figure}
\begin{center}
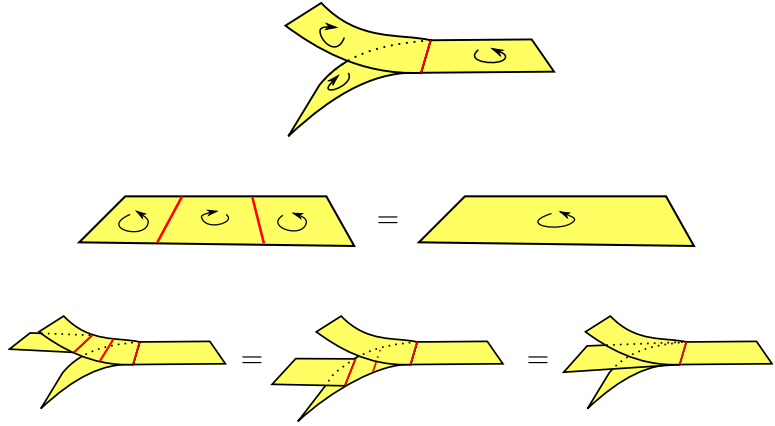
\caption{\label{fig:Branched} Each of the edges above are coloured by $\mf{c}_n$, where $n$ is the number of domains branching off the given edge. As we cross a branch locus, the orientation of the domains reverses.
 As depicted, we may move branch loci past each other. }
\end{center}
\end{figure}
\begin{example}[quasi-triangular structures]\label{rem:quasiTriStr}
Let $\g$ be a quasi-triangular Lie quasi-bialgebra, i.e.\ $\g$ is a Lie algebra with a chosen element $s\in(S^2\g)^\g$. Let $\mf d$ be the Drinfel'd double of $\g$. This means that $\mf d$ is a quadratic Lie algebra, $\g\subset\mf d $ is a Lagrangian subalgebra, and $\mf p\subset\mf d$ is an ideal such that $\mf d=\g\oplus\mf p$ as a vector space. Additionally, the restriction of the quadratic form on $\mf d$ to $\mf p \cong \g^*$ is $s$.\footnote{These properties uniquely define $\mf{d}$. In particular, the Lie bracket is given by $$[\xi+\alpha,\eta+\beta]=[\xi,\eta]+\ad_\xi^*\alpha-\ad_{\eta}^*\beta,\quad\xi,\eta\in\g,\quad\alpha\in\mf p,\beta\in\mf{p}^\perp$$ where $\ad^*$ denotes the contragredient representation of $\g$.}

 There is a natural groupoid structure on $\mf d$, where $\g$ is the space of objects and composition is defined by
$$(\xi+\alpha)(\xi+\beta)=\xi+\alpha+\beta\quad \forall \xi\in\g,\alpha\in\mf p,\beta\in\mf p^\perp,$$
and the source and target maps are $$\on{s}(\xi+\alpha)=\xi,\quad\on{t}(\xi+\alpha)=\xi+s(\alpha,\cdot),\quad \xi\in \g,\alpha\in \mf{p}$$

The graph of multiplication,
\begin{equation}\label{eq:CAgrpdOvPtMult}\gr(\on{Mult})=\{(\xi \eta,\xi,\eta)\mid \xi,\eta\in\mf{d}\text{ are composable}\}\subseteq\mf{d}\oplus{\bar{\mf{d}}\oplus\bar{\mf{d}}},\end{equation}
is a Lagrangian Lie subalgebra. 
See \cite{Drinfeld:1989tu} and \cite{LiBland:2011vqa} for more details.

Similarly, the graph of iterated multiplication
$$\mf{c}_n:=\{(\xi_1,\dots,\xi_n)\mid\xi_1\xi_2\cdots\xi_n\in\g\}\subseteq\mf{d}^n$$
is also a Lagrangian Lie subalgebra. As such it can be used to colour the branch locus of $n$ domains each with structure group $D$ (a connected Lie group with the Lie algebra $\mf d$). Note that crossing such a branch locus reverses the orientation of the domain.

The associativity of multiplication on $\mf{d}$ plays out as follows: paying attention to the orientations, if (as in \cref{fig:Branched}) we
\begin{itemize}
\item move two branch loci past each other, or 
\item break a $\mf{c}_{m+n-2}$-coloured branch locus into two separate $\mf{c}_m$ and $\mf{c}_n$ coloured branch loci (or vice-versa),
\end{itemize} 
the resulting moduli spaces are canonically symplectomorphic. 

\begin{figure}
\begin{center}
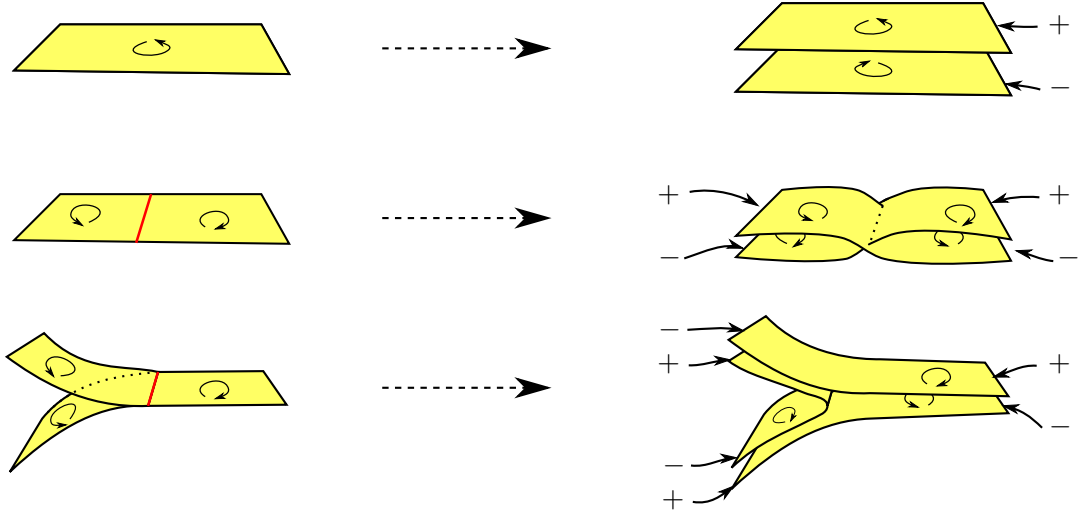
\caption{\label{fig:TwoSheeting} We replace each domain in our branched surface by two copies, one with the same orientation (labelled $+$), and one with the opposite orientation (labelled $-$). Each $\mf{c}_n$-coloured domain wall is replaced by cyclically gluing the incident sheets together, respecting the orientations. In this way we obtain an oriented surface from our $\mf{c}_n$-coloured (branched) surface.}
\end{center}
\end{figure}

In fact, there is a clear interpretation of the 
the moduli spaces constructed by sewing domains together using the Lagrangian relations $\mf{c}_n$. One may identify them with certain traditional moduli spaces in the following way: Suppose that $\cup \Sigma_d\to \Sigma$ is our $\mf{c}_n$-coloured surface with domains $\Sigma_d$. First we form a two sheeted (branched cover) $\tilde\Sigma$ of $\Sigma$ as follows: double each domain $\Sigma_d$ to two sheets $\Sigma_d^+\cup\Sigma_d^-$, where the sheet $\Sigma_d^+$ is canonically identified with $\Sigma_d$, while the sheet $\Sigma_d^-$ is also canonically identified with $\Sigma_d$ but with the opposite orientation. At each $\mf{c}_n$-coloured domain wall with incident domains $\Sigma_{d_1},\dots, \Sigma_{d_n}$, cyclically glue the sheets 
$\Sigma_{d_1}^\pm,\dots, \Sigma_{d_n}^\pm$
together along their corresponding boundary segment, respecting the orientations, as in \cref{fig:TwoSheeting}. In this way, one constructs the oriented surface $\tilde\Sigma$. 

The groupoid inversion $\on{Inv}:\mf{d}\dasharrow\bar{\mf{d}}$, being a morphism of Lie algebras, integrates to an involution of the Lie group $D$. This involution in turn lifts to an involution $$\mc{A}_{\tilde\Sigma}(D)\to \mc{A}_{\tilde\Sigma}(D)$$ of the connections on $\tilde\Sigma$, mapping the fibre of the principal bundle over the $+$-sheet to the fibre over the $-$-sheet via $\on{Inv}:D\to D$. The involution is a symplectomorphism which is compatible with the gauge transformations, and thus descends to a symplectomorphic involution on the moduli space, $\mc{M}_{\tilde\Sigma}(D)$. The fixed points of this involution are naturally identified with the moduli space $\mc{M}_{\Sigma}(D)$ of flat connections on the original $\mf{c}_n$-coloured surface.
\end{example}

\begin{example}[\!\!\cite{Fock:1999wz,Boalch:2007ty,Boalch:2009tn,Boalch:2011vt}]\label{ex:DblSymplBoalch}
Suppose that $\g$ is a quasi-triangular Lie-bialgebra, where the $s\in(S^2\g)^\g$ is non-degenerate, i.e. $\g$ is a quadratic Lie algebra. Equivalently, the double is $\mf{d}=\g\oplus\bar\g$, and the Manin triple is $(\mf{d};\g_\Delta,\mf{h})$ where $\g_\Delta\subset \g\oplus\bar\g=\mf{d}$ is the diagonal, and $\mf{h}\subset\g\oplus\bar\g$ is a complementary Lagrangian subalgebra. Notice that we may view $\mf{h}$ as either a Lagrangian subalgebra of $\mf{d}$, or a Lagrangian relation from $\g$ to itself. We let $G$ and $H\subset G\times G$ denote the simply connected (resp. connected) Lie groups corresponding to $\g$ and $\h$. 

\begin{figure}
\begin{center}
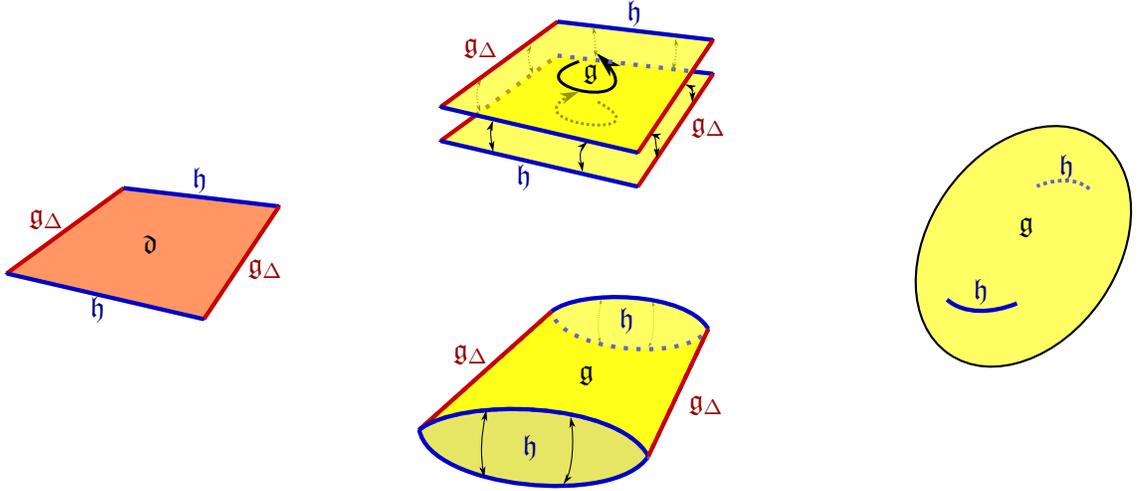
\caption{\label{fig:PBoalchDblSymp} 
The moduli spaces for the quilted surfaces pictured above may each be identified with the Lu-Weinstein double symplectic groupoid integrating the Lie-Poisson structures on $G$ and $H$.
In the leftmost quilted surface, a single domain carries the structure Lie algebra $\mf{d}$, while each of the domains in the other quilted surfaces carry the structure Lie algebra $\g$. The double-ended arrows between edges in the middle two quilted surfaces signify that those pairs of edges have been coloured by the corresponding Lagrangian relations. The rightmost quilted surface depicts a sphere containing two domain walls each coloured by $\h$.}
\end{center}
\end{figure}

Of course $G$ and $H$ are Poisson Lie groups, and we may construct the Lu-Weinstein double symplectic groupoid integrating the Lie-Poisson structures on $G$ and $H$ as a moduli space $\mc{M}$, as in \cref{ex:LuWeinSympDbl}. Specifically, $\mc{M}$ is a moduli space of $\mf{d}$-valued connections over a square, where alternating edges of the square are coloured with $\g$ and $\h$ as in the leftmost quilted surface pictured in \cref{fig:PBoalchDblSymp}.

However, since $\mf{d}=\g\oplus\bar \g$, we may equally well view $\mc{M}$ as a moduli space of $\g$ connections on two squares, where alternating edges of the first square are sewn to the corresponding edges of the second square using the Lagrangian relations $\g_\Delta\subset \g\oplus\bar\g$ and $\h\subset\g\oplus\bar\g$, respectively (see the middle two quilted surfaces in \cref{fig:PBoalchDblSymp}).

Now $\g_\Delta\subseteq\g\oplus\bar\g$ is just the graph of the identity map, so a $\g_\Delta$-coloured domain wall relates the (identical) structure groups in the incident domains by identifying them. Effectively, a $\g_\Delta$-coloured domain wall
 can be erased. Thus $\mc{M}$ may be viewed as a moduli space of $\g$-connections on the cylinder, where either boundary of the cylinder has been broken into two segments which are then sewn to each other using the Lagrangian relation $\h\subset\g\oplus\bar\g$.
 
 That is to say, $\mc{M}$ is a moduli space of $\g$-connections over the sphere $S^2$, where two (contractible, non-intersecting) domain walls $\gamma_1,\gamma_2\subset S^2$ have been coloured with $\h$ (see the rightmost quilted surface in \cref{fig:PBoalchDblSymp}).
 
Thus, in this (quasi-triangular) case, the Lu-Weinstein double symplectic groupoid can be identified with a certain moduli space of connections on the sphere. This fact was first discovered by Fock and Rosly \cite{Fock:1999wz} (in terms of graph connections) and Boalch \cite{Boalch:2007ty,Boalch:2009tn} (in the case where $\g$ is reductive and endowed with the standard quasi-triangular Lie bialgebra structure). Moreover, Boalch's perspective shows that placing these contractible domain walls on the sphere has the much deeper interpretation of prescribing certain irregular singularities for the connection.

\begin{remark}
In fact, Boalch \cite{Boalch:2009tn,Boalch:2011vt,Boalch:2007ty} also provides an interpretation of these $\h$-coloured domain wall in terms of quasi-Hamiltonian geometry (in the case where $\g$ is reductive and endowed with the standard quasi-triangular Lie-bialgebra structure). Indeed, in this case, a neighborhood of each domain wall may be identified with the quilted surface described in \cref{ex:FissionSpace} (for $r=1$), for which the corresponding moduli space is Boalch's fission space.
\end{remark}
\end{example}

\subsection{Poisson structures}
In this section, we will describe some Poisson structures which may be constructed using our approach. Later, in \cref{sec:PoisStrDirRed} we will generalize these results to the case where $\g$ is a quasi-triangular Lie quasi-bialgebra rather than a quadratic Lie algebra.

Let $(\Sigma,V)$ be a marked surface with boundary graph $\Gamma$. 
First, recall from \cref{thm:qhamIntro} that the moduli space $\mc{M}_{\Sigma,V}(G)$ for a marked surface $(\Sigma,V)$ carries a $(\mf{d}^{V_\Gamma},\g^{V_\Gamma},G^{E_\Gamma})$-quasi-Hamiltonian structure,
where $\mf{d}=\g\oplus\ol{\g}$, the Lagrangian Lie subalgebra $\mf{g}^{V_\Gamma}=\mf{g}_\Delta^{V_\Gamma}\subset\mf{d}^{V_\Gamma}$ is embedded as the diagonal, and  $\mf{d}^{V_\Gamma}$ acts on the $e\in E_\Gamma$-th factor of $G^{E_\Gamma}$ via the vector field 
\begin{equation}\label{eq:VertAct}\xi_{\on{out}(e)}^L-\eta_{\on{in}(e)}^R,\quad (\xi,\eta)\in\g^{V_\Gamma}\oplus\ol{\g}^{V_\Gamma}=\mf{d}^{V_\Gamma}.\end{equation}
 Here the superscripts ${}^L$,${}^R$ denote left,right invariant vector fields. 
The bivector field on $\mc{M}_{\Sigma,V}(G)$ is computed in \cref{sec:qPoisModSpc} and leads to the result of \cite[Theorem~3]{LiBland:2012vo}, which we summarize briefly.\footnote{In fact, the computation in \cref{sec:qPoisModSpc} results in  minus the bivector field described in \cite{LiBland:2012vo}, due to us orienting $\partial\Sigma$ in the opposite way. 

Strictly speaking, the bivector field on $\mc{M}_{\Sigma,V}(G)$ depends on the choice of a complement $\mf{k}\subset\mf{d}^{V_\Gamma}$ to $\h=\g^{V_\Gamma}$. In this case, $\mf{k}$ can be chosen canonically as $\mf{k}:=\g_{\bar\Delta}^{V_\Gamma}$, where
$$\g_{\bar\Delta}:=\{(\xi,-\xi)\in(\g\oplus\ol{\g})\}.$$
} 

If $a,b\in\Pi_1(\Sigma,V)$, let us represent them by transverse smooth paths $\alpha,\beta$. For any point $A$ in their intersection, let
\begin{align*}
\lambda(A)&=
\begin{cases}
1 &\text{if }A\in\partial\Sigma\\
2 &\text{otherwise}
\end{cases}\\
\on{sign}(A):=\on{sign}(\alpha,\beta;A)&=
\begin{cases}
1 &\text{if }(\beta'|_A,\alpha'|_A)\text{ is positively oriented}\\
-1 &\text{otherwise.}
\end{cases}
\end{align*}
as in \cref{fig:Signp}.
\begin{figure}[h]
\begin{center}
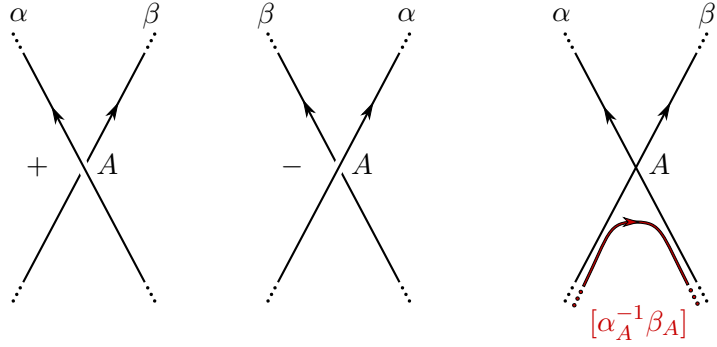
\end{center}
\caption{\label{fig:Signp} $\on{sign}(A)=\pm 1$ is determined by comparing the orientation of $\alpha$ and $\beta$ with that of $\Sigma$. The path $[\alpha_A^{-1}\beta_A]$ is shown to the right.}
\end{figure}

Let $\alpha_A$ denote the portion of $\alpha$ parametrized from the beginning up to the point $A$. Finally, let
$$(a,b):=\sum_A \lambda(A)\on{sign}(A) [\alpha_A^{-1}\beta_A]\in\mathbb{Z}\Pi_1(\Sigma,V).$$
(When $V$ contains only one point, this is a skew symmetrized version of the intersection form described in \cite{Turaev:2007jh}). 

Then for any $a,b\in\Pi_1(\Sigma,V)$,
\begin{equation}\label{eq:pi_hpair}
\pi\big(\on{ev}_{a}^*(g^{-1}\d g),\on{ev}_{b}^*(g^{-1}\d g)\big)=\frac{1}{2}(\Ad_{\on{ev}_{(a,b)}}\otimes 1)\,s,
\end{equation}
where $\on{ev}_a,\on{ev}_b:\mc{M}_{\Sigma,V}(G)=\Hom(\Pi_1(\Sigma,V),G)\to G$ denotes evaluation, $s\in\g\otimes\g$ is the inverse of the quadratic form, and $g^{-1}\d g$ denotes the left invariant Maurer-Cartan form on $G$.

\begin{example}[The two sided polygon, $P_2$]\label{ex:PoisP2}
Suppose $(\Sigma, V)=P_2$ is the disk with two marked points and $E_\Gamma=\{e_1,e_2\}$, as in \cref{fig:2MrkDisk}. Then we may identify $\mc{M}_{P_2}(G)$ with $G$ via $g_{e_1}$ (since $g_{e_2}=g_{e_1}^{-1}$).
Under this identification, the bivector field is trivial, $\pi_{P_2}=0$ (cf. \cite{LiBland:2012vo}).

\begin{figure}
\begin{center}
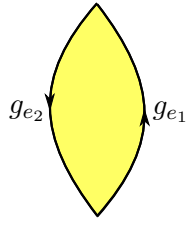
\caption{\label{fig:2MrkDisk} The cyclic product of holonomies, $g_{e_2}g_{e_1}$, is trivial, so $g_{e_2}=g_{e_1}^{-1}$.}
\end{center}
\end{figure}
\end{example}

\subsubsection{Colouring edges}\label{sec:PoisColSurf}
Suppose now that we colour each marked point $v\in V_\Gamma$ with a Lagrangian subalgebra 
$\mf{l}_v\subseteq \mf{d},$ 
 and each edge $e\in E_\Gamma$ with a submanifold $S_e\subseteq G$ in a compatible way: Specifically, we require that $$S:=\prod_{e\in E_\Gamma}S_e\subseteq G^{E_\Gamma}$$ be $\mf{l}$-invariant, where
$$\mf{l}:=\bigoplus_{v\in V_\Gamma}\mf{l}_v\subseteq \mf{d}^{V_\Gamma},$$
and the action is given in \cref{eq:VertAct}. Then \cref{thm:reductionIntro} implies that 
$$\mc{M}_{red}:=\mu^{-1}(S)/(\mf{l}\cap\g^{V_\Gamma})$$
is a Poisson manifold (provided it is a manifold). 

It is not difficult to describe the bivector field on $\mu^{-1}(S)/(\mf{l}\cap\g^{V_\Gamma})$. First, let $\g_\Delta\subseteq\mf{d}$ denote the diagonal and  $\g_{\bar\Delta}\subseteq\mf{d}$ the off-diagonal:
$$\g_{\bar\Delta}:=\{(\xi,-\xi)\in(\g\oplus\ol{\g})\}.$$
Each Lagrangian Lie subalgebra $\mf{l}_v\subseteq\mf{d}^{V_\Gamma}$ defines an element $\tau_v\in \wedge^2\big(\h/(\mf{l}\cap\h)\big)$ by the equation $$\mf{l}_v=\{(\alpha+\tau_v^\sharp\alpha)\mid \alpha\in\g_{\bar\Delta},\;\la\alpha,\mf{l}_v\cap\g_\Delta\ra=0\}+\mf{l}_v\cap\g_\Delta.$$
Here $\tau_v^\sharp\alpha\in \g_\Delta/(\mf{l}_v\cap\g_\Delta)$ is defined by
$$\tau_v^\sharp\alpha:=\frac{1}{2}\sum_{i,j}\tau^{ij}_v\la \alpha,\xi_i\ra\xi_j-\frac{1}{2}\sum_{i,j}\tau^{ij}_v\la \alpha,\xi_j\ra\xi_i$$
when we represent $\tau_v$ as $\tau_v=\frac{1}{2}\sum_{i,j}\tau^{ij}_v\xi_i\wedge\xi_j$.

We have the following theorem:
\begin{theorem}\label{thm:PoissThmIntro}
If the intersection, $\mu^{-1}(S)$, of $S\times\mc{M}_{\Sigma,V}(G)$ with the graph of $\mu$ is clean, and the $\mf{l}\cap\g^{V_\Gamma}$-orbits of $\mu^{-1}(S)$ form a regular foliation,
then the bivector field
$$\pi+\sum_{v\in V_\Gamma}\rho_v(\tau_v)\in \Gamma\big(\wedge^2 T(\mu^{-1}(S))/\rho(\mf{l}\cap\g^{V_\Gamma})\big)$$
is $\mf{l}\cap \g^{V_\Gamma}$ invariant and descends to define the Poisson structure on $\mu^{-1}(S)/(\mf{l}\cap\g^{V_\Gamma})$. Here  $\rho:\g^{V_\Gamma}\to\mf{X}(\mc{M}_{\Sigma,V}(G))$ denotes the action by infinitesimal gauge transformations at the marked points, and $\rho_v$ is the restriction of $\rho$ to the $v\in V_\Gamma$-th factor. 

Moreover, for any $\mf{l}$-orbit $O\subseteq S$, the image of $\mu^{-1}(O)$ in $\mu^{-1}(S)/(\mf{l}\cap\g^{V_\Gamma})$ will be a symplectic leaf.
\end{theorem}
\begin{proof}
This will follow from \cref{thm:BivPartRed}, while the statement for the symplectic leaves will follow from \cref{thm:ExactPartRed}.
\end{proof}
\begin{remark}
As in \cref{sec:DomWallBrnch}, one may also sew domains together to obtain Poisson structures on the moduli spaces of branched surfaces. The general reduction statement is \cref{thm:BivPartRed}. 
\end{remark}

\begin{example}[Double Poisson Lie group \cite{LiBland:2012vo}]\label{ex:DblPoisLieGrp}
 Suppose that $\g=\mf{e}\oplus\mf{f}$ as a vector space, where $\mf{e},\mf{f}\subseteq\g$ are Lagrangian Lie subalgebras, and that $\mf{e},\mf{f}\subseteq\g$ integrate to Lie subgroups $E,F\subseteq G$ such that $E\cap F=1$. Let $\Sigma$ be a disk with two marked points labelled as in \cref{fig:DblPoissLieGrpDsk}. We colour the edges with the full group $G$, 
\begin{align*}
S_{e_1}&=G,&S_{e_2}&=G,
\end{align*}
and the vertices as
\begin{align*}
\mf{l}_{v_1}&=\mf{e}\oplus\mf{f},&\mf{l}_{v_2}&=\mf{f}\oplus\mf{e}.
\end{align*}
Therefore, $$\mu^{-1}(S)=\mc{M}_{P_2}=\{(g_1,g_2)\in G\times G\mid g_1g_2=1\}.$$ Meanwhile the residual gauge transformations, 
$$\bigoplus_{v\in V_\Gamma}\mf{l}_v\cap\g_\Delta=0$$
(since $\mf{e}\cap\mf{f}=0$). Thus we may identify the moduli space $\mc{M}_{red}\cong G$, via the map $(g_1,g_2)\to g_1$.

\begin{figure}[h]
\begin{center}
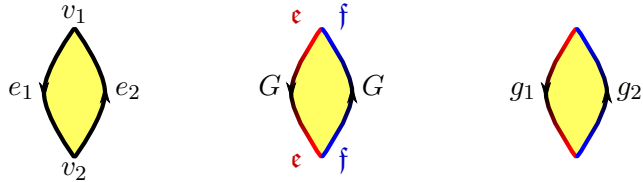
\caption{\label{fig:DblPoissLieGrpDsk}The double Poisson Lie group.
The holonomies $g_i\in G$ satisfy $g_1g_2=1$.}
\end{center}
\end{figure}
Next, we compute the bivector field, $\pi=\pi_{P_2}+\rho_{v_1}(\tau_{v_1})+\rho_{v_2}(\tau_{v_2})$. Now, as explained in \cref{ex:PoisP2}, $\pi_{P_2}=0$ so only the term $\sum_i\rho_{v_i}(\tau_{v_i})$ contributes.
 Now,
$$\tau_{v_1}=\frac{1}{2}\sum_i (\zeta^i,\zeta^i)\wedge(\eta_i,\eta_i),\quad \tau_{v_2}=\frac{1}{2}\sum_i (\eta_i,\eta_i)\wedge(\zeta_i,\zeta_i),$$
where $\{\eta_i\}\subset\mf{e}$ and $\{\zeta^i\}\subset\mf{f}$ are basis in duality.
Therefore, 
$$\pi=\frac{1}{2}\sum_i (\zeta^i)^L\wedge (\eta_i)^L+(\eta_i)^R\wedge(\zeta^i)^R.$$ 
In fact, $\pi$ defines the Poisson Lie group structure on $G$ corresponding the double Lie bialgebra structure on $\g$ resulting from the Manin triple $(\g,\mf{e},\mf{f})$ \cite{thesis-3,LiBland:2012vo}.
The symplectic leaves are computed as the restriction of the $\mf{l}$-orbits, which in this case can be seen to correspond to the orbits of the dressing action on $G$.\footnote{In fact computing the symplectic leaves via \cref{thm:PoissThmIntro} is precisely the computation found in \cite{LiBland:2009ul}.}

\begin{remark}
In the case where $\g$ is a quasi-triangular Lie-bialgebra, the double Poisson Lie group was constructed as a moduli space of graph connections in the work of Fock and Rosly \cite{Fock:1999wz}.
\end{remark}

\end{example}

\begin{example}[Poisson Lie group \cite{LiBland:2012vo}]\label{ex:PoisLie}
Suppose the Lie groups $G,E,F$ and their Lie algebras are as in \cref{ex:DblPoisLieGrp}, and let $\Sigma$ be a disk with two marked points, as in \cref{fig:PoissLieGrpDsk}. We colour the vertices as in \cref{ex:DblPoisLieGrp}, but 
we colour the edges as 
\begin{align*}
S_{e_1}&=E,&S_{e_2}&=G.
\end{align*}
Therefore, $$\mu^{-1}(S)=\{(e,g)\in E\times G\mid eg=1\},$$ while the residual gauge transformations are trivial, as before. Thus we may identify the moduli space $\mc{M}_{red}\cong E$, via the map $(e,g)\to e$.

\begin{figure}
\begin{center}
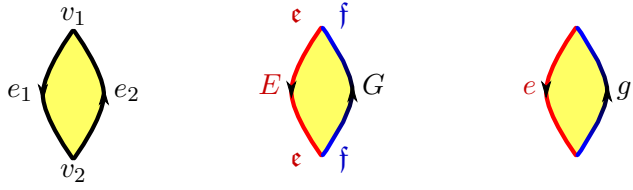
\caption{\label{fig:PoissLieGrpDsk} The Poisson Lie group.
The holonomies $g\in G$ and $e\in E$ satisfy $eg=1$.}
\end{center}
\end{figure}
The bivector field, $\pi$, on $E$ is computed to be the restriction of the bivector field 
$$\frac{1}{2}\sum_i (\zeta^i)^L\wedge (\eta_i)^L+(\eta_i)^R\wedge(\zeta^i)^R.$$ 
on $G$ to $E\subseteq G$. Thus, $\pi$ defines the Poisson Lie group structure on $E$ corresponding the Manin triple $(\g,\mf{e},\mf{f})$ \cite{thesis-3,LiBland:2012vo}.
As before, the symplectic leaves are computed as the restriction of the $\mf{l}$-orbits. Once again, they are precisely the orbits of the dressing action on $E$.
\end{example}

\begin{example}[Poisson homogenous spaces]
Suppose the Lie groups $G,E,F$ and their Lie algebras are as in \cref{ex:DblPoisLieGrp}, and let $\Sigma$ be a disk with two marked points, as in \cref{fig:PoisHomDsk}. Suppose further that $\mf{h}\subseteq\g$ is a Lagrangian subalgebra such that $\mf{k}:=\mf{h}\cap\mf{e}$ integrates to a closed Lie subgroup $K\subseteq E$.
 We colour the edges as in \cref{ex:PoisLie}, but we colour the vertices as
\begin{align*}
\mf{l}_{v_1}&=\mf{e}\oplus\mf{h},&\mf{l}_{v_2}&=\mf{f}\oplus\mf{e}.
\end{align*}
Therefore, $$\mu^{-1}(S)=\{(e,g)\in E\times G\mid eg=1\},$$ while the residual gauge transformations are $G\times K$ acting as $$(g,k):(e,g)\to (ek^{-1},kg).$$  We may identify the moduli space $\mc{M}_{red}\cong E/K$, via the map $(e,g)\to [e]$.
The Poisson structure on $\mc{M}_{red}$ is the Poisson homogenous structure corresponding to the Lagrangian Lie subalgebra $\h\subseteq\g$ in Drinfel'd's classification \cite{Drinfeld:1993il}. We leave it to the reader to compute the bivector field and symplectic leaves on $E/K$ via \cref{thm:PoissThmIntro}.

\begin{figure}
\begin{center}
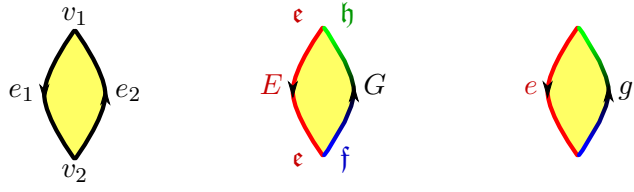
\caption{\label{fig:PoisHomDsk} The Poisson homogeneous space corresponding to the Lagrangian Lie subalgebra $\mf{h}\subseteq\g$.
The holonomies $g\in G$ and $e\in E$ satisfy $eg=1$.}
\end{center}
\end{figure}
\end{example}

\begin{example}[Affine Poisson structure on $G$ \cite{thesis-3}]
Generalizing the setup found in \cref{ex:DblPoisLieGrp}, we suppose that $\mf{h}\subseteq\g$ is a Lagrangian subalgebra which is also complementary (as a vector space) to $\mf{e}\subseteq\g$.
As in \cref{ex:DblPoisLieGrp}, we colour the edges with the full group $G$, but the vertices as
\begin{align*}
\mf{l}_{v_1}&=\mf{e}\oplus\mf{h},&\mf{l}_{v_2}&=\mf{f}\oplus\mf{e},
\end{align*}
(cf. \cref{fig:AffineGrpDsk}).
As before, we have  $\mu^{-1}(S)=\mc{M}_{P_2}$, and the residual gauge transformations are trivial. 
 Thus, we may identify the moduli space $\mc{M}_{red}\cong G$, via the map $(g_1,g_2)\to g_1$.

\begin{figure}[h]
\begin{center}
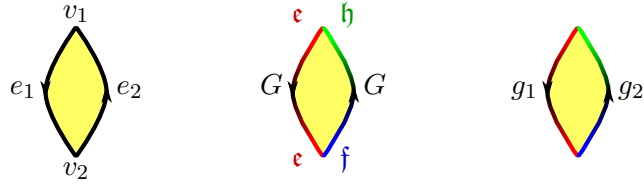
\caption{\label{fig:AffineGrpDsk}The affine Poisson structure on $G$.
The holonomies $g_i\in G$ satisfy $g_1g_2=1$.}
\end{center}
\end{figure}
Meanwhile the bivector field, $\pi$, is 
$$\pi=\frac{1}{2}\sum_i (\zeta_{\mf{f}}^i)^L\wedge (\eta_i)^L+(\eta_i)^R\wedge(\zeta_{\mf{h}}^i)^R,$$
where the bases $\{\zeta_{\mf{f}}^i\}\subseteq\mf{f}$ and $\{\zeta_{\mf{h}}^i\}\subseteq\mf{h}$ are both dual to $\{\eta_i\}\subseteq\mf{g}$.
In fact, $\pi$ defines the affine Poisson structure on $G$ corresponding to the Manin triples $(\g,\mf{e},\mf{f})$ and $(\g,\mf{e},\mf{h})$ (as described by Lu \cite{thesis-3}).
The symplectic leaves are computed as the restriction of the $\mf{l}$-orbits.
\end{example}

\begin{example}[Lu Yakimov Poisson homogenous spaces \cite{LiBland:2012vo}]
Suppose the Lie groups $G,E,F$ and their Lie algebras are as in \cref{ex:DblPoisLieGrp}, and that $C\subseteq G$ is a closed Lie subgroup whose Lie algebra $\mf{c}\subseteq\g$ is coisotropic. Let $\Sigma$ be a disk with two marked points and edges and vertices labelled as in \cref{fig:LuYakimovDsk}. We colour the edges with the full group $G$, 
\begin{align*}
S_{e_1}&=G,&S_{e_2}&=G,
\end{align*}
and the vertices as
\begin{align*}
\mf{l}_{v_1}&=\mf{l}_{\mf{c}}=\{(\xi,\xi')\in\mf{c}\oplus\mf{c}\mid\xi-\xi'\in\mf{c}^\perp\},&\mf{l}_{v_2}&=\mf{f}\oplus\mf{e}\end{align*}
(cf. \cref{fig:LuYakimovDsk}).
Therefore, $$\mu^{-1}(S)=\mc{M}_{P_2}=\{(g_1,g_2)\in G\times G\mid g_1g_2=1\}.$$  Meanwhile the residual gauge transformations, 
$$\bigoplus_{v\in V_\Gamma}\mf{l}_v\cap\g_\Delta=\mf{c}_{v_0},$$
where $\mf{c}_{v_0}=\{(\xi,\xi)\in\mf{c}\oplus\mf{c}\}\subseteq \mf{l}_{v_0}$.
Thus, up to a gauge transformation, 
$(g_1,g_2)\sim(g_1c^{-1},cg_2)$
 (for any $c\in C$), and we may identify the moduli space $\mc{M}_{red}\cong G/C$, via the map $(g_1,g_2)\to [g_1]$.

\begin{figure}
\begin{center}
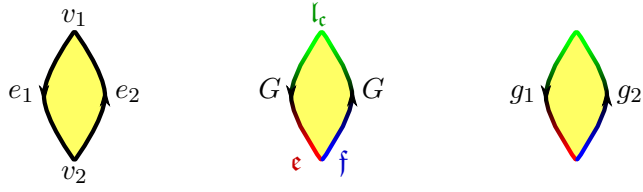
\caption{\label{fig:LuYakimovDsk}Lu-Yakimov Poisson homogenous spaces.
The holonomies $g_i\in G$ satisfy $g_1g_2=1$.}
\end{center}
\end{figure}
The bivector field, $\pi$, on $G/C$ can be computed to be the projection of the bivector field 
$$\frac{1}{2}\sum_i (\eta_i)^R\wedge(\zeta^i)^R.$$ 
on $G$ to $G/C$. Thus, $\pi$ defines the Lu-Yakimov Poisson structure on $G/C$ corresponding the Manin triple $(\g,\mf{e},\mf{f})$ \cite{Lu06,LiBland:2012vo}.
\end{example}

\begin{example}[Quasi-Triangular Poisson Lie groups \cite{Boalch:1999wk,Boalch:2001fw,Boalch:2001gw,Boalch:2011vt,Boalch:2009tn,Boalch:2007ty}]

As in \cref{ex:DblSymplBoalch}, suppose that $(\mf{d};\g_\Delta,\mf{h})$  is the Manin triple corresponding to a (non-degenerate) quasi-triangular Lie-bialgebra where $\g_\Delta\subset \g\oplus\bar\g=\mf{d}$ is the diagonal, and $\mf{h}\subset\g\oplus\bar\g$ is a complementary Lagrangian subalgebra. As before we may view $\mf{h}$ as either a Lagrangian subalgebra of $\mf{d}$, or a Lagrangian relation from $\g$ to itself. We let $G$ and $H\subset G\times G=D$ denote the simply connected (resp. connected) Lie groups corresponding to $\g$, $\h$ and $\mf{d}$. 

\begin{figure}
\begin{center}
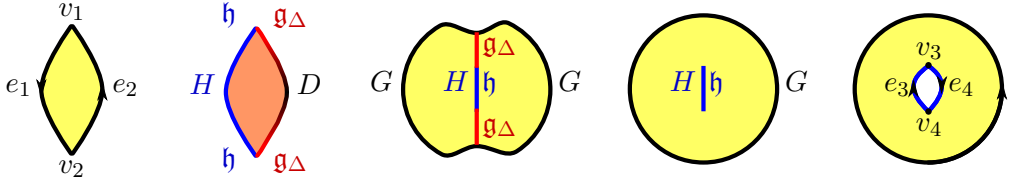
\caption{\label{fig:PoissLieGrpBoalch}The dual Poisson Lie group, $H$. The holonomies $h_1,d_2\in D=G\times G$ along the edges $e_1$ and $e_2$ (respectively) satisfy $h_1d_2=1$, with $h_1\in H$; while the holonomies $g_3,g_4\in G$ along the edges $e_3$ and $e_4$ satisfy $(g_3^{-1},g_4)\in H$}
\end{center}
\end{figure}

We may construct the Lie-Poisson structure on $H$ by identifying it with a moduli space $H\cong \mc{M}$, as in \cref{ex:PoisLie}. More specifically let $\Sigma$ be a disk with two marked points, with edges and vertices labelled as in \cref{fig:PoissLieGrpBoalch}. As pictured in \cref{fig:PoissLieGrpBoalch}, we color the edges as
\begin{align*}
S_{e_1}&=H,&S_{e_2}&=D,
\end{align*}
and the vertices as
\begin{align*}
\mf{l}_{v_1}&=\mf{h}\oplus\mf{g}_\Delta,&\mf{l}_{v_2}&=\mf{g}_\Delta\oplus \mf{h}.
\end{align*}
Therefore, $$\mu^{-1}(S)=\{(h_1,d_2)\in H\times D\mid h_1d_2=1\}.$$ Meanwhile the residual gauge transformations are trivial (since $\h$ and $\g_\Delta$ are complements). Thus we may identify the moduli space $\mc{M}\cong H$, via the map $(h_1,d_2)\to h_1$.

As in \cref{ex:DblSymplBoalch}, since $\mf{d}=\g\oplus\bar\g$, we may also view $\mc{M}$ as a moduli space of $\g$ connections on two stacked copies of $\Sigma$, where their left edges are sewn using the Lagrangian relation $\h\subseteq \g\oplus\bar\g$, while their top and bottom right half-edges are sewn using the relation $\g_\Delta\subseteq \g\oplus\bar\g$. In the same fashion as one opens a pita-pocket, we may imagine pulling the two copies of $\Sigma$ apart starting from the middle of their right edges (leaving the left edges incident to each other), in which case the resulting quilted surface is pictured in the middle of \cref{fig:PoissLieGrpBoalch}.

Since $\g_\Delta\subseteq \g\oplus\bar\g$ is just the graph of the identity map, a $\g_\Delta$-coloured domain wall relates the (identical) structure groups in the incident domains by identifying them. Effectively, a $\g_\Delta$-coloured domain wall can be erased. Thus $\mc{M}$ may be viewed as a moduli space of $\g$-connections on the annulus, $\Sigma'$, where the inner boundary has been broken into two segments which are then sewn together using the Lagrangian relation $\h\subseteq \g\oplus\bar\g$.

More explicitly, the inner boundary of the annulus $\Sigma'$ has two marked points, $v_3$ and $v_4$, dividing it into two segments $e_3$ and $e_4$, while the outer boundary has no marked points. As pictured in the right two quilted surfaces of \cref{fig:PoissLieGrpBoalch}, we colour the edges as 
$$S=S_{e_3}\times S_{e_4}=\{(g_3,g_4)\in G\times G\mid (g_3^{-1},g_4)\in H\}\cong H $$
and the vertices as
\begin{align*}
\mf{l}_{v_3}&=\h^{-1}\subseteq \g\oplus\bar\g,&\mf{l}_{v_4}&=\h\subseteq \g\oplus\bar\g,
\end{align*}
where $\h^{-1}=\{(\xi,\eta)\in\g\oplus\bar\g\mid (\eta,\xi)\in \h\}$ (the inverse refers to the pair-groupoid structure).

Since $\mu$ is a diffeomorphism, $\mu^{-1}(S)\cong S\cong H.$ The residual gauge transformations are trivial,
since $\h$ and $\g_\Delta$ are complements. As before, we may identify the moduli space $\mc{M}\cong H$.

That is to say, the (dual) Poisson Lie group $H\cong \mc{M}$ is naturally a moduli space of flat $\g$-connections over the disk containing a (contractible) $\h$-coloured domain wall. Thus, in this (quasi-triangular) case, the Poisson Lie group $H$ can be identified with a certain moduli space of flat $\g$-connections on the disk. This fact was first observed by Fock and Rosly \cite{Fock:1999wz} (in terms of graph connections) and Boalch \cite{Boalch:1999wk,Boalch:2001fw,Boalch:2001gw,Boalch:2011vt,Boalch:2009tn,Boalch:2007ty} (in the case where $\g$ is reductive and endowed with the standard quasi-triangular Lie-bialgebra structure). Moreover, Boalch's perspective shows that placing this contractible domain wall on the disk has the much deeper interpretation of prescribing a certain irregular singularity for the connection.

\begin{remark}
In fact, Boalch \cite{Boalch:2009tn,Boalch:2011vt,Boalch:2007ty} also provides an interpretation of this $\h$-coloured domain wall in terms of quasi-Hamiltonian geometry (in the case where $\g$ is reductive and endowed with the standard quasi-triangular Lie-bialgebra structure). Indeed, in this case, a neighborhood of the domain wall may be identified with the quilted surface described in \cref{ex:FissionSpace} (for $r=1$), for which the corresponding moduli space is Boalch's fission space.
\end{remark}

\end{example}

\section{Background}
\subsection{Courant algebroids}
Courant algebroids and Dirac structures are the basic tools in the theory of generalized moment maps.
\begin{definition}[\cite{ManinTriplesBi}]
A \emph{Courant algebroid} is a vector bundle $\mbb{E}\to M$ endowed with a non-degenerate symmetric bilinear form $\la\cdot,\cdot\ra$ on the fibres, a bundle map $\mbf{a}:\mbb{E}\to TM$ called the \emph{anchor} and a bracket $\Cour{\cdot,\cdot}:\sect(\mbb{E})\times\sect(\mbb{E})\to\sect(\mbb{E})$ called the \emph{Courant bracket} satisfying the following axioms for sections $e_1,e_2,e_3\in\sect(\mbb{E})$ and functions $f\in C^\infty(M)$:
\begin{enumerate}
\item[c1)] $\Cour{e_1,\Cour{e_2,e_3}}=\Cour{\Cour{e_1,e_2},e_3}+\Cour{e_2,\Cour{e_1,e_3}}$, 
\item[c2)] $\mbf{a}(e_1)\la e_2,e_3\ra=\la\Cour{e_1,e_2},e_3\ra+\la e_2,\Cour{e_1,e_3}\ra$, 
\item[c3)] $\Cour{e_1,e_2}+\Cour{e_2,e_1}=\mbf{a}^*\d\la e_1,e_2\ra$.
\end{enumerate}
Here $\mbf{a}^*:T^*M\to\mbb{E}^*\cong\mbb{E}$ is the map dual to the anchor.

A subbundle $E\subseteq\mbb{E}\rvert_S$ along a submanifold $S\subseteq M$ is called a \emph{Dirac structure with support on $S$} if 
$$e_1\rvert_S,e_2\rvert_S\in\sect(E)\Rightarrow \Cour{e_1,e_2}\rvert_S\in\sect(E),$$
(it is \emph{involutive}) and $E^\perp=E$ (it is \emph{Lagrangian}). If $S=M$, then $E$ is simply called a \emph{Dirac structure}.
\end{definition}
\begin{remark}
As shown in \cite{Uchino02,Roytenberg99}, one may also derive the following useful identities from the Courant axioms:
\begin{enumerate}
\item[c4)] $\Cour{e_1,fe_2}=f\Cour{e_1,e_2}+(\mbf{a}(e_1)f )e_2$ 
\item[c5)] $\Cour{fe_1,e_2}=f\Cour{e_1,e_2}-(\mbf{a}(e_2)f) e_1+\la e_1,e_2\ra\mbf{a}^*\d f$
\item[c6)] $\mbf{a}\Cour{e_1,e_2}=[\mbf{a}(e_1),\mbf{a}(e_2)]$
\end{enumerate}
\end{remark}
For any Courant algebroid $\mbb{E}$, we denote by $\ol{\mbb{E}}$ the Courant algebroid with the same 
bracket and anchor, but with the metric negated.\footnote{Note that this also negates the map $\mbf{a}^*:T^*M\to\mbb{E}^*\cong\mbb{E}$, so axiom c3) still holds.}

\begin{example}
A Courant algebroid over a point is a quadratic Lie algebra. Dirac structures are Lagrangian Lie subalgebras.
\end{example}

\begin{example}[Standard Courant algebroid \cite{Courant:tm,Courant:1990uy}]
The vector bundle $\mbb{T}M:=TM\oplus T^*M$ is a Courant algebroid with metric
$$\la v_1+\mu_1,v_2+\mu_2\ra=\mu_1(v_2)+\mu_2(v_1),\quad v_1,v_2\in TM,\quad\mu_1,\mu_2\in T^*M$$
and bracket
$$\Cour{X+\alpha,Y+\beta}=[X,Y]+\Lied_X\beta-\iota_Yd\alpha,\quad X,Y\in \mf{X}(M),\alpha,\beta\in \Omega^1(M).$$

%
 
\end{example}

The standard Courant algebroid is an example of an important class of Courant algebroids called \emph{exact Courant algebroids}.
\begin{definition}[Exact Courant algebroids \cite{LetToWein,Severa:2001}]
A Courant algebroid $\mbb{E}\to M$ is called \emph{exact} if the sequence
\begin{equation}\label{eq:ExCourSeq}0\to T^*M\xrightarrow{\mbf{a}^*}\mbb{E}\xrightarrow{\mbf{a}} TM\to 0\end{equation}
is exact.
\end{definition}

If a Dirac structure $E\subseteq\mbb{E}$ is supported on $S\subseteq M$, then $\mbf{a}(E)\subseteq TS$, and $\mbf{a}^*\big(\ann(TS)\big)\subseteq E$.
\begin{definition}[Exact Dirac structures]
Suppose $\mbb{E}\to M$ is a Courant algebroid (not necessarily exact), and $E\subseteq \mbb{E}$ is a Dirac structure with support on $S\subseteq M$. We say that $E$ is an \emph{exact Dirac structure} if the sequence
\begin{equation}\label{eq:ExDirSeq}0\to \ann(TS)\xrightarrow{\mbf{a}^*} E\xrightarrow{\mbf{a}} TS\to 0\end{equation}
is exact.
\end{definition}

\begin{lemma}\label{lem:ExDirStr}
Suppose $\mbb{E}\to M$ is a Courant algebroid, and $E\subseteq \mbb{E}$ is a Dirac structure with support on $S\subseteq M$. Then the following two statements are equivalent:
\begin{enumerate} 
\item $E$ is an exact Dirac structure.
\item The Courant algebroid $\mbb{E}$ is exact along $S$ (that is, the sequence \cref{eq:ExCourSeq} is exact at every $x\in S$), and $\mbf{a}:E\to TS$ is surjective.
\end{enumerate}
 
\end{lemma}
\begin{proof}
Let $x\in S$, and consider the commutative diagram 
$$\begin{tikzpicture}
\mmat{m}{
0&\ann(T_xS)&E_x&T_xS&0\\
0&T^*_xM&\mbb{E}_x&T_xM&0\\
0&T^*_xM/\ann(T_xS)&E^*_x&T_xM/T_xS&0\\};
\draw[->] (m-1-1) edge (m-1-2)
	(m-1-2) edge node {$\mbf{a}^*$} (m-1-3)
	(m-1-3) edge node {$\mbf{a}$} (m-1-4)
	(m-1-4) edge (m-1-5)
	(m-1-2) edge (m-2-2)
	(m-1-3) edge (m-2-3)
	(m-1-4) edge (m-2-4)
	(m-2-1) edge (m-2-2)
	(m-2-2) edge node {$\mbf{a}^*$} (m-2-3)
	(m-2-3) edge node {$\mbf{a}$} (m-2-4)
	(m-2-4) edge (m-2-5)
	(m-2-2) edge (m-3-2)
	(m-2-3) edge (m-3-3)
	(m-2-4) edge (m-3-4)
	(m-3-1) edge (m-3-2)
	(m-3-2) edge node {$\mbf{a}^*$} (m-3-3)
	(m-3-3) edge node {$\mbf{a}$} (m-3-4)
	(m-3-4) edge (m-3-5);
\end{tikzpicture}$$
Note that the vertical sequences are exact. 

Suppose that $E$ is an exact Dirac structure. Then the top horizontal sequence is exact, by assumption.  The lower horizontal sequence is dual to the top sequence, and hence also exact. The five lemma then implies that all terms in the long exact sequence vanish. In particular, the central horizontal sequence is exact.

Conversely, suppose that the Courant algebroid $\mbb{E}$ is exact along $S$ and $\mbf{a}:E_x\to T_xS$ is surjective (and hence $T^*_xM/\ann(T_xS)\to E_x^*$ is injective). Once again, the five lemma implies that all terms in the long exact sequence vanish. We conclude that $E$ is an exact Dirac structure.
\end{proof}

\begin{example}[Action Courant algebroids \cite{LiBland:2009ul}]\label{ex:ActionCA}
Suppose $\mf{d}$ is a Lie algebra equipped with an invariant metric. 
Given a Lie algebra action $\rho\colon \mf{d}\to \mf{X}(M)$ on a manifold $M$, 
let $\mbb{E}=\mf{d}\times M$ with anchor map $\mbf{a}(\xi,m)=\rho(\xi)_m$, and with
the bundle metric coming from the metric on $\mf{d}$. As shown in
\cite{LiBland:2009ul}, the Lie bracket on constant sections $\mf{d}\subseteq
C^\infty(M,\mf{d})=\sect(\mbb{E})$ extends to a Courant bracket \emph{if and only if}
the stabilizers $\mf{d}_m\subseteq \mf{d}$ are coisotropic, i.e.~ $\mf{d}_m\supseteq \mf{d}_m^\perp$. Explicitly, for $\xi_1,\xi_2\in \sect(\mbb{E})=C^\infty(M,\mf{d})$ the
Courant bracket reads (see \cite[$\mathsection$ 4]{LiBland:2009ul})
\begin{equation}
\label{eq:actioncourant} \Cour{\xi_1,\xi_2}=[\xi_1,\xi_2]+\Lied_{\rho(\xi_1)}\xi_2-\Lied_{\rho(\xi_2)}\xi_1+\rho^*\la\d\xi_1,\xi_2\ra.\end{equation}
Here $\rho^*\colon T^*M\to \mf{d}\times M$ is the dual map to the action map $\rho\colon \mf{d}\times M\to TM$, 
using the metric to identify $\mf{d}^*\cong \mf{d}$. 
We refer to $\mf{d}\times M$ with bracket \eqref{eq:actioncourant} as an \emph{action Courant algebroid}. 
\end{example}

\begin{example}[Cartan Courant algebroid \cite{Severa:2001}]\label{ex:CartCourAlg}
Suppose $\g$ is a Lie group endowed with an invariant metric, $\la\cdot,\cdot\ra$. We let $\ol{\g}$ denote the Lie algebra $\g$ with the metric negated, $-\la\cdot,\cdot\ra$. Suppose $G$ is a Lie group with Lie algebra $\g$ and which preserves the metric. The Lie algebra $\ol{\g}\oplus \g$ acts on $G$ by $\rho:(\ol{\g}\oplus \g)\times G\to TG$,
$$\rho(\xi,\eta)=-\xi^R+\eta^L,\quad \xi,\eta\in\g,$$
where $\xi^L,\xi^R\in\mf{X}(G)$ denotes the left/right-invariant vector field on $G$ which is equal to $\xi\in\g$ at the identity element.

The stabilizer at the identity element is the diagonal subalgebra, $\g_\Delta\subseteq\ol{\g}\oplus\g$ which is Lagrangian. Now $\rho$ is equivariant with respect to the $G$-action on $(\ol{\g}\oplus \g)\times G$ given by
$$g':(\xi,\eta;g)\to (\Ad_{g'}\xi,\eta,g'\cdot g),\quad g'\in G, (\xi,\eta;g)\in (\ol{\g}\oplus \g)\times G,$$ and the left action of $G$ on $TG$. Since this action is transitive on the base of the vector bundles and $G$ preserves the metric on $\g$, it follows that all stabilizers are Lagrangian. Thus $(\ol{\g}\oplus \g)\times G$ is an action Courant algebroid, called the \emph{Cartan Courant algebroid}.

The diagonal subalgebra $\g_\Delta\subseteq(\ol{\g}\oplus \g)$ defines a Dirac structure $$\g_\Delta\times G\subseteq(\ol{\g}\oplus \g)\times G$$ called the \emph{Cartan Dirac structure}.

\begin{remark}
The Cartan Courant algebroid was first introduced in \cite{Severa:2001}, and later simplified to the above description in \cite{Alekseev:2009tg}. The Cartan Dirac structure was discovered independently by Alekseev, \v{S}evera and Strobl \cite{Severa:2001,Klimcik:2002eg,Alekseev00}. The description given above was found in \cite{Alekseev:2009tg}
\end{remark}
\end{example}


The Dirac structure of central focus in this paper is the following generalization of the Cartan Dirac structure.
\begin{example}[$\Gamma$-twisted Cartan-Dirac structure]\label{ex:twstCartDir}
Suppose $\Gamma$ is a permutation graph,  with edge set $E_\Gamma$, vertex set $V_\Gamma$, and (bijective) incidence maps $$\on{in},\on{out}:E_\Gamma\to V_\Gamma.$$ 

The diagonal subalgebra $\g_\Delta^{V_\Gamma}\subseteq (\ol{\g}\oplus\g)^{V_\Gamma}$ is Lagrangian, and hence so is its image 
$$\g_\Gamma:=(\on{in}\oplus\on{out})^!\big(\g_\Delta^{V_\Gamma}\big)\subseteq (\ol{\g}\oplus\g)^{E_\Gamma}$$
under the isomorphism $(\on{in}\oplus\on{out})^!:(\ol{\g}\oplus\g)^{V_\Gamma}\to (\ol{\g}\oplus\g)^{E_\Gamma}$.

Thus $$\g_\Gamma\times G^{E_\Gamma}\subseteq \big((\ol{\g}\oplus\g)\times G\big)^{E_\Gamma}$$ 
is a Dirac structure, called the $\Gamma$-twisted Cartan Dirac structure.

The following picture can be helpful. We associate a copy of the Courant algebroid $(\ol{\g}\oplus\g)\times G$ to each edge, as in \cref{fig:3Cycla}. The Dirac structure $\g_\Gamma$ acts diagonally at each vertex, as pictured in \cref{fig:3Cyclb}.


\begin{figure}
\begin{center}
\begin{subfigure}[t]{.45\linewidth}
\centering
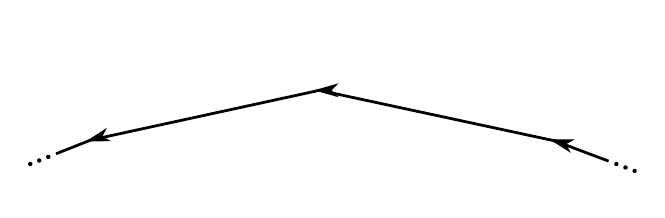
\caption{The $\Gamma$-twisted Cartan Courant algebroid. Here $\on{in}(e)=v=\on{out}(v')$, where $v\in V_\Gamma$ is the vertex and $e,e'\in E_\Gamma$ are edges.}\label{fig:3Cycla}
\end{subfigure}
\hspace{.05\linewidth}
\begin{subfigure}[t]{.45\linewidth}
\centering
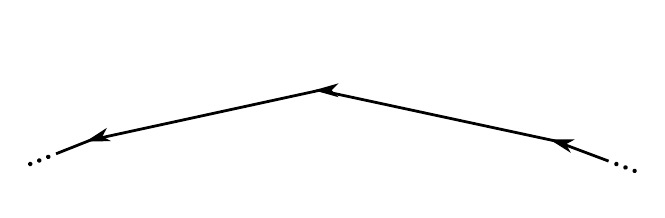
\caption{The element $\xi^{(\cdot)}\in \g^{V_\Gamma}_\Delta \cong  \g_\Gamma$ of the $\Gamma$-twisted Cartan Dirac structure acts diagonally by $\xi^v$ at the vertex $v$. }\label{fig:3Cyclb}
\end{subfigure}
\caption{}
\end{center}

\end{figure}
\begin{remark}
$\Gamma$-twisted Cartan-Dirac structures were discovered independently by Alejandro Cabrera, who provides their construction in terms of Dirac reduction of the Lie-Poisson structure on the dual of the loop Lie algebra, $\Omega^1(S^1,\g)$.
\end{remark}

\begin{remark}\label{rem:DiscreteCartDir}
The following was explained to the authors by Eckhard Meinrenken.
Suppose $\sigma_\Gamma:V_\Gamma\to V_\Gamma$ is the permutation induced by $\Gamma$ (by the discrete flow along the edges of $\Gamma$), and consider the group $G_{\on{big}}:=G^{V_\Gamma}\rtimes \mbb{Z}$,
$$(g,i)\cdot(g',i')=(g(\sigma_\Gamma^i)^!(g'),i+i'),$$
where $\big((\sigma_\Gamma^i)^!(g')\big)_v=(g')_{\sigma_\Gamma^i(v)}$ for any $v\in V_\Gamma$.
Consider the embedding of manifolds $$\big((\on{in}^{-1})^!, -1\big):G^{E_\Gamma}\to G^{V_\Gamma}\rtimes \mbb{Z}=G_{\on{big}}.$$
The Lie algebra of $G_{\on{big}}$ is $\g^{V_\Gamma}$, and for every $\xi\in \g^{V_\Gamma}$, the left and right invariant vector fields $\xi^L,\xi^R\in\mf{X}(G_{\on{big}})$ restrict to the $e\in E_\Gamma$'th factor of $\big((\on{in}^{-1})^!, -1\big)(G^{E_\Gamma})$ as $\xi^R_{\on{in}(e)}$ and $\xi^L_{\on{out}(e)}$, (respectively).
Thus, the $\Gamma$-twisted Cartan-Dirac structure on $G^{E_\Gamma}$ may be canonically identified with the restriction of the Cartan-Dirac structure on $G_{\on{big}}$ to $\big((\on{in}^{-1})^!, -1\big)(G^{E_\Gamma})$.
\end{remark}

\end{example}

\subsection{Courant relations and morphisms of Manin pairs}

\subsubsection{Relations}
A smooth relation $S\colon M_1\dasharrow M_2$ between manifolds is an immersed submanifold 
$S\subseteq M_2\times M_1$. We will write $m_1\sim_S m_2$ if $(m_2,m_1)\in S$. 
Given smooth relations $S\colon M_1\dasharrow M_2$ and 
$S'\colon M_2\dasharrow M_3$, the set-theoretic composition $S'\circ S$ is the image of 
\begin{equation}\label{eq:intersection}
 S'\diamond S=(S'\times S)\cap (M_3\times (M_2)_\Delta \times M_1)
 \end{equation}
under projection to $M_3\times M_1$, where $(M_2)_\Delta\subseteq M_2\times M_2$ denotes the diagonal.

We say that the two relations \emph{compose cleanly} if 
\eqref{eq:intersection} is a clean intersection in the sense of Bott (i.e. it is smooth, and the intersection of the tangent bundles is the tangent bundle of the intersection), and the map from
$S'\diamond S$ to $M_2\times M_1$ has constant rank. In this case, the composition 
$S'\circ S\colon M_1\dasharrow M_3$ is a well-defined smooth relation. See \cite[Appendix A]{LiBland:2011vqa} for more information on the composition of smooth relations. For background on clean intersections of manifolds, see 
e.g.~ \cite[page 490]{Hormander:2007tm}. 

For any relation $S:M_1\dasharrow M_2$, we let $S^\top:M_2\dasharrow M_1$ denote the \emph{transpose} relation, 
$$S^\top=\{(m_1,m_2)\in M_1\times M_2\mid (m_2,m_1)\in S\}.$$

\subsubsection{Courant relations}
As popularized by the second author \cite{Severa:2005vla,LetToWein}, Dirac structures can be interpreted as the `canonical relations' between Courant algebroids:
\begin{definition}[Courant relations and morphisms \cite{Alekseev:2002tn,Severa:2005vla,Bursztyn:2009wi}]
Suppose $\mbb{E}_1\to M_1$ and $\mbb{E}_2\to M_2$ are two Courant algebroids. A relation
$$R:\mbb{E}_1\dasharrow \mbb{E}_2$$
is called a \emph{Courant relation} if $R\subseteq \mbb{E}_2\times\ol{\mbb{E}_1}$ is a Dirac structure supported on a submanifold $S\subseteq M_2\times M_1$. A Courant relation is called \emph{exact} if the underlying Dirac structure is exact.

When $S=\gr(\mu)$ is the graph of a smooth map $\mu:M_1\to M_2$, $R$ is called a \emph{Courant morphism}.

We define the range $\on{ran}(R)\subseteq \mbb{E}_2\rvert_S$ and the kernel $\on{ker}(R)\subseteq\mbb{E}_2\rvert_S$ of $R$ by
\begin{align*}
\on{ran}(R)&:=\{e\in\mbb{E}_2\rvert_S\mid e'\sim_R e \text{ for some } e'\in \mbb{E}_1\}\\
\on{ker}(R)&:=\{e\in\mbb{E}_1\rvert_S\mid e\sim_R 0\}.\\
\end{align*}
\end{definition}
As an example, any Dirac structure $E\subseteq \mbb{E}$ defines a Courant morphism $$E:\mbb{E}\dasharrow \ast$$ to the trivial Courant algebroid (or a Courant relation from the trivial Courant algebroid). Similarly, the diagonal $\mbb{E}_\Delta\subseteq\mbb{E}\times\ol{\mbb{E}}$ defines the Courant morphism 
$$\mbb{E}_\Delta:\mbb{E}\dasharrow\mbb{E}$$
corresponding to the identity map.

The key property of Courant relations is the ability to compose them:
\begin{proposition}[{\!\!\cite[Proposition 1.4]{LiBland:2011vqa}}]\label{prop:CompCourRel}
Suppose $R:\mbb{E}_1\dasharrow\mbb{E}_2$ and $R':\mbb{E}_2\dasharrow\mbb{E}_3$ are two Courant relations which compose cleanly, then their composition,
$$R'\circ R:\mbb{E}_1\dasharrow \mbb{E}_3,$$
is a Courant relation.
\end{proposition}
%

\begin{example}[Standard lift]
Suppose $S:M_1\to M_2$ is a relation, then the \emph{standard lift} of $S$,
$$R_S:=TS\oplus\ann(TS)\subseteq \mbb{T}M_2\times\ol{\mbb{T}M_1},$$ 
defines a Courant relation
$$R_S:\mbb{T}M_1\dasharrow\mbb{T}M_2.$$
\end{example}

\begin{example}[Coisotropic subalgebras]\label{ex:CoisoAlgRed}
Suppose $\mf{d}$ is a Lie algebra equipped with an invariant metric. A subalgebra $\mf{c}\subseteq\mf{d}$ is said to be \emph{coisotropic} if $\mf{c}^\perp\subseteq\mf{c}$. In this case, $\mf{c}^\perp\subseteq\mf{c}$ is an ideal, and the metric on $\mf{d}$ descends to define a metric on 
$$\mf{d}_\mf{c}:=\mf{c}/\mf{c}^\perp.$$
The natural relation 
$$R_\mf{c}:\mf{d}\dasharrow\mf{d}_\mf{c},\quad \xi\sim_{R_\mf{c}} \xi+\mf{c}^\perp\quad\text{ for } \xi\in \mf{c},$$
is a Courant relation, where $(\xi+\mf{c}^\perp)\in\mf{c}/\mf{c}^\perp$ denotes the equivalence class of $\xi\in\mf{c}$.

For any Lagrangian subalgebra $\h\subseteq \mf{d}$, \cref{prop:CompCourRel} implies that $$\h_\mf{c}:=R_\mf{c}\circ\h=(\h\cap\mf c)/(\h\cap\mf c^\perp)$$ is a Lagrangian subalgebra of $\mf{d}_\mf{c}$.
\end{example}

\subsubsection{Morphisms of Manin pairs}
A pair $(\mbb{E},E)$ consisting of a Courant algebroid, $\mbb{E}$, together with a Dirac structure $E\subseteq\mbb{E}$ is known as a \emph{Manin pair} \cite{Alekseev99,Bursztyn:2009wi}.

\begin{definition}[\!\!\cite{Bursztyn:2009wi}]\label{def:MorpManPair}
Suppose $\mbb{E}_1\to M_1$ and $\mbb{E}_2\to M_2$ are two Courant algebroids.
A Courant morphism $$R:\mbb{E}_1\dasharrow \mbb{E}_2,$$ supported on the graph of a map $\mu:M_1\to M_2$, defines a \emph{morphism of Manin pairs},
\begin{equation}\label{eq:MorpMP}R:(\mbb{E}_1,E_1)\dasharrow (\mbb{E}_2,E_2)\end{equation}
if
\begin{enumerate}
\item[m1)] $R\circ E_1\subseteq E_2$, and
\item[m2)] $\on{ker}(R)\cap E_1=0$
\end{enumerate}
Here $\on{ker}(R):=(0\times\ol{\mbb{E}_1})\cap R$.

The morphism of Manin pairs, \labelcref{eq:MorpMP}, is said to be \emph{exact} if the underlying Dirac structure is exact.
 
 Suppose \begin{equation}\label{eq:MorpMP2}R':(\mbb{E}_2,E_2)\dasharrow (\mbb{E}_3,E_3)\end{equation}
 is a second morphism of Manin pairs. Conditions (m1) and (m2) imply that the composition of relations $R'\circ R$ is clean. Moreover, the composition defines a morphism of Manin pairs 
 \begin{equation}\label{eq:MorpMP3}R'\circ R:(\mbb{E}_1,E_1)\dasharrow (\mbb{E}_3,E_3),\end{equation}
 (cf. \cite{Bursztyn:2009wi}).
\end{definition}

\begin{remark}
In \cite{LiBland:2010wi}, a morphism of Manin pairs, \labelcref{eq:MorpMP} was said to be \emph{full} if the map $$\mbf{a}\rvert_R:R\to T\gr(\mu)$$
was a surjection. The concept of exact morphisms of Manin pairs is a stronger, but more natural, condition. 
\end{remark}

If \labelcref{eq:MorpMP} is a morphism of Manin pairs, then there exists map $\rho_R:\mu^*E_2\to E_1$ uniquely determined by the condition
\begin{equation}\label{eq:rhoR}\rho_R(e)\sim_R e,\quad e\in E_2.\end{equation}
The induced map of section $\rho_R:\sect(E_2)\to \sect(E_1)$ is a morphism of Lie algebras. Thus \labelcref{eq:MorpMP} defines an action of $E_2$ on $M_1$ which factors through the action of $E_1$.

\subsection{Quasi-Hamiltonian manifolds}\label{sec:Q-HamMflds}
If $(\mathbb E,E)$ is a Manin pair, a \emph{quasi-Hamiltonian $(\mathbb E,E)$-manifold} in the sense of \cite{Bursztyn:2009wi} is a manifold $M$ together with a morphism of Manin pairs
$$(\mathbb TM,TM)\dasharrow(\mathbb E,E).$$
We shall say that the quasi-Hamiltonian space is \emph{exact} (or \emph{quasi-symplectic}) if the morphism of Manin pairs is exact.
To simplify notation (when $M$ is a complicated expression), we will often denote the Manin pair on the left as $(\mathbb TM,TM)=(\mathbb T,T)M$.

\begin{example}[Poisson and symplectic structures]\label{ex:PoissSympl}
Let $0$ denote the trivial Courant algebroid over a point. Consider a morphism of Manin pairs
\begin{equation}\label{eq:MMPPoisson}R:(\mbb T,T)M\dasharrow (0,0).\end{equation}
In this case, $R\subseteq \mbb{T}M$ is just a Dirac structure with support on all of $M$. Condition (m1) is vacuous, while condition (m2) is equivalent to $R\cap TM=0$. As explained in \cite{Courant:tm,Courant:1990uy}, it follows that 
$$R=\gr(\pi^\sharp):=\{(\pi(\alpha,\cdot)+\alpha)\mid\alpha\in\Omega^1(M)\}\subseteq \mbb{T}M$$
is the graph of a Poisson bivector field $\pi\in\mf{X}^2(M)$. In this way, there is a one-to-one correspondence between morphisms of Manin pairs of the form \labelcref{eq:MMPPoisson} and Poisson structures on $M$ \cite{Bursztyn:2009wi}.

\Cref{eq:MMPPoisson} is an exact morphism of Manin pairs if $\mbf{a}\rvert_R:\gr(\pi^\sharp)\to TM$ is a surjection. Equivalently, 
$$\pi^\sharp:T^*M\to TM$$
 is an isomorphism, or the Poisson structure on $M$ is symplectic. In this way, there is a one-to-one correspondence between exact morphisms of Manin pairs of the form \labelcref{eq:MMPPoisson} and symplectic structures on $M$ \cite{LiBland:2010wi}.
 
 The Poisson and symplectic structures that appear in this paper will all arise in this way.
\end{example}

\begin{example}[$E$-invariant submanifolds]\label{ex:EInvSubMMP}
Let $(\mbb{E},E)$ be a Manin pair over a manifold $N$, and suppose $M\subseteq N$ is an $E$-invariant submanifold, i.e. $\mbf{a}(E\rvert_M)\subseteq TM.$ Then $$\mbf{a}^*\rvert_M:T^*N\rvert_M\to \mbb{E}$$ descends to a map $$\mbf{a}^*\rvert_M: T^*M\cong T^*N/\ann(TM)\to \mbb{E}.$$  
As in \cite[Example~1.6]{LiBland:2011vqa}, define the Courant relation 
$R_{E,M}:\mbb{T}M\dasharrow \mbb{E}$
by
$$\mbf{a}(e)+i^*\alpha\sim_{R_{E,M}} e+\mbf{a}^*\alpha,\quad e\in E,\quad \alpha\in T^*N,$$
where $i:M\to N$ denotes the inclusion.
Then
\begin{equation}\label{eq:REM}R_{E,M}:(\mbb T,T)M\dasharrow (\mbb{E},E)\end{equation}
 is a morphism of Manin pairs.

Moreover, \labelcref{eq:REM} is an \emph{exact} morphism of Manin pairs if and only if $\mbf{a}(E\rvert_M)=TM$ and the Courant algebroid $\mbb{E}\to N$ is exact along $M$.

\begin{remark}
In fact, \cref{eq:REM} is the unique morphism of Manin pairs supported on $\gr(i)$.
To see why, suppose 
$$R:(\mbb T,T)M\dasharrow (\mbb{E},E)$$
is such a morphism of Manin pairs.
 Since $R$ is supported on $\gr(i)$, we must have $$i^*\alpha\sim_{R}\mbf{a}^*\alpha,$$ for any $\alpha\in T^*N$.
On the other hand, as explained in \cite[Proposition 3.3]{Bursztyn:2009wi}, for any $e\in E\rvert_M$ there exists a unique $X\in TM$ such that 
$$X\sim_R e.$$
Since $R$ is supported on $\gr(i)$, we have $i_*X=\mbf{a}(e),$ or $X=\mbf{a}(e)$. Thus $R=R_{E,M}$.
\end{remark}
\end{example}

\section{Quasi-Hamiltonian reduction}

\subsection{Reduction theorems}\label{sec:RedTheorems}
Let $\mf d$ be a quadratic Lie algebra acting on a manifold $N$ so that all the stabilizers are coisotropic, and let $\mf h\subset\mf d$ be a Lagrangian Lie subalgebra. We shall consider the following special case of general quasi-Hamiltonian $(\mathbb E,E)$-manifolds.
\begin{definition}
A \emph{quasi-Hamiltonian $(\mf d,\mf h)\times N$-manifold} is a manifold $M$ together with a morphism of Manin pairs
$$(\mathbb T,T)M\dasharrow (\mf d,\mf h)\times N.$$
It is \emph{exact} (or \emph{quasi-symplectic}) if the morphism is exact.
\end{definition}

\begin{example}
If $\mf d=0$ and $N$ is a point then (as we saw in \cref{ex:PoissSympl}) a quasi-Hamiltonian structure is the same as a Poisson or (in the exact case) symplectic structure. More generally, a quasi-Hamiltonian $(0,0)\times N$-structure on $M$ is equivalent to a Poisson structure on $M$ and to a map $\mu:M\to N$ such that $\mu^*(C^\infty(N))\subset C^\infty(M)$ is in the Poisson centre.

 If $\g$ is a quadratic Lie algebra then an \emph{exact} quasi-Hamiltonian $(\g\oplus\bar\g,\g_\Delta)\times G$-structure on $M$ is equivalent to a quasi-Hamiltonian $G$-structure in the sense of Alekseev, Malkin and Meinrenken. If $\h$ is a Lie algebra then a quasi-Hamiltonian $(\h\ltimes\h^*)\times\h^*$-structure on $M$ is equivalent to a Poisson structure on $M$ together with a moment map $M\to\h^*$ generating an action of $\h$. In the exact case the Poisson structure is symplectic.
\end{example}

In this subsection we present a  reduction procedure, which will be the main tool used in our study of the moduli spaces of flat connections.

\begin{definition}
Let $M$ be a quasi-Hamiltonian $(\mf d,\mf h)\times N$-manifold.
 \emph{Reductive data} $(\mf{c},S)$ consists of a coisotropic Lie subalgebra $\mf{c}\subseteq\mf{d}$ together with a $\mf{c}$-invariant submanifold $S\subseteq N$ such that
\begin{enumerate}
\item[r1)] the $\mf{c}^\perp$-orbits in $S$ form a regular foliation\footnote{By a \emph{regular foliation}, we mean that the leaf space carries the structure of a smooth manifold for which the quotient map is a surjective submersion.} with quotient $q_N:S\to  N_{\mf{c},S}$,
\item[r2)] the graph $\gr(\mu)$, where $\mu$ is the underlying map $M\to N$, intersects $S\times M$ cleanly, 
 and the $\h\cap\mf{c}^\perp$-orbits in $\mu^{-1}(S)$ form a regular foliation with quotient $q_M:\mu^{-1}(S)\to  N_{\mf{c},S}$.
\end{enumerate}

\end{definition}

\begin{theorem}[Quasi-Hamiltonian reduction]\label{thm:PartRed}
 Suppose $(\mf{c},S)$ is reductive data for a morphism of Manin pairs 
 \begin{equation}\label{eq:PartRedMMPin}R:(\mbb T,T)M\dasharrow(\mf{d},\h)\times N.\end{equation}
 Then
\begin{equation}\label{eq:PartRedMMPout}R_{\mf{c},S}:(\mbb{T} , T )M_{\mf{c},S}\dasharrow (\mf{d}_\mf{c},\h_\mf{c})\times N_{\mf{c},S}\end{equation}
is a morphism of Manin pairs, where 
$$M_{\mf{c},S}=\mu^{-1}(S)/{\h\cap\mf c^\perp}$$
$$R_{\mf{c},S}:=R_2\circ R\circ R_1^\top,$$
\begin{align*}
R_2&:=R_\mf{c}\times(\gr(q_N)\circ \gr(i_N)^\top),\\
R_1&:=R_{q_M}\circ R_{i_M}^\top,
\end{align*}
 $i_N:S\to N$ and $i_M:\mu^{-1}(S)\to M$ are the canonical inclusions, and $R_\mf{c}:\mf{d}\dasharrow\mf{d}_\mf{c}$ and $\h_\mf c $ are as in \cref{ex:CoisoAlgRed}.
\end{theorem}
The first two statements in \cref{thm:reductionIntro} follow as consequences of \cref{thm:PartRed}.

\begin{theorem}[Quasi-Hamiltonian reduction in the exact case]\label{thm:ExactPartRed}
If, in the setup of \cref{thm:PartRed}, the following additional assumptions hold:
\begin{itemize}
\item $\mf{c}$ acts transitively on $S$,
\item \labelcref{eq:PartRedMMPin} is an exact morphism of Manin pairs,
\end{itemize}
 then 
 \begin{itemize}
\item \labelcref{eq:PartRedMMPout} is an exact morphism of Manin pairs.
\end{itemize}
\end{theorem}
The third statement in \cref{thm:reductionIntro} follows as a consequence of \cref{thm:ExactPartRed}.

We delay the proof of both these theorems to \cref{app:PrtRedPf}. 

Notice that if $\mf c\subset\mf d$ is Lagrangian then the reduced manifold is Poisson or (in the exact case) symplectic, as $\mf d_{\mf c}=0$.

\subsection{Bivector fields and quasi-Poisson structures}\label{sec:qPoissonRed}
In this section we shall explain  \cref{thm:PartRed} in more traditional terms, using bivector fields.

Suppose that 
$$R:(\mbb T,T)M\dasharrow (\mf{d},\h)\times N$$
is a morphism of Manin pairs over $\mu:M\to N$, and the subspace $\mf{k}\subseteq\mf{d}$ is a Lagrangian complement to $\h\subseteq\mf{d}$: that is, $\mf{d}=\h\oplus \mf{k}$. Then axiom (m1) of \cref{def:MorpManPair} implies that $R$ composes transversely with $ \mf{k}$, while property (m2) implies that $ \mf{k}\circ R\subseteq\mbb{T}M$ is a Lagrangian complement to $TM$. Thus there exists a unique bivector field $\pi_{ \mf{k}}\in\mf{X}^2(M)$ such that
$$ \mf{k}\circ R=\gr(\pi^\sharp):=\{\big(\pi_{ \mf{k}}(\alpha,\cdot)+\alpha\big)\mid\alpha\in T^*M\}.$$
The triple $(\mf{d},\h; \mf{k})$ is called a quasi-Manin triple, and $(M,\pi_{\mf{k}},\rho_R)$ is called a Hamiltonian quasi-Poisson $(\mf{d},\h; \mf{k})$-space with moment map $\mu:M\to N$ (cf. \cite{Bursztyn:2009wi,PonteXu:08}). The bivector field $\pi_{\mf k}$ is called a quasi-Poisson structure.



Suppose that $(\mf{c},S)$ is reductive data for the morphism of Manin pairs $R$. We want to reinterpret \cref{thm:PartRed} using the language of quasi-Poisson geometry. Thus we will be interested in Lagrangian complements to the reduced Lie algebra $\h_\mf{c}$.
\begin{lemma}\label{lem:tautoComp}
Each Lagrangian complement $\mf{k}_\mf{c}'\subseteq\mf{d}_\mf{c}$ to $\h_\mf{c}$ defines an element 
$$\tau\in \wedge^2\big(\h/(\h\cap\mf{c}^\perp)\big).$$
\end{lemma}
\begin{proof}
Suppose $\mf{k}'_\mf{c}\subseteq\mf{d}_\mf{c}:=\mf{c}/\mf{c}^\perp$ is a Lagrangian complement to $\h_\mf{c}$. Let $R_\mf{c}:\mf{d}\dasharrow \mf{d}_\mf{c}$ be the relation described in \cref{ex:CoisoAlgRed}, and define $\mf{k}':=\mf{k}'_\mf{c}\circ R_\mf{c}\subseteq \mf{c}$. 
Now $\mf{k}'$ can be seen as the graph
$$\mf{k}'=\{(\xi+\tau^\sharp(\xi)+\eta)\mid\xi\in \mf{k},\; \langle\xi,\cdot\rangle\rvert_{\h\cap\mf{k}'}=0, \text{ and } \eta\in\h\cap\mf{k}'\}$$
 of a map $$\tau^\sharp:\ann(\h\cap\mf{k}')\cong\big(\h/(\h\cap\mf{k}')\big)^*\to\h/(\h\cap\mf{k}'),$$
where $\ann(\h\cap\mf{k}'):=(\h\cap\mf{k}')^\perp\cap\mf{k}$. Let $\tau\in\h/(\h\cap\mf{k}')\otimes \h/(\h\cap\mf{k}')$ be the element defined by $\tau^\sharp(\xi)=\tau(\xi,\cdot)$, then the fact that $\mf{k}'$ is Lagrangian forces $\tau$ to be skew-symmetric.
Finally, since $\mf{k}'_\mf{c}+\h_\mf{c}=\mf{c}/\mf{c}^\perp$ we have $\h+\mf{k}'=\h+\mf{c}$. Hence $\h\cap\mf{k}'=\h\cap\mf{c}^\perp$. Thus $\tau\in \wedge^2\big(\h/(\h\cap\mf{c}^\perp)\big)$.
\end{proof}

\begin{theorem}\label{thm:BivPartRed}
Suppose that $(\mf{c},S)$ is reductive data for the morphism of Manin pairs
\begin{equation}\label{eq:BivPartRedMMPin}R:(\mbb T,T)M\dasharrow (\mf{d},\h)\times N,\end{equation}
and $(M,\pi_{\mf{k}},\rho_R)$ is the Hamiltonian quasi-Poisson $(\mf{d},\h; \mf{k})$-space corresponding to the Lagrangian complement $\mf{k}\subseteq\mf{d}$  to $\h$.
Let $$R_{\mf{c},S}:(\mbb{T} , T )M_{\mf{c},S}\dasharrow (\mf{d}_\mf{c},\h_\mf{c})\times N_{\mf{c},S}$$ denote the reduced morphism of Manin pairs described in \cref{thm:PartRed}, and let $(M_{\mf{c},S},\pi_{\mf{k}'_\mf{c}},\rho_{R_{\mf{c},S}})$ be the Hamiltonian $(\mf{d}_\mf{c},\h_\mf{c};\mf{k}'_\mf{c})$-quasi Poisson manifold corresponding to a chosen Lagrangian complement $\mf{k}'_\mf{c}$ to $\h_\mf{c}$.

As in \cref{lem:tautoComp}, let $\tau \in \wedge^2\big(\h/(\h\cap\mf{c}^\perp)\big)$ be the element corresponding to $\mf{k}'_\mf{c}$.
Then  
$$(\pi_{\mf{k}}+\rho_R(\tau))\rvert_{\mu^{-1}(S)}$$ is an  $\h\cap\mf{c}^\perp$-invariant section of 
$\wedge^2\big(T(\mu^{-1}(S))/\rho_R(\h\cap\mf{c}^\perp)\big)$ which 
is mapped to $\pi_{\mf{k}'_\mf{c}}$ under the surjective submersion $\mu^{-1}(S)\to M_{\mf{c},S}$.
\end{theorem}
\begin{proof}
Let $R_\mf{c}:\mf{d}\dasharrow \mf{d}_\mf{c}$ be the relation described in \cref{ex:CoisoAlgRed}. Let $\mu:M\to N$ denote the map supporting $R$, and let $$R_{i_M}:\mbb{T}\mu^{-1}(S)\dasharrow \mbb{T}M, \quad R_{q_M}: \mbb{T}\mu^{-1}(S)\dasharrow \mbb{T}M_{\mf{c},S}$$ denote the standard lifts of the inclusion and projection, respectively. Then 
$$\gr(\pi_{\mf{k}'_\mf{c}}^\sharp)=\mf{k}'_\mf{c}\circ R_\mf{c}\circ R\circ R_{i_M}\circ R_{q_M}^\top.$$

Now, as explained in \cref{lem:tautoComp}, 
\begin{subequations}\label[pluralequation]{eq:TauBivRel}
\begin{equation}\mf{k}'_\mf{c}\circ R_\mf{c}=\mf{k}':=\{(\xi+\tau^\sharp(\xi)+\eta)\mid\xi\in \mf{k},\; \langle\xi,\cdot\rangle\rvert_{\h\cap\mf{c}^\perp}=0, \text{ and } \eta\in\h\cap\mf{c}^\perp\}.\end{equation}
Meanwhile, since $\gr(\pi_\mf{k}^\sharp)=\mf{k}\circ R$, it follows that 
\begin{equation}X+\alpha\sim_R \xi+\eta, \quad \xi\in\mf{k},\eta\in\h\Leftrightarrow X=\rho_R(\eta)+\pi_\mf{k}^\sharp\alpha,\text{ and }\xi=j\circ\rho_R^*\alpha,\end{equation}
\end{subequations}
where $j:\h^*\to\mf{k}$ inverts the isomorphism $\mf{k}\to\mf{d}/\h\cong\h^*$. 

Using \cref{eq:TauBivRel} we compute
$$\mf{k}'_\mf{c}\circ R_\mf{c}\circ R=\big\{\big(\rho_R(\eta)+(\pi_\mf{k}+\rho(\tau)\big)^\sharp(\alpha)+\alpha\big)\rvert\eta\in \h\cap\mf{c}^\perp \text{ and } \alpha\in \ann\big(\rho_R(\h\cap\mf{c}^\perp)\big)\big\}.$$
This shows that $\pi_\mf{k}+\rho(\tau)$ is $\h\cap\mf{c}^\perp$-invariant. Moreover,
$\mf{k}'_\mf{c}\circ R_\mf{c}\subseteq\mf{c}$ and thus elements of $\mf{k}'_\mf{c}\circ R_\mf{c}$ act to preserve $S$. In turn, this implies that $$(\pi_\mf{k}+\rho(\tau))\rvert_{\mu^{-1}(S)}\in \Gamma\bigg(\wedge^2 \big(TM\rvert_{\mu^{-1}(S)}/\rho_R(\h\cap\mf{c}^\perp)\big)\bigg)$$ is in fact a section of 
$\wedge^2 (T\mu^{-1}(S)/\rho_R(\h\cap\mf{c}^\perp))$.

Now, $\gr(\pi_{\mf{k}'_\mf{c}}^\sharp)=\mf{k}'_\mf{c}\circ R_\mf{c}\circ R\circ R_{i_M}\circ R_{q_M}^\top$, i.e.
$$(q_M)_*\big(\pi_\mf{k}+\rho(\tau)\big)\rvert_{\mu^{-1}(S)}=\pi_{\mf{k}'_\mf{c}}.$$
\end{proof}

\subsection{Exact morphisms of  Manin pairs and 2-forms}
In this section, we will examine exact morphisms of Manin pairs in more detail. We recall from \cite{Bursztyn:2009wi,LiBland:2010wi} that once isotropic splittings are chosen, these are uniquely determined by a map between the underlying spaces, and a 2-form on the domain. This description in terms of 2-forms can be useful for simplifying calculations.

Suppose $\mbb{E}$ is an exact Courant algebroid over $N$. That is, the sequence
\begin{equation}\label{eq:ExCourExSeq}0\to T^*N\xrightarrow{\mbf{a}^*}\mbb{E}\xrightarrow{\mbf{a}} TN\to 0\end{equation}
 is exact. Let $s:TN\to\mbb{E}$ be a splitting of \cref{eq:ExCourExSeq} such that $s(TN)\subseteq\mbb{E}$ is isotropic (such splittings are called \emph{isotropic splittings}). Then, as explained in \cite{Severa:2001}, the formula
$$\iota_X\iota_Y\iota_Z\gamma:=\la\Cour{s(X),s(Y)},s(Z)\ra,\quad X,Y,Z\in\mf{X}(N)$$
defines a closed 3-form, $\gamma\in\Omega^3(N)$, called the \emph{curvature 3-form of the splitting $s$}. The isomorphism $s\oplus\mbf{a}^*:TN\oplus T^*N\xrightarrow{\cong} \mbb{E}$ 
identifies the metric on $\mbb{E}$ with $$\la X+\alpha,Y+\beta\ra=\alpha(Y)+\beta(X),\quad X,Y\in TN, \quad \alpha,\beta\in T^*N,$$
and the bracket with
\begin{equation}\label{eq:ExCourBrk}\Cour{X+\alpha,Y+\beta}=[X,Y]+\Lied_X\beta-\iota_Yd\alpha+\iota_X\iota_Y\gamma,\quad X,Y\in\mf{X}(N),\quad \alpha,\beta.\end{equation}
The Courant algebroid with underlying bundle $TN\oplus T^*N$, and bracket given by \cref{eq:ExCourBrk} is called the \emph{$\gamma$-twisted exact Courant algebroid} over $N$, and denoted by $\mbb{T}_\gamma N$.

\begin{example}[The Cartan Courant algebroid \cite{Severa:2001}]\label{ex:CartCourSplit}
The Cartan Courant algebroid $$(\ol{\g}\oplus\g)\times G$$ reviewed in \cref{ex:CartCourAlg} is exact. 
In  \cite{Alekseev:2009tg} it is shown that the map
$s:TG\to(\ol{\g}\oplus\g)\times G$ defined as
$$s:X\to \frac{1}{2}\big(-\iota_X(\d g\:g^{-1}),\iota_X(g^{-1}\d g)\big), \quad X\in TG$$
is a $(\ol{\g}\oplus\g)$-invariant isotropic splitting of the Cartan Courant algebroid. The corresponding curvature 3-form is computed as
$\gamma=\frac{1}{24}\la[g^{-1}\d g,g^{-1}\d g],g^{-1}\d g\ra$
(note that the normalization differs from that in \cite{Alekseev:2009tg}).
\end{example}

Suppose $L\subseteq\mbb{E}$ is a Lagrangian subbundle with support on $S\subseteq N$, such that $\mbf{a}(L)=TS$, then there is a  2-form $\omega\in \Omega^2(S)$ uniquely determined by the formula
$$\iota_X\iota_Y\omega=\la s(X),e\ra,\quad X,Y\in TS,\quad e\in L\text{ and }\mbf{a}(e)=Y.$$
Thus, we may identify $L$ with 
$$L=\gr(\omega^\flat):=\{\big(s(X)+\mbf{a}^*(\iota_X\omega+\alpha)\big)\mid X\in TS, \alpha\in \ann(TS)\}\subseteq\mbb{E}.$$
A quick calculation using \cref{eq:ExCourBrk} shows that $L$ is a Dirac structure with support on $S$ if and only if $\d\omega=i^*\gamma$, where $i:S\to N$ is the inclusion (cf. \cite{Severa:2001} and \cite[Proposition 2.8]{Bursztyn:2009wi}). 

Suppose
$$R:\mbb{T}M\dasharrow \mbb{E}$$
 is a Courant morphism supported on the graph of a map $\mu:M\to N$. Since $\mbb{E}$ is an exact Courant algebroid, $R$ is \emph{exact} if $\mbf{a}(R)=T\gr(\mu)$ (cf. \cref{lem:ExDirStr}), in which case
the considerations above show that $$R=\gr(\omega^\flat)\subseteq \mbb{E}\times \overline{\mbb{T}M},$$
where $\omega\in\Omega^2(M)\cong\Omega^2(\gr(\mu))$ satisfies $\d\omega=\mu^*\gamma$  \cite{Bursztyn:2009wi,LiBland:2010wi}.
That is, $R=R_{\mu,\omega}$, where $R_{\mu,\omega}$ is defined by
\begin{equation}\label{eq:RphiomeDef}X-\iota_X\omega+\mu^*\alpha\sim_{R_{\mu,\omega}} s(\mu_*X)+\mbf{a}^*\alpha,\quad X\in TM, \quad\alpha\in T^*N.\end{equation}

In particular, if $E\subseteq\mbb{E}$ is a Dirac structure, once an isotropic splitting $s:TN\to \mbb{E}$ is chosen, a morphism of Manin pairs
$$R:(\mbb T,T)M\dasharrow (\mbb{E},E)$$
is entirely determined by the underlying map $\mu:M\to N$ between the spaces, and a 2-form $\omega\in \Omega^2(M)$ (satisfying certain conditions).

\begin{remark}[Twisted quasi-Hamiltonian structures]\label{rem:twistedQHam}Recall the $\Gamma$-twisted Cartan Dirac structure described in \cref{ex:twstCartDir},
$$\big((\bar{\g}\oplus\g)^{E_\Gamma},\g_\Gamma\big)\times G^{E_\Gamma}.$$
Suppose we use the splitting $TG^{E_\Gamma}\to \big((\bar{\g}\oplus\g)\times G\big)^{E_\Gamma}$ described in \cref{ex:CartCourSplit} to identify $$\big((\bar{\g}\oplus\g)\times G\big)^{E_\Gamma}\cong \mbb{T}_\gamma G^{E_\Gamma},$$ where 
$\gamma=\frac{1}{24}\la[g^{-1} \d g,g^{-1}\d g],g^{-1}\d g\ra$.
Then exact morphisms of Manin pairs $$R:(\mbb T,T)M\dasharrow \big((\bar{\g}\oplus\g)^{E_\Gamma},\g_\Gamma\big)\times G^{E_\Gamma}$$ are in one-to-one correspondence with quadruples
 $(M,\mu,\rho,\omega)$, where
 \begin{itemize}
 \item $\mu:M\to G^{E_\Gamma}$ is a smooth map,
 \item $\rho:\g^{V_\Gamma}\to \mf{X}(M)$ is a Lie algebra action, and
 \item $\omega\in \Omega^2(M)$ is a 2-form,
 \end{itemize}
such that
\begin{enumerate}
\item $\d\omega=\mu^*\gamma$,
\item $\mu:M\to G^{E_\Gamma}$ is $\g^{V_\Gamma}$-equivariant with respect to the $\g^{V_\Gamma}$ action on $G^{E_\Gamma}$ given on the $e\in{E_\Gamma}$-th factor by
$$\xi\to -\xi_{\on{in}(e)}^R+\xi_{\on{out}(e)}^L,\quad \xi\in \g^{V_\Gamma},$$
\item $\on{ker}(\d\mu)_x\cap\on{ker}(\omega^\flat)_x=0$, for every $x\in M$, and
\item $$\iota_{\rho(\xi)}\omega=\frac{1}{2}\mu^*\sum_{e\in E_\Gamma}\la g_e^{-1}\d g_e,\xi_{\on{out}(e)}\ra+\la\d g_e\:g_e^{-1},\xi_{\on{in}(e)}\ra.$$
\end{enumerate}
(Note that conditions (2),(3) and (4) determine $\rho$ uniquely in terms of $\mu$ and $\omega$). The quadruple $(M,\mu,\rho,\omega)$ corresponds to the morphism of Manin pairs \cref{eq:RphiomeDef}, supported on the graph of $\mu$.

The quadruples $(M,\mu,\rho,\omega)$ generalize quasi-Hamiltonian $G^{E_\Gamma}$-structures in the sense that they incorporate an automorphism of $G^{E_\Gamma}$ into their definition. Here the automorphism is simply the permutation of factors described by the permutation graph, $\Gamma$. Allowing arbitrary automorphisms leads to a definition of twisted quasi-Hamiltonian spaces along the lines of the one given in  \cite{LiBland:2012vo} for twisted quasi-Poisson structures (cf. \cite[Definition 3]{LiBland:2012vo}).

Let $G_{\on{big}}:=G^{V_\Gamma}\rtimes \mbb{Z}$, where $\mbb{Z}$ acts by permuting the factors according to the graph $\Gamma$ (cf. \cref{rem:DiscreteCartDir}).
As explained to the authors by Eckhard Meinrenken, quasi-Hamiltonian $\big((\bar{\g}\oplus\g)^{E_\Gamma},\g_\Gamma\big)\times G^{E_\Gamma}$-structures are quasi-Hamiltonian $G_{big}$-structures in the (original) sense of Alekseev-Malkin-Meinrenken \cite{Alekseev97} for which the moment map takes values in $G^{V_\Gamma}\times \{-1\}\subset G_{big}$  (cf. \cref{rem:DiscreteCartDir}).
\end{remark}

We now examine the behaviour of twisted exact Courant algebroids under the partial reduction procedure described in \cref{thm:PartRed}.
Suppose that $\mf{d}\times N$ is an exact Courant algebroid, and a $\mf{d}$-invariant isotropic splitting $s:TN\to \mf{d}\times N$ is chosen, defining an isomorphism 
$$\mf{d}\times N\cong \mbb{T}_\gamma N,$$
We let $E\subseteq\mbb{T}_\gamma N$ denote the Dirac structure corresponding to $\h\times N\subseteq\mf{d}\times N$ under this isomorphism.
Suppose 
$$R_{\mu,\omega}:(\mbb T,T)M\dasharrow(\mbb{T}_\gamma N,E)\cong (\mf{d},\h)\times N$$
is a exact morphism of Manin pairs, and $S\subseteq N$ is an orbit of the coisotropic subalgebra $\mf{c}\subseteq\mf{d}$,
and the assumptions of both \cref{thm:PartRed,thm:ExactPartRed} hold for the reductive data $(\mf{c},S)$.
Let $i_N:S\to N$ and $i_M:\mu^{-1}(S)\to M$ denote the inclusions, and $q_N:S\to N_{\mf{c},S}$ and $q_M:\mu^{-1}(S)\to M_{\mf{c},S}$ the quotients (by $\mf{c}^\perp$ and $\h\cap\mf{c}^\perp$, respectively). 



We would like to define a splitting of the reduced Courant algebroid $\mf{d}_{\mf{c}^\perp}\times N_{\mf{c},S}$. As explained in \cite[Proposition 3.6]{Bursztyn:2007ko}, unless $s(TS)\subseteq \mf{c}$, this will depend on a choice of a $\mf{c}^\perp$-invariant connection 1-form $\theta\in\Omega^1(S,\mf{c}^\perp)$ for the bundle $q_N:S\to N_{\mf{c},S}$. For any $X\in\mf{X}(N_{\mf{c},S})$, let $X^h\in\mf{X}(S)$ denote its horizontal lift with respect to the chosen connection. The map $X\to s(X^h)\in\Gamma(\mf{d}\times S)$ may not take values in $\Gamma(\mf{c}\times S)$, but 
$$s_\theta:X\to s(X^h)+\mbf{a}^*(\iota_{X^h}\la\theta,\vartheta_s\ra)$$
 does, where $\vartheta_s\in\Omega^1\big(S,(\mf{c}^\perp)^*\big)$ is defined by $\la\xi,\vartheta_s\ra:=s^*\xi,$ for $\xi\in\mf{c}^\perp.$ Note also that $s_\theta(X)$ is $\mf{c}^\perp$-invariant, and hence descends to a unique section of $\Gamma(\mf{c}\times N_{\mf{c},S})$.
Thus, the composition 
$$\mf{X}(N_{\mf{c},S})\xrightarrow{s_\theta} \Gamma(\mf{c}\times S)\to \Gamma(\mf{c}/\mf{c}^\perp\times N_{\mf{c},S})=\Gamma(\mf{d}_\mf{c}\times N_{\mf{c},S})$$
 defines an isotropic splitting $\tilde s_\theta:TN_{\mf{c},S}\to\mf{d}_\mf{c}\times N_{\mf{c},S}$.
We define $\tilde\gamma_\theta\in\Omega^3(N_{\mf{c},S})$ to be the associated curvature 3-form.
Let $E_{\mf{c},S}\subseteq\mbb{T}_{\tilde\gamma_\theta} N_{\mf{c},S}$ denote the Dirac structure corresponding to $\h_{\mf{c},S}$ under the isomorphism defined by $\tilde s_\theta$. 

\begin{proposition}[Partial reduction for split exact Courant algebroids]\label{thm:PartRedSplEx}
Suppose that
\begin{equation}\label{eq:SpExPartRedMMPin}R_{\mu,\omega}:(\mbb T,T)M\dasharrow(\mbb{T}_\gamma N,E)\cong (\mf{d},\h)\times N\end{equation}
is an exact morphism of Manin pairs, that the assumptions of both \cref{thm:PartRed,thm:ExactPartRed} hold for the reductive data $(\mf{c},S)$. Let $\theta\in\Omega^1(S,\mf{c}^\perp)$ be a $\mf{c}^\perp$-invariant connection 1-form for the bundle $q_N:S\to N_{\mf{c},S}$.

Then, under the isomorphism $\mf{d}_\mf{c} \times N_{\mf{c},S}\cong\mbb{T}_{\tilde\gamma_\theta} N_{\mf{c},S}$ defined by the isotropic splitting $\tilde s_\theta$, the reduced morphism of Manin pairs \labelcref{eq:PartRedMMPout}, described in \cref{thm:PartRed}, is identified with
\begin{equation}\label{eq:SpExPartRedMMPout}R_{\tilde\mu,\tilde\omega_\theta}:(\mbb{T},T)M_{\mf{c},S}\dasharrow (\mbb{T}_{\tilde\gamma_\theta}N_{\mf{c},S},E_{\mf{c},S}),\end{equation}
where $\tilde\omega_\theta\in  \Omega^2(M_{\mf{c},S})$ is defined by the equation 
\begin{equation}\label{eq:omegatheta}q_M^*\tilde\omega_\theta=i^*_M\omega-\mu^*\la\theta,\vartheta_s-\frac{1}{2}s\circ\mbf{a}(\theta)\ra,\end{equation}
 and 
$\tilde\mu:M_{\mf{c},S}\to N_{\mf{c},S}$ is the unique map such that 
$$\begin{tikzpicture}
\mmat{m}{\mu^{-1}(S)&S\\M_{\mf{c},S}&N_{\mf{c},S}\\};
\path[->] (m-1-1) edge node{$\mu$} (m-1-2);
\path[->] (m-1-1) edge node[swap]{$q_M$} (m-2-1);
\path[->] (m-2-1) edge node{$\tilde \mu$} (m-2-2);
\path[->] (m-1-2) edge node{$q_N$} (m-2-2);
\end{tikzpicture}$$
commutes.

Moreover, the 2-form $\tilde\omega_\theta$ is independent of $\theta$ if $s(TS)\subseteq \mf{c}$, more precisely, the term $\mu^*\la\theta,\vartheta_s-\frac{1}{2}s\circ\mbf{a}(\theta)\ra$ in \cref{eq:omegatheta} vanishes. 
\end{proposition}
We defer the proof of \cref{thm:PartRedSplEx} to \cref{app:PrtRedPf}

\begin{remark}\label{rem:cLag2Form}
In the special case where $\mf{c}=\mf{c}^\perp$ is Lagrangian, then $\mf{c}^\perp$ acts transitively on $S$, so \cref{eq:omegatheta} simplifies to 
$$q_M^*\tilde\omega_\theta=i^*_M\omega-\frac{1}{2}\mu^*\la\theta,s\circ\mbf{a}(\theta)\ra.$$
\end{remark}
\begin{remark}\label{rem:SymmetricOrbits}
Consider the $\Gamma$-twisted Cartan Dirac structure 
$\big((\ol{\g}\oplus\g)^{E_\Gamma},\g_\Gamma)\times G^{E_\Gamma}$.
The conditions that $\mf{c}\subseteq(\ol{\g}\oplus\g)^{E_\Gamma}$ be symmetric and that $S\subseteq G^{E_\Gamma}$ be the $\mf{c}$-orbit through the identity imply that 
\begin{equation}\label{eq:sInL}s(TS)\subseteq \mf{c}\end{equation}
 Indeed, \cref{eq:sInL} is easily checked at the identity of $G^{E_\Gamma}$, and the invariance of $s$ implies that it holds at every other point in the $\mf{c}$-orbit, $S$.

Thus, \cref{eq:omegatheta} simplifies to 
$$q_M^*\tilde\omega_\theta=i^*_M\omega,$$
and \cref{thm:MainThmIntro} follows as a corollary to \cref{thm:PartRedSplEx}.
\end{remark}

\begin{example}[Conjugacy clases and quasi-Hamiltonian geometry]
Recall the splitting $s:TG\to (\ol{\g}\oplus\g)\times G$ of the Cartan Courant algebroid described in \cref{ex:CartCourSplit}.
Suppose $\mf{c}=\g_\Delta\subseteq (\ol{\g}\oplus\g)$ is the diagonal subalgebra, and $S\subseteq G$ is a conjugacy class. Then,
$$s\circ\mbf{a}(\xi,\xi)=(\xi-\Ad_g\xi,-\Ad_{g^{-1}}\xi+\xi), \quad\xi\in\g,\quad g\in G.$$
So $$\iota_{\mbf{a}(\eta,\eta)}\iota_{\mbf{a}(\xi,\xi)}\frac{1}{2}\mu^*\la\theta,s\circ\mbf{a}(\theta)\ra=\la\eta,\Ad_g\xi\ra-\la\Ad_g\eta,\xi\ra.$$ 
Thus $\frac{1}{2}\la\theta,s\circ\mbf{a}(\theta)\ra$
is proportional to the quasi-Hamiltonian 2-form on the conjugacy class, $S$, described in \cite[Proposition 3.1]{Alekseev97}.

Suppose now that $M$ is a quasi-Hamiltonian $G$ space with moment map $\mu:M\to G$, and 2-form $\omega$. 
It follows that partial reduction of the morphism of Manin pairs
$$R_{\mu,\omega}:(\mbb T,T)M\dasharrow (\ol{\g}\oplus\g,\g_\Delta)\times G$$
by $(\mf{c},S)$, yields the same symplectic structure as the quasi-Hamiltonian reduction
$$M\circledast S/\!/_{1} G,$$
where $M\circledast S$ denotes the fusion of $M$ with $S$, as described in \cite{Alekseev97}.
\end{example}

\subsection{Commutativity of reductions}
Suppose that $S_1,S_2\subseteq N$ intersect cleanly, and 
that $(\mf{c}_1,S_1)$ and $(\mf{c}_2,S_2)$ are both  reductive data for a morphism of Manin pairs 
 \begin{equation}\label{eq:PartRedMMPinCom}R:(\mbb T,T)M\dasharrow(\mf{d},\h)\times N.\end{equation}

Notice that the Lie algebra $\mf{c}_{2,1}:=R_{\mf{c}_1}\circ\mf{c}_2\subseteq \mf{d}_{\mf{c}_1}$ is coisotropic. 
Let $S_{2,1}$ denote the image of $S_2\cap S_1$ under the quotient map $S_1\to N_{\mf{c}_1,S_1}$. In practise, $(\mf{c}_{2,1},S_{2,1})$ will often form reductive data for $$R_{\mf{c}_1,S_1}:(\mbb{T} , T )M_{\mf{c}_1,S_1}\dasharrow (\mf{d}_{\mf{c}_1},\h_{\mf{c}_1})\times N_{\mf{c}_1,S_1}.$$

Similarly, with $\mf{c}_{1,2}:=R_{\mf{c}_2}\circ\mf{c}_1\subseteq \mf{d}_{\mf{c}_2}$ and $S_{1,2}$ the image of $S_1\cap S_2$ under the quotient map $S_2\to N_{\mf{c}_2,S_2}$, the pair  $(\mf{c}_{1,2},S_{1,2})$ will often form reductive data for
$$R_{\mf{c}_2,S_2}:(\mbb{T} , T )M_{\mf{c}_2,S_2}\dasharrow (\mf{d}_{\mf{c}_2},\h_{\mf{c}_2})\times N_{\mf{c}_2,S_2}.$$


\begin{proposition}[Commutativity of reductions]\label{prop:ComPartRed}
Suppose that $S_1,S_2\subseteq N$ intersect cleanly,  
that $(\mf{c}_1,S_1)$, $(\mf{c}_2,S_2)$ and $(\mf{c}_1\cap\mf{c}_2,S_1\cap S_2)$ all form reductive data for the morphism of Manin pairs 
 \begin{equation*}R:(\mbb T,T)M\dasharrow(\mf{d},\h)\times N,\end{equation*}
while  $(\mf{c}_{2,1},S_{2,1})$ and $(\mf{c}_{1,2},S_{1,2})$ form reductive data for $$R_{\mf{c}_1,S_1}:(\mbb{T} , T )M_{\mf{c}_1,S_1}\dasharrow (\mf{d}_{\mf{c}_1},\h_{\mf{c}_1})\times N_{\mf{c}_1,S_1}$$
and 
$$R_{\mf{c}_2,S_2}:(\mbb{T} , T )M_{\mf{c}_2,S_2}\dasharrow (\mf{d}_{\mf{c}_2},\h_{\mf{c}_2})\times N_{\mf{c}_2,S_2},$$
respectively.
 
 Then
 \begin{equation}\label{eq:ComRed}(R_{\mf{c}_1,S_1})_{\mf{c}_{2\!,\!1},S_{2\!,\!1}}=R_{\mf{c}_1\cap\mf{c}_2,S_1\cap S_2}=(R_{\mf{c}_2,S_2})_{\mf{c}_{1\!,\!2},S_{1\!,\!2}}.\end{equation}
 
\end{proposition}
The proof of this proposition is deferred to \cref{app:ComParRedPrf}.

In this paper, \cref{prop:ComPartRed} will always be applied as the following corollary:
\begin{corollary}[Reductions of distinct factors commute]\label{cor:ComPartRed}
Suppose $\mf{d}\times N,\mf{d}'\times N'$, and $\mf{d}''\times N''$ are all action Courant algebroids, $S\subseteq N$ and $S'\subseteq N'$ are submanifolds and 
$(\mf{c}\oplus\mf{d}'\oplus\mf{d}'',S\times N'\times N'')$, $(\mf{d}\oplus\mf{c}'\oplus\mf{d}'',N\times S'\times N'')$ and $(\mf{c}\oplus\mf{c}'\oplus\mf{d}'',S\times S'\times N'')$ each form reductive data for a morphism of Manin pairs
$$R:(\mbb T,T)M\dasharrow(\mf{d}\oplus\mf{d}'\oplus\mf{d}'',\h)\times (N\times N'\times N'').$$
Then 
\begin{multline*}(R_{\mf{c}\oplus\mf{d}'\oplus\mf{d}'',S\times N'\times N''})_{(\mf{d}_\mf{c}\oplus\mf{c}'\oplus\mf{d}'',N_{\mf{c},S}\times S'\times N'')}=R_{\mf{c}\oplus\mf{c}'\oplus\mf{d}'',S\times S'\times N''}\\=(R_{\mf{d}\oplus\mf{c}'\oplus\mf{d}'',N\times S'\times N''})_{(\mf{c}\oplus\mf{d}'_{\mf{c}'}\oplus\mf{d}'',S\times N'_{\mf{c}'\!,S'}\times N'')}.\end{multline*}
\end{corollary}

\section{Quasi-Hamiltonian structures on  moduli spaces of flat connections}\label{sec:SymplStrMMP}

Suppose that $(\Sigma,V)$ is a marked surface, i.e.\ $\Sigma$ is a compact oriented surface and $V\subset\partial\Sigma$ a finite set which intersects each component of $\Sigma$ and $\partial\Sigma$ non-trivially. Let $\Gamma$ be the boundary graph of $\Sigma$ with the vertex set $V$ (see Section \ref{sec:qHamIntro} for details). In this section we shall prove \cref{thm:qhamIntro} using quasi-Hamiltonian reduction. In other words, we want to construct an exact morphism of Manin pairs
\begin{equation}\label{eq:Rsigma}
R_{\Sigma,V}:(\mathbb T,T)\mathcal M_{\Sigma,V}\dasharrow  \big((\ol{\g}\oplus\g)^{E_{\Gamma}},\g_{\Gamma}\big)\times G^{E_{\Gamma}}
\end{equation}
over the map
$$\mu:\mathcal M_{\Sigma,V}\to G^{E_{\Gamma}}$$
given by the boundary holonomies.

Let us start with a simple case:
\begin{proposition}[union of polygons]\label{prop:unPolyg}
Let $\Sigma$ be a disjoint union of discs and $V\subset\partial\Sigma$ a finite subset meeting every boundary circle. Then there is a unique exact morphism
$$(\mathbb T,T)\mathcal M_{\Sigma,V}\dasharrow  \big((\ol{\g}\oplus\g)^{E_{\Gamma}},\g_{\Gamma}\big)\times G^{E_{\Gamma}}$$
over $\mu$.
\end{proposition}
\begin{proof}
The map $\mu$ is in this case an embedding. As explained in \cref{ex:EInvSubMMP}, this implies that there exists a unique exact morphism of Manin pairs
\end{proof}

\subsection{Sewing construction}\label{sec:Sew}
\begin{figure}
\begin{center}
\begin{subfigure}[t]{.3\linewidth}
\centering
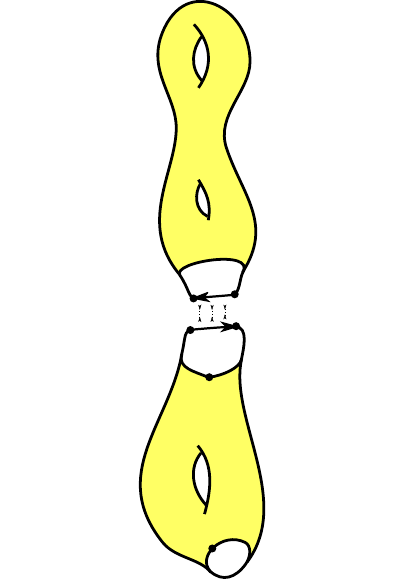
\caption{The initial surface.}
\label{fig:Sew1}
\end{subfigure}
\begin{subfigure}[t]{.3\linewidth}
\centering
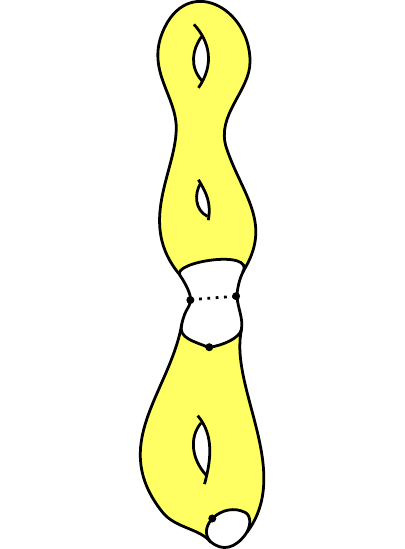
\caption{Sewing.}
\label{fig:Sew2}
\end{subfigure}
\begin{subfigure}[t]{.3\linewidth}
\centering
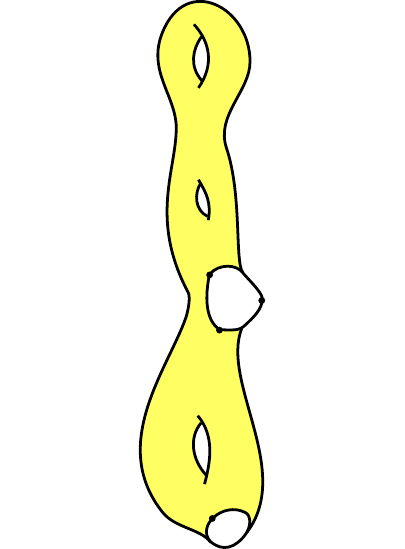
\caption{The resulting surface.}
\label{fig:Sew3}
\end{subfigure}
\end{center}
\caption{\label{fig:Sew}Sewing two edges from $\Sigma$. The two edges must have opposite orientation to ensure that the resulting surface is orientable.}
\end{figure}

Suppose $\Sigma$ is a (possibly disconnected) marked surface, and $e_1,e_2\in \Gamma$ are two distinct edges from the boundary graph, we may `sew' the  surface together along $e_1$ and $e_2$ to form a new surface 
$$\Sigma':=\frac{\Sigma}{e_1\sim e_2},$$
as pictured in \cref{fig:Sew}. In this section, we describe the analogous procedure for the corresponding morphisms of Manin pairs, \cref{eq:Rsigma}.


First, note that 
$$\mf{l}_{sew}:=\{\big((\xi,\eta);(\eta,\xi)\big)\mid\xi,\eta\in\g\}\subseteq (\ol{\g}\oplus\g)\bigoplus(\ol{\g}\oplus\g)$$
is a Lagrangian subalgebra. Since $G$ is connected, the $\mf{l}_{sew}$-orbit through the identity of $G\times G$ is $$G_\Delta^\natural:=\{(g,g^{-1})\mid g\in G\}.$$
\begin{definition}\label{def:Sewing}
Suppose that 
 $$\big((\ol{\g}\oplus\g)\bigoplus(\ol{\g}\oplus\g)\bigoplus\mf{d}'\big)\times \big(G\times G\times N\big)$$
 is the product of two Cartan Courant algebroids with an action Courant algebroid, $\mf{d}'\times N$, and that $\h\subseteq(\ol{\g}\oplus\g)\bigoplus(\ol{\g}\oplus\g)\bigoplus\mf{d}'$ is a Lagrangian subalgebra.
Further suppose that
\begin{equation}\label{eq:SewingIn}R:(\mbb T,T)M\dasharrow \big((\ol{\g}\oplus\g)\bigoplus(\ol{\g}\oplus\g)\bigoplus\mf{d}',\h\big)\times \big(G\times G\times N\big)\end{equation}
is a morphism of Manin pairs.

Let $\mf{c}_{sew}=\mf{l}_{sew}\oplus\mf{d}'$ and let 
$$S_{sew}=G_\Delta^\natural\times N\subseteq G\times G\times N.$$
If $(\mf{c}_{sew},S_{sew})$ is reduction data for \labelcref{eq:SewingIn}, then the  reduction,
$$R_{\mf{c}_{sew},S_{sew}}:(\mbb{T},T)M_{\mf{c}_{sew},S_{sew}}\dasharrow (\mf{d}',\h_{\mf{c}_{sew},S_{sew}})\times N$$
is called the \emph{sewing} of \labelcref{eq:SewingIn}.
\end{definition}

\begin{figure}
\begin{center}
\begin{subfigure}[t]{.3\linewidth}
\centering
 \def\svgwidth{1.5in}
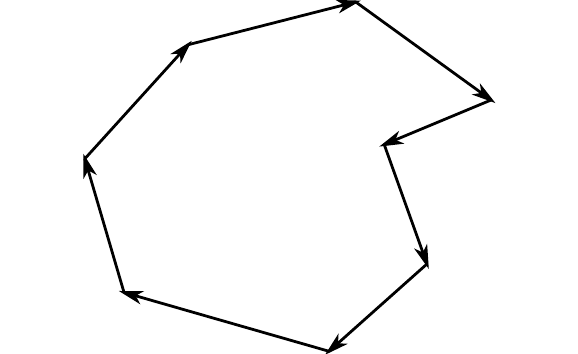
\caption{The graph $\Gamma'$, with chosen edges $e_1,e_2\in E_{\Gamma'}$.}
\label{fig:GraphEdId1}
\end{subfigure}
\begin{subfigure}[t]{.3\linewidth}
\centering
 \def\svgwidth{1.5in}
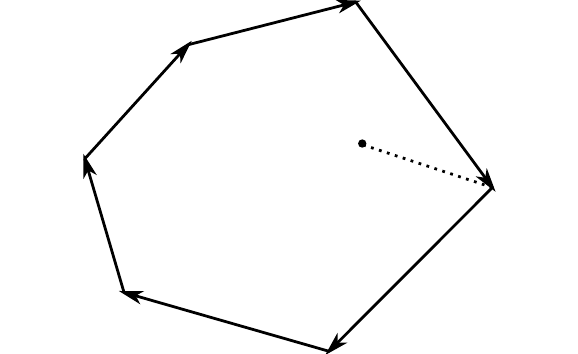
\caption{The graph with edges $e_1$ and $e_2$ identified with opposite orientation.}
\label{fig:GraphEdId2}
\end{subfigure}
\begin{subfigure}[t]{.3\linewidth}
\centering
 \def\svgwidth{1.5in}
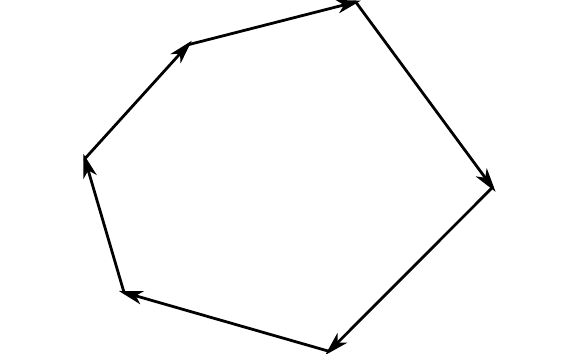
\caption{$\Gamma$ is the graph $\Gamma'$ with edges $e_1$ and $e_2$ identified, and the identified edges removed along with any isolated vertices.}
\label{fig:GraphEdId3}
\end{subfigure}
\end{center}
\caption{}
\end{figure}

In a typical example of sewing, $\Gamma'$ will be a permutation graph,  $e_1,e_2\subseteq E_{\Gamma'}$ will be two (distinct) edges, and 
$$R:(\mbb T,T)M\dasharrow \big((\ol{\g}\oplus\g)^{E_{\Gamma'}},\g_{\Gamma'}\big)\times G^{E_{\Gamma'}}$$
will be a morphism of Manin pairs. 


We let $\Gamma$ denote the graph with edge set $E_{\Gamma}:=E_{\Gamma'}\setminus \{e_1,e_2\}$ and vertex set 
$$V_{\Gamma}:=\frac{\on{out}(E_{\Gamma})}{\on{in}(e_1)\sim\on{out}(e_2)\text{ and }\on{in}(e_2)\sim\on{out}(e_1)},$$
as pictured in \cref{fig:GraphEdId3}.

Let \begin{align*}\mf{c}_{sew}^{e_1, e_2}&:=\{(\xi,\eta)\in(\ol{\g}\oplus\g)^{E_{\Gamma'}}\mid (\xi,\eta)_{\{e_1,e_2\}}\in\mf{l}_{sew}\}\\
S_{sew}^{e_1,e_2}&:=\{g\in G^{E_{\Gamma'}}\mid g_{\{e_1,e_2\}}\in G_\Delta^\natural\}
\end{align*}
as in \cref{def:Sewing}. Here we are using the notation $(\xi,\eta)_{\{e_1,e_2\}}:=\big((\xi_{e_1},\eta_{e_1});(\xi_{e_2},\eta_{e_2})\big),$ as described in \cref{sec:Notation}.
The morphism of Manin pairs
$$R_{\mf{c}_{sew}^{e_1,e_2},S_{sew}^{e_1,e_2}}:(\mbb{T},T)M_{\mf{c}_{sew}^{e_1,e_2},S_{sew}^{e_1,e_2}}\dasharrow ((\ol{\g}\oplus\g)^{E_{\Gamma}},\g_{\Gamma})\times G^{E_{\Gamma}}$$
is called the result of \emph{sewing edges $e_1$ and $e_2$ together}. Here we have used the following lemma to simplify the right hand side:

\begin{lemma}\label{lem:CorectDirStr}
The Lie subalgebras $\g_{\Gamma}\subseteq (\ol{\g}\oplus\g)^{E_{\Gamma}}$ and $(\g_{\Gamma'})_{\mf{c}_{sew}^{e_1,e_2}}\subseteq (\ol{\g}\oplus\g)^{E_{\Gamma'}}$ are equal.
\end{lemma}
\begin{proof}
Let us recall that
$$\g_{\Gamma'}=(\on{in}\oplus\on{out})^!\g_\Delta^{V_{\Gamma'}}\cong \g^{V_{\Gamma'}}$$
and
$$(\g_{\Gamma'})_{\mf{c}_{sew}^{e_1,e_2}}=\g_{\Gamma'}\cap\mf c_{sew}^{e_1,e_2}/\g_{\Gamma'}\cap(\mf c_{sew}^{e_1,e_2})^\perp.$$
By definition of $\mf{c}_{sew}^{e_1,e_2}$,   $\g^{V_{\Gamma'}}\cap\mf c_{sew}^{e_1,e_2}$ is the subalgebra of $\g^{V_{\Gamma'}}$ where the components corresponding to identified vertices are equal. After we divide by $\g_{\Gamma'}\cap(\mf c_{sew}^{e_1,e_2})^\perp$ we obtain the Lie algebra $\g_{\Gamma}=\g^{V_{\Gamma}}$.
\end{proof}

%

The morphism of Manin pairs \labelcref{eq:Rsigma} which we assign to the surface $\Sigma$ will satisfy the following sewing property:

\emph{Sewing property.}
Let $(\Sigma,V)$ be obtained out of $(\Sigma',V')$ by sewing edges $e_1$ and $e_2$.
There is a canonical isomorphism 
\begin{equation}\label{eq:sewM}
\mc{M}_{\Sigma,V}\xrightarrow{\cong}(\mc{M}_{\Sigma',V'})_{\mf{c}_{sew}^{e_1,e_2},S_{sew}^{e_1,e_2}},
\end{equation}
 which identifies the following two morphisms of Manin pairs:
$$R_{\Sigma,V}:(\mbb{T},T)\mc{M}_{\Sigma,V}\dasharrow ((\ol{\g}\oplus\g)^{E_{\Gamma}},\g_{\Gamma})\times G^{E_{\Gamma}},$$
and
$$(R_{\Sigma',V'})_{\mf{c}_{sew}^{e_1,e_2},S_{sew}^{e_1,e_2}}:(\mbb{T},T)\mc{M}_{\Sigma,V}\dasharrow ((\ol{\g}\oplus\g)^{E_{\Gamma}},\g_{\Gamma})\times G^{E_{\Gamma}}.$$

Notice that
$$\mu^{-1}(S_{sew}^{e_1,e_2})=\{f\in\mc M_{\Sigma',V'}(G)\mid f(e_1)f(e_2)=1\}$$
and that
$$\mu^{-1}(S_{sew}^{e_1,e_2})/\g^{V'}\cap(\mf{c}_{sew}^{e_1,e_2})^\perp=\mc{M}_{\Sigma,V}(G).$$
Thus, the isomorphism \ref{eq:sewM} is clear. The non-trivial part of the statement is the behaviour of the $R_{\Sigma,V}$'s

\subsubsection{Commutativity of sewing}\label{sec:SewComm}
Suppose $\{e_1,e_2\},\{e_1',e_2'\}\subseteq E_\Gamma$ are two distinct pairs of distinct edges (i.e. $\{e_1,e_2\}\cap\{e_1',e_2'\}=\emptyset$), and let $i:\{e_1,e_2\}\to E_\Gamma$ and $i':\{e_1',e_2'\}\to E_\Gamma$ denote the inclusions. We may form the graph $\Gamma''$, with edge set $E_{\Gamma''}:=E_\Gamma\setminus \{e_1,e_2,e_1',e_2'\}$ and vertex set 
$$V_{\Gamma''}:=\frac{\on{out}(E_{\Gamma''})}{\on{in}(e_1)\sim\on{out}(e_2),\on{in}(e_2)\sim\on{out}(e_1),\text{ and }\on{in}(e_1')\sim\on{out}(e_2'),\on{in}(e_2')\sim\on{out}(e_1')}.$$
It is the graph obtained from $\Gamma$ by identifying the oppositely directed edges $e_1\sim e_2$ and $e_1'\sim e_2'$, and then deleting these newly identified edges and any isolated vertices.

 Let
$$\begin{array}{rlcccccc}
\mf{c}_1&:=\mf{c}_{sew}^{e_1,e_2}&=&i(\mf{l}_{sew})&\oplus&(\ol{\g}\oplus\g)^{\{e_1',e_2'\}}&\oplus& (\ol{\g}\oplus\g)^{E_{\Gamma''}},\\
\mf{c}_2&:=\mf{c}_{sew}^{e_1',e_2'}&=&(\ol{\g}\oplus\g)^{\{e_1,e_2\}}&\oplus& i'(\mf{l}_{sew})&\oplus& (\ol{\g}\oplus\g)^{E_{\Gamma''}},\\
\mf{c}_{1,2}&:=\mf{c}_1\cap\mf{c}_2&=&i(\mf{l}_{sew})&\oplus& i'(\mf{l}_{sew})&\oplus& (\ol{\g}\oplus\g)^{E_{\Gamma''}},
\end{array}$$
and 
$$\begin{array}{rlcccccc}
S_1&:=S_{sew}^{e_1,e_2}&=&i(G^\natural_\Delta)&\times& G^{\{e_1',e_2'\}}&\times& G^{E_{\Gamma''}},\\
S_2&:=S_{sew}^{e_1',e_2'}&=&G^{\{e_1,e_2\}}&\times& i'(G^\natural_\Delta)&\times& G^{E_{\Gamma''}},\\
S_{1,2}&:=S_1\cap S_2&=&i(G^\natural_\Delta)&\times& i'(G^\natural_\Delta)&\times& G^{E_{\Gamma''}},
\end{array}$$
If the pairs $(\mf{c}_{sew}^{e_1,e_2},S_{sew}^{e_1,e_2})$, $(\mf{c}_{sew}^{e_1',e_2'},S_{sew}^{e_1',e_2'})$ and $(\mf{c}_1\cap\mf{c}_2,S_1\cap S_2)$ all form reductive data for $R$, then the assumptions of \cref{cor:ComPartRed} are satisfied. Therefore, 
$$\big(R_{\mf{c}_{sew}^{e_1,e_2},S_{sew}^{e_1,e_2}}\big)_{\mf{c}_{sew}^{e_1',e_2'},S_{sew}^{e_1',e_2'}}=R_{\mf{c}_{1,2},S_{1,2}}=\big(R_{\mf{c}_{sew}^{e_1',e_2'},S_{sew}^{e_1',e_2'}}\big)_{\mf{c}_{sew}^{e_1,e_2},S_{sew}^{e_1,e_2}}$$
as morphisms of Manin pairs
$$(\mbb{T},T)M''\dasharrow ((\ol{\g}\oplus\g)^{E_{\Gamma''}},\g_{\Gamma''})\times G^{E_{\Gamma''}},$$
 where $M''=M_{\mf{c}_{1,2},S_{1,2}}$.
That is to say, it makes no difference in which order we sew pairs of edges: the results are all naturally isomorphic.

\subsection{The quasi-Hamiltonian structure on $\mc M_{\Sigma,V}(G)$}\label{sec:RSigViaTriang}

\begin{theorem}\label{thm:Independence}
There is a unique way to assign to every marked surface $(\Sigma,V)$  an exact morphism of Manin pairs
$$R_\Sigma:(\mbb{T},T)\mc{M}_{\Sigma,V}\dasharrow  \big((\ol{\g}\oplus\g)^{E_{\Gamma}},\g_{\Gamma}\big)\times G^{E_{\Gamma}}$$
supported on the graph of $\mu:\mc{M}_{\Sigma,V}\to G^{E_\Gamma}$ such that the assignment satisfies the sewing property.
\end{theorem}

\begin{figure}\label{fig:Pachner}
\begin{center}
\begin{subfigure}[t]{.45\linewidth}
\centering
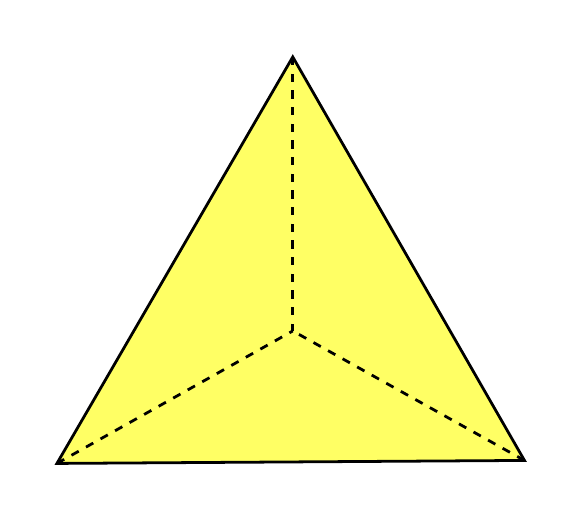
\caption{A triangulation of the $3$-gon.}
\end{subfigure}
\hspace{.05\linewidth}
\begin{subfigure}[t]{.45\linewidth}
\centering
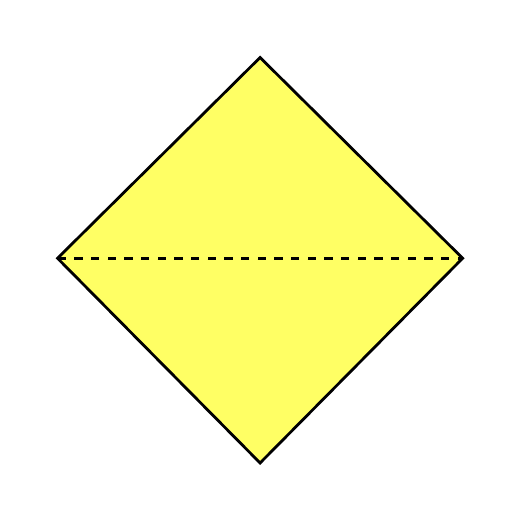
\caption{A Triangulation of the $4$-gon.}
\end{subfigure}
\caption{\label{fig:PachnerMoves}Triangulations of $3$ and $4$-gons. The dashed lines indicate internal edges of the triangulation along which we sew.}
\end{center}
\end{figure}

\begin{proof}
If $\Sigma'$ is a disjoint union of disks then $R_{\Sigma'}$ exists and is unique by \cref{prop:unPolyg}. 

Suppose that $\Sigma$ is triangulated and $V=\mc T^0\cap\partial\Sigma$. Let $\Sigma'$ be the disjoint union of the triangles and $V'$ the set of its vertices. Then, by sewing the respective edges of $\mc M_{\Sigma',V'}(G)$ according to the triangulation, we get an exact morphism of Manin pairs
$$R_{\mc T}:(\mbb{T},T)\mc{M}_{\Sigma,V}\dasharrow  \big((\ol{\g}\oplus\g)^{E_{\Gamma}},\g_{\Gamma}\big)\times G^{E_{\Gamma}}.$$
By commutativity of sewing these morphisms satisfy the sewing property. It remains to show that they are independent of the triangulation $\mc T$.

Thus, we need to prove that $R_\mc T$ is invariant under Pachner moves applied to $\mc T$. Pachner moves are, however, simply changes of triangulations of a polygon (either a triangle or a square, see \cref{fig:Pachner}). The independence thus follows from \cref{prop:unPolyg}.
\end{proof}

\section{Poisson structures on the moduli space of flat connections}\label{sec:PoisStrDirRed}

Suppose $G$ is a (connected) Lie group
with Lie algebra $\g$ and $s\in S^2(\g)^G$ is a $G$-invariant symmetric 2-tensor.
 We let $\mf{d}$ denote the Drinfel'd double of $\g$, as in \cref{rem:quasiTriStr}.
Suppose now that $\Gamma$ is a permutation graph.
Then $\mf{d}^{V_\Gamma}$ acts on $G^{E_\Gamma}$ with coisotropic stabilizers, as follows:  $\mf{d}^{V_\Gamma}$ acts on the $e\in E_\Gamma$-th factor $G^{E_\Gamma}$ via the vector field
$$\rho(\xi)_e= -\on{s}(\xi_{\on{in}(e)})^R+\on{t}(\xi_{\on{out}(e)})^L,\quad \xi\in \mf{d}^{V_\Gamma},$$
(cf. \cref{fig:3Cyclb1}).
Thus $$\g^{V_\Gamma}\times G^{E_\Gamma}\subseteq \mf{d}^{V_\Gamma}\times G^{E_\Gamma}$$ 
is a Dirac structure.

\begin{figure}[!h]
\begin{center}
\begin{subfigure}[t]{.45\linewidth}
\centering
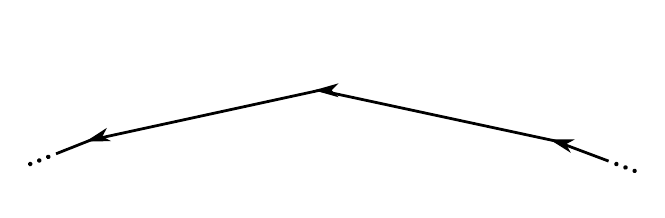
\caption{Here $\on{in}(e)=v=\on{out}(e')$, where $v\in V_\Gamma$ is the vertex and $e,e'\in E_\Gamma$ are edges.}
\label{fig:3Cycla1}
\end{subfigure}
\hspace{.05\linewidth}
\begin{subfigure}[t]{.45\linewidth}
\centering
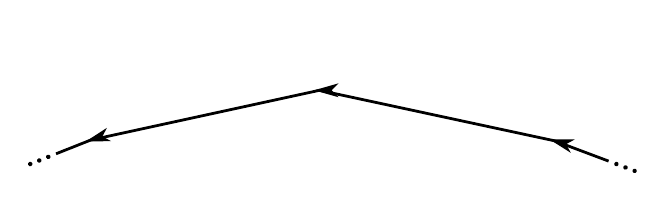
\caption{The element $\xi_{(\cdot)}\in \mf{d}^{V_\Gamma}$ acts diagonally by $\xi_v$ at the vertex $v$. }
\label{fig:3Cyclb1}
\end{subfigure}
\caption{}
\end{center}
\end{figure}

Suppose $(\Sigma, V)$ is a marked surface (where we now allow components of $\partial\Sigma$ to intersect $V$ trivially), with boundary graph $\Gamma$. In this section, we prove there exists a natural quasi-Hamiltonian $(\mf{d},\g)^{V_\Gamma}\times G^{E_\Gamma}$ structure on $\mc{M}_{\Sigma,V}(G)$, i.e a morphism of Manin pairs
$$R_{\Sigma,V}:(\mbb{T},T)\mc{M}_{\Sigma,V}(G)\dasharrow (\mf{d},\g)^{V_\Gamma}\times G^{E_\Gamma}$$
over the map $$\mu:\mc{M}_{\Sigma,V}(G)\to G^{E_\Gamma}$$ given by the boundary holonomies.

As before, for disjoint unions of polygons the quasi-Hamiltonian structure is uniquely defined.
\begin{proposition}[union of polygons]\label{prop:unPolygP}
Let $\Sigma$ be a disjoint union of discs and $V\subset\partial\Sigma$ a finite subset. Then there is a unique exact morphism of Manin pairs
$$(\mathbb T,T)\mathcal M_{\Sigma,V}\dasharrow  \big((\ol{\g}\oplus\g)^{E_{\Gamma}},\g_{\Gamma}\big)\times G^{E_{\Gamma}}$$
over $\mu$.
\end{proposition}
\begin{proof}
As in \cref{prop:unPolyg}, $\mu$ is an embedding, so this follows from \cref{ex:EInvSubMMP}.
\end{proof}

\subsection{Fusion}\label{sec:fus}

Suppose that $P,Q\in V_\Gamma$ are two distinct vertices. The operation of \emph{fusion} at the ordered pair of vertices $(P,Q)$ described in  \cite{LiBland:2012vo} (following \cite{Alekseev97,Alekseev00}), 
can be understood in terms of a morphism of Manin pairs 
$$R_{(P,Q)}:(\mf{d},\g)^{V_\Gamma}\times G^{E_\Gamma}\dasharrow (\mf{d},\g)^{V_{\Gamma^*}}\times G^{E_{\Gamma^*}}$$
where the graph $\Gamma^*$ is a permutation graph constructed from $\Gamma$, as we shall now explain:

Let ${}_{\overleftarrow{P}}e=\on{in}^{-1}(P)$ and $e_{\overleftarrow{P}}:=\on{out}^{-1}(P)$ denote the edges entering and exiting $P$. Similarly, let ${}_{\overleftarrow{Q}}e=\on{in}^{-1}(Q)$ and $e_{\overleftarrow{Q}}:=\on{out}^{-1}(Q)$ denote the edges entering and exiting $Q$.
\begin{description}
\item[Case 1: ${}_{\overleftarrow{Q}}e=e_{\overleftarrow{P}}$]
In this case, $\Gamma^*$ is obtained by discarding the edge ${}_{\overleftarrow{Q}}e=e_{\overleftarrow{P}}$ and identifying the vertices $P$ and $Q$  (cf. \cref{fig:GraphFusDel}).

Meanwhile for $(\xi,g)\in\mf{d}^{V_\Gamma}\times G^{E_\Gamma}$ and $(\xi^*,g^*)\in\mf{d}^{V_{\Gamma^*}}\times G^{E_{\Gamma^*}}$
$$(\xi,g)\sim_{R_{(P,Q)}}(\xi^*,g^*)$$
if and only if $g^*_e=g_e$ for every $e\in E_{\Gamma^*}$ and
\begin{equation}\label{eq:FusMMPond}
\xi^*_v=\begin{cases}
\xi_{Q}\circ \xi_{P} &\text{if }v\text{ is the vertex obtained by identifying }P\text{ and } Q\\
\xi_v&\text{otherwise,}
\end{cases}\end{equation}
(in particular, we assume that $\xi_{Q}$ and $\xi_{P}$ are composable elements of the Lie groupoid $\mf{d}$).
\item[Case 2: ${}_{\overleftarrow{Q}}e\neq e_{\overleftarrow{P}}$]
In this case, $\Gamma^*$ is obtained by identifying the vertices $P$ and $Q$ and composing the edges  $e_{\overleftarrow{P}}$ and ${}_{\overleftarrow{Q}}e$ to form a new edge $e_{\widehat{PQ}}$ (cf. \cref{fig:GraphFusComp}).

Meanwhile for $(\xi,g)\in\mf{d}^{V_\Gamma}\times G^{E_\Gamma}$ and $(\xi^*,g^*)\in\mf{d}^{V_{\Gamma^*}}\times G^{E_{\Gamma^*}}$
$$(\xi,g)\sim_{R_{(P,Q)}}(\xi^*,g^*)$$
if and only if
$$g^*_e=\begin{cases}
g_{e_{\overleftarrow{P}}}g_{{}_{\overleftarrow{Q}}e}&\text{if }e=e_{\widehat{PQ}}\\
g_e&\text{otherwise},
\end{cases}$$
while $\xi^*$ and $\xi$ satisfy \cref{eq:FusMMPond}, as before.
\end{description}

\begin{figure}[h]
\begin{center}
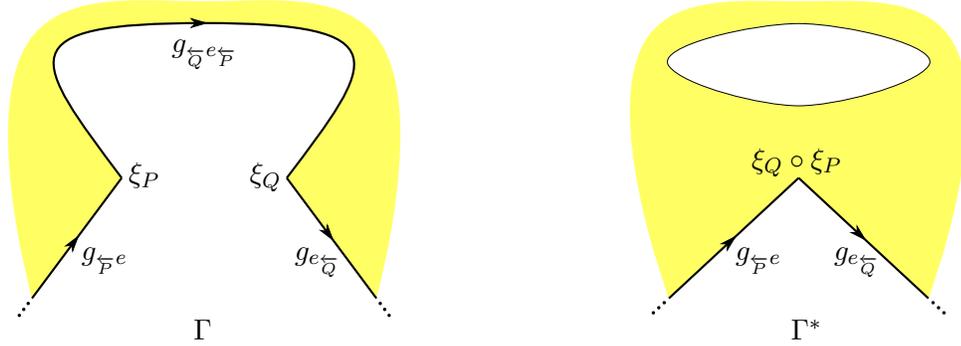
\caption{\label{fig:GraphFusDel}The graph $\Gamma^*$ is the permutation graph obtained from $\Gamma$ by deleting the edge, ${}_{\overset{\leftarrow}{Q}}e_{\overset{\leftarrow}{P}}$, passing from $P$ to $Q$, and identifying the vertices $P$ and $Q$. (The graphs are isomorphic outside the pictured regions.) Heuristically, we have obtained the graph $\Gamma^*$ by gluing short sections from both ends of ${}_{\overset{\leftarrow}{Q}}e_{\overset{\leftarrow}{P}}$ together (effectively divorcing it from the graph).
\\
Meanwhile, the Courant morphism, $R_{(P,Q)}$, is defined by composing the corresponding elements labelling the vertices and by forgetting the element which labelled the deleted edge. }
\end{center}
\end{figure}

\begin{figure}[h]
\begin{center}
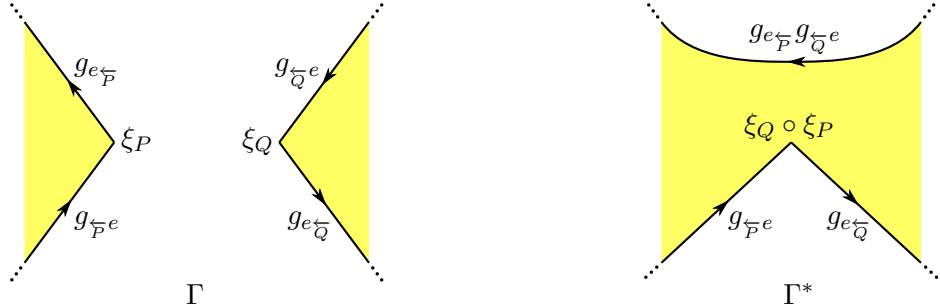
\caption{\label{fig:GraphFusComp}The graph $\Gamma^*$ is the permutation graph obtained from $\Gamma$ by identifying the vertices $P$ and $Q$ and composing the edge entering $Q$ with the edge leaving $P$. (The graphs are isomorphic outside the pictured regions.) Heuristically, we have obtained the graph $\Gamma^*$ by gluing a short section of the edges $e_{\overset{\leftarrow}{P}}$ and ${}_{\overset{\leftarrow}{Q}}e$ together. \\ Meanwhile, the Courant morphism, $R_{(P,Q)}$, is defined by composing the corresponding elements labelling the edges and vertices. }
\end{center}
\end{figure}

Now suppose that $M$ is a quasi-Hamiltonian $(\mf{d},\g)^{V_\Gamma}\times G^{E_\Gamma}$-space defined by the morphism of Manin pairs
$$R:(\mbb{T},T)M\dasharrow(\mf{d},\g)^{V_\Gamma}\times G^{E_\Gamma}.$$
Then the morphism of Manin pairs
$$R^*:=R_{(P,Q)}\circ R:(\mbb{T},T)M\dasharrow(\mf{d},\g)^{V_{\Gamma^*}}\times G^{E_{\Gamma^*}}$$
defines a quasi-Hamiltonian $(\mf{d},\g)^{V_{\Gamma^*}}\times G^{E_{\Gamma^*}}$-structure on $M$ which we call the \emph{fusion} of $R$ at the ordered pair $(P,Q)$ of vertices.

%

\begin{remark}[Associativity of Fusion]
Since the Courant morphism $R_{(P,Q)}$ is defined in terms of the groupoid structure on $\mf{d}$ and the group structure on $G$, it follows that fusion is an associative operation.
\end{remark}

\subsection{The quasi-Hamiltonian structure for quasi-triangular structure Lie algebras}\label{sec:qPoisModSpc}
Let $(\Sigma,V)$ be a marked surface.
If we choose an ordered pair $(P,Q)$ of marked points ($P\neq Q\in V$) then the corresponding \emph{fused surface} $\Sigma^*$ is obtained by gluing a short piece of the arc starting at $P$ with a short piece of the arc ending at $Q$ (so that $P$ and $Q$ get identified).
The subset $V^*\subset\partial\Sigma^*$ is obtained from $V$ by identifying $P$ and $Q$. The map 
$$M_{\Sigma^*,V^*}(G)\to M_{\Sigma,V}(G),$$
 coming from the map $(\Sigma,V)\to(\Sigma^*,V^*)$, is a diffeomorphism.

\begin{theorem}\label{thm:qPoiss-mod}
There is unique way to assign to every marked surface $(\Sigma,V)$ a morphism of Manin pairs
$$R_{\Sigma,V}:(\mbb{T},T)\mc{M}_{\Sigma,V}(G)\dasharrow (\mf{d},\g)^{V_\Gamma}\times G^{E_\Gamma}$$
supported on the graph of $\mu:\mc{M}_{\Sigma,V}(G)\to G^{E_\Gamma}$ such that if $(\Sigma^*,V^*)$ is obtained from $(\Sigma,V)$ by fusion, then $R_{\Sigma^*,V^*}$ is obtained from $R_{\Sigma,V}$ by the corresponding fusion.
\end{theorem}
We defer the proof until the next section.
\begin{remark}
When $s\in S^2(\g)^G$ is non-degenerate (i.e. $\g$ is quadratic), then $\mf{d}=\g\oplus\bar\g$ is the pair groupoid and $$(\on{in}\oplus\on{out})^!(\on{s}\oplus\on{t}):\mf{d}^{V_\Gamma}\to (\bar\g\oplus\g)^{E_\Gamma}$$ is an isomorphism.
In this case, it is not difficult to convince oneself that the morphisms of Manin pairs 
$$R_{\Sigma,V}:(\mbb{T},T)\mc{M}_{\Sigma,V}(G)\dasharrow ((\bar\g\oplus\g)^{E_\Gamma},\g_\Delta^{V_\Gamma})\times G^{E_\Gamma}$$
described in \cref{thm:qPoiss-mod} satisfy the sewing property. That is they are precisely the ones described in \cref{thm:Independence}.
\end{remark}

\subsubsection{quasi-Hamiltonian $(\mf{d},\g)^{V_\Gamma}\times G^{E_\Gamma}$-manifolds and quasi-Poisson geometry}
Using the material in \cref{sec:qPoissonRed}, we intend to relate quasi-Hamiltonian $(\mf{d},\g)^{V_\Gamma}\times G^{E_\Gamma}$-spaces to the quasi-Poisson spaces studied in \cite{Alekseev00,LiBland:2012vo}.
A canonical choice of complement to $\g\subseteq\mf{d}$ is $\g_{\bar\Delta}^*:=(\on{s}+\on{t})^*(\g^*)\subseteq\mf{d}$, explicitly
$$\g_{\bar\Delta}^*=\big\{\big(\alpha-\frac{1}{2}s(\alpha,\cdot)\big)\mid \alpha\in \mf{p}\subseteq\mf{d}\big\}.$$ Similarly, $\mf{k}=(\g_{\bar\Delta}^*)^{V_\Gamma}$ is a canonical choice of complement to $\g^{V_\Gamma}\subseteq\mf{d}^{V_\Gamma}$.

Let $\ol{\sigma_\Gamma}:V_\Gamma\to V_\Gamma$ be the permutation given by walking along the graph $\Gamma$ \emph{against} the direction of each edge. Suppose that $(M,\rho,\pi)$ is a quasi-Poisson $G^{V_\Gamma}$-manifold, in the sense of \cite{Alekseev00,LiBland:2012vo} and $\mu:M\to G^{V_\Gamma}$ is a $\ol{\sigma_\Gamma}^!$-twisted moment map in the sense of \cite{LiBland:2012vo}. Let $\tilde\mu:M\to G^{E_\Gamma}$ be defined by
$$\tilde\mu(m)_e=\big(\mu(m)_{\on{in}(e)}\big)^{-1}, \quad m\in M, e\in E_\Gamma.$$

\begin{proposition}\label{prop:TwistQPoiss}
There exists a unique morphism of Manin pairs
$$R:(\mbb{T},T)M\dasharrow (\mf{d}^{V_\Gamma},\g^{V_\Gamma})\times G^{E_\Gamma}$$ over the map $\tilde\mu$
which is compatible with the $\g^{V_\Gamma}$ action on $M$ and such that $\mf{k}\circ R=\on{gr}(\pi^\sharp)$. 

Moreover, the converse holds whenever the action of $\g^{V_\Gamma}$ on $M$ integrates to an action of $G^{V_\Gamma}$.
\end{proposition}
\begin{proof}
This follows from a direct application of \cite[Proposition 3.5]{Bursztyn:2009wi}. (cf. \cref{sec:qPoissonRed}).
\end{proof}

 In this sense there is a one-to-one correspondence between quasi-Poisson $G^{V_\Gamma}$-manifolds with $\ol{\sigma_\Gamma}^!$-twisted moment maps and quasi-Hamiltonian $(\mf{d}^{V_\Gamma},\g^{V_\Gamma})\times G^{E_\Gamma}$-manifolds.

\begin{proposition}\label{prop:FusQPois}
Suppose $(M,\rho,\pi)$ is a quasi-Poisson $G^{V_\Gamma}$-manifold with $\ol{\sigma_\Gamma}^!$-twisted moment map, and let
$$R:(\mbb{T},T)M\dasharrow (\mf{d}^{V_\Gamma},\g^{V_\Gamma})\times G^{E_\Gamma}$$
be the corresponding quasi-Hamiltonian $(\mf{d}^{V_\Gamma},\g^{V_\Gamma})\times G^{E_\Gamma}$-structure on $M$.
Let $R^*$ denote the fusion of $R$ at the ordered pair of vertices $(P,Q)\subset V_\Gamma$. Then the bivector field for the quasi-Poisson $G^{V_{\Gamma^*}}$-structure corresponding to $R^*$ is
$$\pi^*:=\pi+\rho(\tau),$$ 
where $\tau\in \wedge^2(\g^{V_\Gamma})$ is the insertion of $\psi\in\bigwedge^2(\g^{P}\oplus\g^{Q})$,
$$\psi=\frac{1}{2}\sum_{i,j}s^{ij}\,(\xi_i,0)\wedge(0,\xi_j)$$
 at the ${P,Q}$-th factors. Here  $s=\sum_{i,j}s^{ij}\,\xi_i\otimes \xi_j$ in some basis $\xi_i$ of $\g$.
\end{proposition}
Thus, (up to a sign difference) fusion in the sense of \cref{sec:fus} is precisely the same as fusion in the sense of \cite{Alekseev00,LiBland:2012vo}.
\begin{proof}
Let $\mf{k}_*=(\g_{\bar\Delta}^*)^{V_{\Gamma^*}}$. A straightforward computation shows that $$\mf{k}_*\circ R_{(P,Q)}=\{(\xi+\tau^\sharp(\xi))\mid\xi\in\mf{k}\}.$$
By definition, $\on{gr}\big((\pi^*)^\sharp\big)=\mf{k}_*\circ R^*$. Thus, we see from $$\mf{k}_*\circ R_{(P,Q)}\circ R=\on{gr}(\pi)+\on{gr}(\rho(\tau))$$
that $\pi^*=\pi+\rho(\tau)$.
\end{proof}

\begin{remark}[Sign differences]\label{rem:SignDifferences}
In \cite{LiBland:2012vo}, the bivector field resulting from fusing the ordered pair $(P,Q)$ is defined to be $\pi^*=\pi-\rho(\tau)$. This difference is essentially due to the fact that we orient our boundary graph to agree with the orientation of $\partial\Sigma$, whereas the opposite convention is used in \cite{LiBland:2012vo}.
\end{remark}

\begin{proof}[Proof of \cref{thm:qPoiss-mod}]
\Cref{prop:TwistQPoiss,prop:FusQPois} show that it suffices to prove the equivalent statement for quasi-Poisson $G^{V_\Gamma}$-structures.
However, \cite[Theorem~2]{LiBland:2012vo} and \cite[Theorem 4]{LiBland:2012vo} shows there exists a unique quasi-Poisson $G^{V_\Gamma}$-structure on $\mc{M}_{\Sigma,V}(G)$ with $\ol{\sigma_\Gamma}^!$-twisted moment map which is compatible with fusion (notice that \cref{prop:unPolygP} implies the first two properties of \cite[Theorem~2]{LiBland:2012vo} are automatically satisfied).
\end{proof}
\begin{corollary}
The proof of \cref{thm:qPoiss-mod} also shows that the bivector field on $\mc{M}_{\Sigma,V}$ is given by \cref{eq:pi_hpair}.
\end{corollary}

\appendix
\section{Proofs of reduction theorems}\label{app:PrtRedPf}

Before proving \cref{thm:PartRed,thm:ExactPartRed,thm:PartRedSplEx}, we first establish some lemmas.

\begin{lemma}\label{lem:EstMMP}
Suppose that $(\eta;Z)\in R_{\mf{c},S}$, where $\eta\in\mf{d}_\mf{c}$,  $Z\in TM_{\mf{c},S}$ and $R_{\mf{c},S}$ is as in \cref{thm:PartRed}. 
 Then
\begin{enumerate}
\item $\eta\in\h_\mf{c}$, and
\item $\eta=0$ only if $Z=0$.
\end{enumerate}
\end{lemma}
\begin{proof}
Let $\xi\in\mf{d}$, $X\in TM$ and $\alpha\in T^*M$ be chosen so that $(\xi;X+\alpha)\in R$ and
$$(\xi;X+\alpha)\sim_{(R_2\times R_1)}(\eta;Z).$$
Since $R_1=R_{q_M}\circ R_{i_M}^\top$, it follows that $\alpha\in\ann(T\mu^{-1}S)$.
Consequently, there exists $\tilde\alpha\in \ann(TS)$ such that $\alpha=\mu^*\tilde\alpha$.   Since $S$ is $\mf{c}$ invariant, $$\zeta_\alpha:=\mbf{a}^*\tilde\alpha\in\mf{c}^\perp.$$
Moreover, since $R$ is supported on the graph of $\mu$, $(\zeta_\alpha,\alpha)=\mbf{a}^*(-\tilde\alpha,\mu^*\tilde\alpha)\in R$. Thus 
\begin{equation}\label{eq:xizetalXinR}(\xi-\zeta_\alpha;X)\in R\end{equation} and
 \begin{equation}\label{eq:zetAl}(\xi-\zeta_\alpha;X)\sim_{(R_2\times R_1)} (\eta;Z).\end{equation}
 Since \labelcref{eq:PartRedMMPin} is a morphism of Manin pairs, axiom (m1) of \cref{def:MorpManPair} implies that 
 \begin{equation}\label{eq:xizetalinh}\xi-\zeta_\alpha\in \h.\end{equation} Thus \cref{eq:zetAl} implies that $\eta\in\h_\mf{c}$, establishing the first claim.
 
 Next, suppose $\eta=0$, then \cref{eq:zetAl} implies that $\xi-\zeta_\alpha\in\mf{c}^\perp$ in addition to \cref{eq:xizetalinh}. That is $\xi-\zeta_\alpha\in\mf{c}^\perp\cap \h$, and hence \cref{eq:rhoR,eq:xizetalXinR} imply that $X$ is tangent to the $\mf{c}^\perp\cap \h$. Therefore $Z=q_M(X)=0$, establishing the second claim.
\end{proof}

\begin{lemma}\label{lem:cleanComp}
Under the assumptions of \cref{thm:PartRed}, the Courant relation $$R_2\times R_1:(\mf{d}\times N)\times \ol{\mbb{T}M}\dasharrow(\mf{d}_\mf{c}\times N_{\mf{c},S})\times \ol{\mbb{T}M_{\mf{c},S}}$$ composes cleanly with the Dirac structure
$R\subseteq (\mf{d}\times N)\times \ol{\mbb{T}M}$.

Moreover, the composition
$$R_{\mf{c},S}:=R_2\circ R\circ R_1^\top$$ 
is a well defined subbundle of $(\mf{d}_\mf{c}\times N_{\mf{c},S})\times \ol{\mbb{T}M_{\mf{c},S}}$.
\end{lemma}

\begin{proof}
We begin by proving that the composition $(R_2\times R_1)\circ R$ is clean. For this, it is sufficient to show that 
\begin{enumerate}
\item the rank of the intersections $\on{ker}(R_2\times R_1)^\perp \cap R$ and $\on{ker}(R_2\times R_1) \cap R$ are constant, and 
\item the composition of the underlying relation of vector bundle bases, $$\big(\gr(q_N\times q_M)\circ \gr(i_N\times i_M)^\top\big)\circ \gr(\mu)$$ is clean.
\end{enumerate}
We now show that the rank of $\on{ker}(R_2\times R_1) \cap R$ is constant.
 We claim that the sequence
\begin{equation}\label{eq:ConstRankExSeq}0\to\h\cap\mf{c}^\perp\xrightarrow{\xi\to (\xi,\rho_R(\xi))} \on{ker}(R_2\times R_1) \cap R\xrightarrow{(\xi,X+\alpha)\to \alpha}\ann(T\mu^{-1}S)\to 0\end{equation}
is exact, where $\rho_R$ is defined in \cref{eq:rhoR}. 

First, the second map is surjective: for any $\alpha\in\ann(T\mu^{-1}S)$, let $\tilde\alpha\in \ann(TS)$ be chosen so that $\mu^*\tilde\alpha=\alpha$. Since $S$ is $\mf{c}$ invariant, $$\zeta_\alpha:=\mbf{a}^*\tilde\alpha\in\mf{c}^\perp.$$
Moreover, since $R$ is supported on the graph of $\mu$, $$(\zeta_\alpha,\alpha)=\mbf{a}^*(-\tilde\alpha,\mu^*\tilde\alpha)\in  \on{ker}(R_2\times R_1)\cap R.$$

Next, we prove exactness at $\on{ker}(R_2\times R_1) \cap R$. Suppose $(\xi,X)\in\on{ker}(R_2\times R_1)\cap R$. Since \cref{eq:PartRedMMPin} is a morphism of Manin pairs, $\xi\in\h$ and $X=\rho_R(\xi-\zeta_\alpha)$. Since $\xi\in \on{ker}(R_2)=\mf{c}^\perp$, we conclude that $\xi\in\h\cap\mf{c}^\perp$.

 This shows that the sequence \cref{eq:ConstRankExSeq} is exact.
 Since $\h\cap\mf{c}^\perp$ and $\ann(T\mu^{-1}S)$ are both of constant rank, so is  $\on{ker}(R_2\times R_1) \cap R$. Consequently $\big(\on{ker}(R_2\times R_1) \cap R\big)^\perp=\on{ker}(R_2\times R_1)^\perp + R$ is also of constant rank, and thus so is $\on{ker}(R_2\times R_1)^\perp \cap R$.
 
 Next, we need to show that the following composition of relations is clean:
$$\begin{tikzpicture}
\mmat{m}{M&N\\\tilde S&S\\M_{\mf{c},S}&N_{\mf{c},S}\\};
\path[->] (m-1-1) edge node {$\mu$} (m-1-2);
\path[->] (m-2-1) 
			edge node[swap] {$q_M$} (m-3-1)
			edge node {$i_M$} (m-1-1);
\path[->] (m-2-2) edge node {$q_N$} (m-3-2)
		edge node[swap] {$i_N$} (m-1-2);
\path[dashed,->] (m-3-1) edge node {$\tilde \mu$} (m-3-2);
\end{tikzpicture}$$
where, for brevity, we have introduced the notation $\tilde S:=\mu^{-1}(S)$.
 Since $\gr(\mu)$ intersects $S\times M$ cleanly, the composition $\gr(\mu\rvert_{\tilde S})=\gr(i_N)^\top\circ \gr(\mu\circ i_M)$ is clean
$$\begin{tikzpicture}
\mmat{m}{M&N\\\tilde S&S\\M_{\mf{c},S}&N_{\mf{c},S}\\};
\path[->] (m-1-1) edge node {$\mu$} (m-1-2);
\path[->] (m-2-1) edge node {$\mu\rvert_{\tilde S}$} (m-2-2)
			edge node[swap] {$q_M$} (m-3-1)
			edge node {$i_M$} (m-1-1);
\path[->] (m-2-2) edge node {$q_N$} (m-3-2)
		edge node[swap] {$i_N$} (m-1-2);
\path[dashed,->] (m-3-1) edge node {$\tilde \mu$} (m-3-2);
\end{tikzpicture}$$
Next, $$\gr(q_N\circ \mu\rvert_{\tilde S})\times \gr(q_M)^\top\subseteq N_{\mf{c},S}\times \tilde S\times \tilde S\times M_{\mf{c},S}$$ intersects $N_{\mf{c},S}\times \tilde S_\Delta \times M_{\mf{c},S}$ transversely, since $q_M$ and $q_N\circ \mu\rvert_{\tilde S}$ are both maps. Moreover, since $M_{\mf{c},S}$ is the set of $\h\cap \mf{c}^\perp$ orbits in $\tilde S$, while $N_{\mf{c},S}$ is the set of $\mf{c}^\perp$ orbits in $S$, the projection $$\bigg(\gr(q_N\circ \mu\rvert_{\tilde S})\times \gr(q_M)^\top\bigg)\cap \bigg(N_{\mf{c},S}\times \tilde S_\Delta \times M_{\mf{c},S}\bigg)\to \gr(\tilde \mu)$$ is a surjective submersion. Thus, by definition, $\gr(q_N\circ \mu\rvert_{\tilde S})$ composes cleanly with $\gr(q_M)^\top$.

It follows that $R_2\times R_1$ composes cleanly with $R$.

Finally, since $\{(\xi;\rho_R(\xi))\mid \xi\in\h\cap\mf{c}^\perp\}\subseteq R$, it follows that $R$ is $\h\cap\mf{c}^\perp$ invariant. Hence $R_{\mf{c},S}:=R_2\circ R\circ R_1^\top$ is a well defined subbundle of $$\big(\mf{d}_{\mf{c}}\times D_\mf{c}/G_\mf{c}\big)\times \overline{\mbb{T}Q}.$$
\end{proof}

We are now ready to prove \cref{thm:PartRed,thm:ExactPartRed}.

\begin{proof}[Proof of \cref{thm:PartRed}]
Since $N_{\mf{c},S}$ is the space of $\mf{c}^\perp$-orbits of $S$, while $M_{\mf{c},S}$ is the space of $\mf{c}^\perp\cap \h$ orbits of $\mu^{-1}(S)$,  the $\h$-equivariant map 
$$\mu:\mu^{-1}(S)\to S$$
descends to a define a unique map
$$\tilde\mu:M_{\mf{c},S}\to N_{\mf{c},S}.$$

The composition $R_{\mf{c},S}:=R_2\circ R\circ R_1^\top$ is supported on the graph of $\tilde\mu$. Thus \cref{lem:cleanComp,prop:CompCourRel} show that 
$$R_{\mf{c},S}:\mbb{T} M_{\mf{c},S}\dasharrow \mf{d}_\mf{c}\times N_{\mf{c},S}$$
is a Courant morphism.

Finally, \cref{lem:EstMMP} proves that \labelcref{eq:PartRedMMPout} satisfies the defining conditions for a morphism of Manin pairs.
\end{proof}

\begin{proof}[Proof of \cref{thm:ExactPartRed}]
We need to show that \cref{eq:PartRedMMPout} is a exact morphism of Manin pairs. We do this by first showing that $\mf{d}_\mf{c}\times N_{\mf{c},S}$ is an exact Courant algebroid along the image of $\mu(M)\cap S$, and next by showing that the anchor maps $R_{\mf{c},S}$ surjectively onto $T\gr(\mu_{\mf{c},S}:M_{\mf{c},S}\to N_{\mf{c},S})$.

The fact that $\mf{d}_\mf{c}\times N_{\mf{c},S}$ is exact follows from \cite[Theorem 3.3]{Bursztyn:2007ko}, but we include a proof here anyways.
We must show that
\begin{equation}\label{eq:dcExtSeq}0\to T^*N_{\mf{c},S}\xrightarrow{\mbf{a}^*}\mf{d}_\mf{c}\times N_{\mf{c},S}\xrightarrow{\mbf{a}}TN_{\mf{c},S}\to 0\end{equation}
is an exact sequence. By assumption, $\mf{c}$ acts transitively on $S$. Thus $\mf{d}_\mf{c}:=\mf{c}/\mf{c}^\perp$ acts transitively on $N_{\mf{c},S}:=S/\mf{c}^\perp$. It follows that the sequence \labelcref{eq:dcExtSeq} is exact at $TN_{\mf{c},S}$, and hence, by duality, also at $T^*N_{\mf{c},S}$. 

Next, \cref{eq:dcExtSeq} is exact at $\mf{d}_\mf{c}\times N_{\mf{c},S}$ if and only if $\on{ker}(\mbf{a})$ is isotropic. This, in turn, holds if and only if $\mf{c}\cap(\on{ker}(\mbf{a}_\mf{d})+\mf{c}^\perp)$ is isotropic, where $\mbf{a}_\mf{d}$ denotes the anchor map for $\mf{d}\times N$. But
 $$\big(\mf{c}\cap(\on{ker}(\mbf{a}_\mf{d})+\mf{c}^\perp)\big)^\perp=\mf{c}^\perp+\on{Im}(\mbf{a}_\mf{d}^*)\cap\mf{c}^\perp=\mf{c}\cap(\on{Im}(\mbf{a}_\mf{d}^*)+\mf{c}^\perp).$$
 Therefore, \labelcref{eq:dcExtSeq} is exact at $TN_{\mf{c},S}$ if $\on{Im}(\mbf{a}_\mf{d}^*)=\on{ker}(\mbf{a}_\mf{d})$, which holds whenever $\mf{d}\times N$ is exact (as it is along $\mu(M)$).
 
 Next we need to show that the anchor maps $R_{\mf{c},S}$ surjectively onto $T\gr(\mu_{\mf{c},S}:M_{\mf{c},S}\to N_{\mf{c},S})$.
 More precisely, for any $Z\in TM_{\mf{c},S}$, we must show there exists $\eta\in\mf{d}_\mf{c}$ and $\gamma\in T^*M_{\mf{c},S}$ such that $$(\eta,Z+\gamma)\in R_{\mf{c},S}.$$
 Let $X\in T\mu^{-1}(S)$ be chosen so that it maps to $Z$ under the quotient map $q_M:\mu^{-1}(S)\to M_{\mf{c},S}$. Since \cref{eq:PartRedMMPin} is exact, there exists $\alpha\in T^*M$ and $\xi\in \mf{d}$ such that $(\xi,X+\alpha)\in R$. Now $R$ is supported on the graph of $\mu$, so $\mbf{a}_\mf{d}(\xi)=\d\mu(X)\in TS$. Since $\mf{c}$ acts transitively on $S$, we must have $\xi\in\mf{c}+\on{ker}(\mbf{a}_\mf{d})$. Since $\mf{d}\times N$ is exact, $\on{ker}(\mbf{a}_\mf{d})=\on{Im}(\mbf{a}_\mf{d}^*)$. Thus there exists $\xi'\in\mf{c}$ and $\beta\in T^*N$ such that $\xi=\xi'+\mbf{a}_\mf{d}^*\beta$. Since $R$ is supported on the graph of $\mu$, $(\mbf{a}_\mf{d}^*\beta,\mu^*\beta)\in R$, and thus
 $$(\xi',X+\alpha-\mu^*\beta)\in R.$$
 Since $R$ is Lagrangian, pairing this element with $(\zeta,\rho_R(\zeta))\in R$, where $\zeta\in\h\cap\mf{c}^\perp$, we see that
 $$0=\la\xi',\zeta\ra=\la\alpha-\mu^*\beta,\rho_R(\zeta)\ra.$$
 Since $\zeta\in\h\cap\mf{c}^\perp$ was arbitrary, this shows that $\alpha-\mu^*\beta=q_M^*\gamma$ for some $\gamma\in T^*M_{\mf{c},S}$.
 
 Thus, we have shown that $$(\xi'+\mf{c}^\perp,Z+\gamma)\in R_{\mf{c},S},$$ where $\xi'+\mf{c}^\perp$ is the image of $\xi'$ under the quotient map $\mf{c}\to\mf{c}/\mf{c}^\perp$. Since $Z\in TM_{\mf{c},S}$ was arbitrary, we may conclude that \cref{eq:PartRedMMPout} is exact.
 
\end{proof}

\begin{proof}[Proof of \cref{thm:PartRedSplEx}]
First we show that the 2-form $\omega_\theta:=i^*_M\omega-\mu^*\la\theta,\vartheta_s-\frac{1}{2}s\circ\mbf{a}(\theta)\ra\in\Omega^2\big(\mu^{-1}(S)\big)$ is $\h\cap\mf{c}^\perp$-invariant and basic. 
To show invariance, note that (by definition) $(\xi,\rho_R(\xi))\in R$ for any $\xi\in\h$, and thus $R$ is $\h$-invariant. Since $s:TN\to\mf{d}\times N$ is an $\mf{d}$-invariant splitting, it follows that $R_{\mu,\omega}$ and hence $\omega\in\Omega^2(M)$ is $\h$-invariant. Additionally, $s,\mbf{a}$ and $\vartheta_s$ are $\mf{d}$-equivariant, while $\theta$ is $\mf{c}^\perp$-equivariant. Hence $\la\theta,\vartheta_s-\frac{1}{2}s\circ\mbf{a}(\theta)\ra$ is $\mf{c}^\perp$-invariant. It follows that the sum, $\omega_\theta$ is $\h\cap\mf{c}^\perp$-invariant.

The isomorphism $\mf{d}\times N\to \mbb{T}_\gamma N$ is given by $\xi\to \mbf{a}(\xi)+s^*(\xi)$, for all $\xi\in\mf{d}$. Thus, for $\xi,\eta\in\mf{d}$, we have
$$\la\xi,\eta\ra=\la\mbf{a}(\xi)+s^*(\xi),\mbf{a}(\eta)+s^*(\eta)\ra=\la s\circ\mbf{a}(\xi),\eta\ra+\la\xi,s\circ\mbf{a}(\eta)\ra.$$
Thus, since $\mf{c}^\perp$ is coisotropic, the assignment $\xi,\eta\to\la\xi,s\circ\mbf{a}\eta\ra$ defines a skew-symmetric form on $\mf{c}^\perp$.
%

Now suppose $\xi\in\mf{c}^\perp$, then 
\begin{align}
\notag&\iota_{\mbf{a}(\xi)}\la\theta,\vartheta_s-\frac{1}{2}s\circ\mbf{a}(\theta)\ra\\
\notag=&\la\xi,\vartheta_s-\frac{1}{2}s\circ\mbf{a}(\theta)\ra-\la\theta,s\circ\mbf{a}(\xi)-\frac{1}{2}s\circ\mbf{a}(\xi)\ra\\
\notag=&s^*\xi-\frac{1}{2}\la\xi,s\circ\mbf{a}(\theta)\ra-\la\theta,s\circ\mbf{a}(\xi)-\frac{1}{2}s\circ\mbf{a}(\xi)\ra\\
\label{eq:iotaRhoX2}
=&s^*\xi,
\end{align}
where the last line follows from the skew-symmetry of the assignment $\xi,\eta\to\la\xi,s\circ\mbf{a}(\eta)\ra.$
Now, for $\xi\in\h$, we have $\mu_*\rho_R(\xi)=\mbf{a}(\xi)$, and thus \cref{eq:RphiomeDef} implies that
\begin{equation}\label{eq:2FormIsBasic}
\rho_R(\xi)-\iota_{\rho_R(\xi)}\omega+\iota_{\rho_R(\xi)}\mu^*\la\theta,\vartheta_s-\frac{1}{2}s\circ\mbf{a}(\theta)\ra\sim_{R_{\mu,\omega}}\mbf{a}(\xi)+s^*(\xi),\quad \xi\in\h\cap\mf{c}^\perp.\end{equation}
Since \cref{eq:SpExPartRedMMPin} is a morphism of Manin pairs, and since $\mbf{a}(\xi)+s^*(\xi)\in E$ for any $\xi\in\h$, it follows that the left hand side of \cref{eq:2FormIsBasic} lies in $TM$. That is, 
$\iota_{\rho_R(\xi)}\omega_\theta=0$
 for any $\xi\in\h\cap\mf{c}^\perp$.
We conclude that there is a unique 2-form $\tilde\omega_\theta\in\Omega^2(M_{\mf{c},S})$ such that $q_M^*\tilde\omega_\theta=\omega_\theta$.

We define the Courant relation 
$$R_{\mf{c},s,\theta}:=\gr(\mbf{a}\oplus\tilde s_\theta^*)\circ R_\mf{c}\circ \gr(\mbf{a}\oplus s^*)^\top:\mbb{T}_\gamma N\dasharrow \mbb{T}_{\gamma_\theta}N_{\mf{c},S},$$ 
so that 
$$\gr(\mbf{a}\oplus\tilde s_\theta^*)\circ R_{\mf{c},S}=R_{\mf{c},s,\theta}\circ R_{\mu,\omega}\circ\gr(i_M)\circ\gr(q_M)^\top.$$

Now the definition of $s_\theta$ shows that for any $X\in\mf{X}(N_{\mf{c},S})$, we have
$$X^h+\iota_{X^h}\la\theta,\vartheta_s\ra\sim_{R_{\mf{c},s,\theta}} X.$$
Since $X^h$ is horizontal with respect to $\theta$ and $X=(q_N)_*X^h$, we also have
$$X^h+\iota_{X^h}\la\theta,\vartheta_s-\frac{1}{2}s\circ\mbf{a}(\theta)\ra\sim_{R_{\mf{c},s,\theta}} (q_N)_*X^h.$$
On the other hand, for $\xi\in\mf{c}^\perp$, we have $\mbf{a}(\xi)+s^*(\xi)\sim_{R_{\mf{c},s,\theta}} 0$, so
\cref{eq:iotaRhoX2} shows that
$$\mbf{a}(\xi)+\iota_{\mbf{a}(\xi)}\la\theta,\vartheta_s-\frac{1}{2}s\circ\mbf{a}(\theta)\ra\sim_{R_{\mf{c},s,\theta}}(q_N)_*\mbf{a}(\xi).$$
Therefore, for any $X\in TS$,
$$X+\iota_X\la\theta,\vartheta_s-\frac{1}{2}s\circ\mbf{a}(\theta)\ra\sim_{R_{\mf{c},s,\theta}}(q_N)_*X,$$
since this relation holds for both horizontal and vertical vector fields on the bundle $q_N:S\to N_{\mf{c},S}$.
This implies $R_{\mf{c},s,\theta}\circ R_{\mu,\omega}\circ\gr(i_M)=R_{\mu,\omega_\theta}$. Finally, since $\omega_\theta=q_M^*\tilde\omega_\theta$, we conclude that
$$R_{\mf{c},s,\theta}\circ R_{\mu,\omega}\circ\gr(i_M)\circ\gr(q_M)^\top=R_{\tilde\mu,\tilde\omega_\theta},$$
which proves the first part of the proposition.

Now suppose that $s(TS)\subseteq\mf{c}$, then $s^*(\mf{c}^\perp)\subseteq \on{ann}(TS)$, so $\vartheta_s=0$ and $s^*\theta=0$. Thus $\la\theta,\vartheta_s-\frac{1}{2}s\circ\mbf{a}(\theta)\ra=0$, and $\tilde\omega_\theta$ is independent of $\theta$. This concludes the proof of the proposition.
\end{proof}

\subsubsection{Proof of the commutativity of partial reduction}\label{app:ComParRedPrf}

\begin{proof}[Proof of \cref{prop:ComPartRed}]
By symmetry, it is sufficient to prove the first equality in \cref{eq:ComRed}. The key fact is that when $(\mf{c}_1\cap\mf{c}_2,S_1\cap S_2)$ is reductive data then $\mf{c}_1\cap\mf{c}_2$ is coisotropic. As a result, \begin{equation}\label{eq:ComRedSpecRel}\mf{c}_1^\perp\subseteq(\mf{c}_1\cap\mf{c}_2)^\perp\subseteq\mf{c}_1\cap\mf{c}_2\subseteq\mf{c}_2.\end{equation}
Thus, $$\mf{c}_{2,1}=(\mf{c}_1\cap\mf{c}_2)/(\mf{c}_1^\perp\cap\mf{c}_2)=(\mf{c}_1\cap\mf{c}_2)/\mf{c}_1^\perp,$$
and $$\mf{c}_{2,1}^\perp=(\mf{c}_1^\perp+\mf{c}_2^\perp)/\mf{c}_1^\perp.$$
Hence $\xi\in\mf{d}$, $\xi'\in \mf{d}_{\mf{c}_1}$ and $\xi''\in (\mf{d}_{\mf{c}_1})_{\mf{c}_{2\!,\!1}}$ satisfy
$$\xi\sim_{R_{\mf{c}_1}}\xi'\sim_{R_{\mf{c}_{2,1}}}\xi''$$ if and only if  
$$\xi\in\mf{c}_1,\quad \xi'=\xi+\mf{c}_1^\perp,\quad\xi'\in (\mf{c}_1\cap\mf{c}_2)+\mf{c}_1^\perp,\text{ and } \xi''=\xi'+\mf{c}_1^\perp+\mf{c}_2^\perp.$$
Equivalently,
$$\xi\in\mf{c}_1\cap\mf{c}_2,\quad \xi''=\xi+(\mf{c}_1\cap\mf{c}_2)^\perp, \text{ and }\xi'=\xi+\mf{c}_1.$$
So \begin{equation}\label{eq:LieAlgCoisRedCom}R_{\mf{c}_{2,1}}\circ R_{\mf{c}_1}=R_{\mf{c}_1\cap\mf{c}_2}.\end{equation}

Before continuing, we introduce some notation. Let $\mu_{\mf{c}_1,S_1}:M_{\mf{c}_1,S_1}\to N_{\mf{c}_1,S_1}$ be the function whose graph is the support of $R_{\mf{c}_1,S_1}$, let 
\begin{align*}
i_{N_1}&:S_1\to N,&i_{M_1}&:\mu^{-1}(S_1)\to M,\\
i_N&:S_1\cap S_2\to N,& i_M&:\mu^{-1}(S_1\cap S_2)\to M,\\
i_{N_{2\!,\!1}}&:S_{2,1}\to N_{\mf{c}_1,S_1}, & i_{M_{2\!,\!1}}&:\mu_{\mf{c}_1,S_1}^{-1}(S_{2,1})\to M_{\mf{c}_1,S_1}
\end{align*} 
denote the inclusions, and 
\begin{align*}
q_{N_1}&:S_1\to N_{\mf{c}_1,S_1},&q_{M_1}&:\mu^{-1}(S_1)\to M_{\mf{c}_1,S_1},\\
q_N&:S_1\cap S_2\to N_{\mf{c}_1\cap\mf{c}_2,S_1\cap S_2},& q_M&:\mu^{-1}(S_1\cap S_2)\to M_{\mf{c}_1\cap\mf{c}_2,S_1\cap S_2},\\
q_{N_{2\!,\!1}}&:S_{2,1}\to (N_{\mf{c}_1,S_1})_{\mf{c}_{2\!,\!1},S_{2\!,\!1}}, & q_{M_{2\!,\!1}}&:\mu_{\mf{c}_1,S_1}^{-1}(S_{2,1})\to (M_{\mf{c}_1,S_1})_{\mf{c}_{2\!,\!1},S_{2\!,\!1}}
\end{align*} 
denote the quotient maps.

Now \begin{multline*}(R_{\mf{c}_1,S_1})_{\mf{c}_{2\!,\!1},S_{2\!,\!1}}=\bigg(\big(R_{\mf{c}_{2,1}}\circ R_{\mf{c}_1}\big)\times \gr(q_{N_{2\!,\!1}})\circ \gr(i_{N_{2\!,\!1}})^\top\circ  \gr(q_{N_1})\circ \gr(i_{N_1})^\top\bigg)\\\circ R\circ \big(R_{q_{N_{2\!,\!1}}}\circ R_{i_{N_{2\!,\!1}}}^\top\circ  R_{q_{N_1}}\circ R_{i_{N_1}}^\top\big)^\top\end{multline*}

Now  \begin{equation}\label{eq:ComRedComp}x\sim_{  \gr(q_{N_1})\circ \gr(i_{N_1})^\top}y\sim_{\gr(q_{N_{2\!,\!1}})\circ \gr(i_{N_{2\!,\!1}})^\top} z\end{equation} if and only if
$$x\in S_1,\quad y=q_{N_1}(x),\quad y\in S_{2,1},\text{ and } z=q_{N_{2\!,\!1}}(y).$$
\Cref{eq:ComRedSpecRel} implies that the $\mf{c}_1\cap\mf{c}_2$ invariant manifold $S_1\cap S_2$ is also $\mf{c}_1^\perp$ invariant. Since $S_{2,1}$ is the set of $\mf{c}_1^\perp$ orbits in $S_1\cap S_2$, it follows that $y\in S_{2,1}$ if and only if $x\in S_1\cap S_2$. Thus, \cref{eq:ComRedComp} holds if and only if
$$x\in S_1\cap S_2,\quad z=q_{N_{2\!,\!1}}\circ q_{N_1} (x),\text{ and } y=q_{N_1}(x).$$
But 
$$q_{N_{2\!,\!1}}\circ q_{N_1}\rvert_{S_1\cap S_2}:S_1\cap S_2\to N_{\mf{c_1}\cap\mf{c}_2,S_1\cap S_2}$$
 is the quotient map for the $\mf{c}_1^\perp+\mf{c}_2^\perp$ action, i.e. $q_{N_{2\!,\!1}}\circ q_{N_1}\rvert_{S_1\cap S_2}=q_N$. Therefore 
 \begin{equation}\label{eq:RedComIQ1}\gr(q_{N_{2\!,\!1}})\circ \gr(i_{N_{2\!,\!1}})^\top\circ  \gr(q_{N_1})\circ \gr(i_{N_1})^\top=\gr(q_N)\circ \gr(i_N)^\top.\end{equation}
 
  A similar calculation shows that 
  $$\gr(q_{M_{2\!,\!1}})\circ \gr(i_{M_{2\!,\!1}})^\top\circ  \gr(q_{M_1})\circ \gr(i_{M_1})^\top=\gr(q_M)\circ \gr(i_M)^\top,$$ 
  and since the composition is clean we also have
  \begin{equation}\label{eq:RedComIQ2}R_{q_{N_{2\!,\!1}}}\circ R_{i_{N_{2\!,\!1}}}^\top\circ  R_{q_{N_1}}\circ R_{i_{N_1}}^\top=R_{q_M}\circ R_{i_M}^\top\end{equation}
  
  Combining \cref{eq:LieAlgCoisRedCom,eq:RedComIQ1,eq:RedComIQ2} shows that
  $$(R_{\mf{c}_1,S_1})_{\mf{c}_{2\!,\!1},S_{2\!,\!1}}=\big(R_{\mf{c}_1\cap\mf{c}_2}\times (\gr(q_N)\circ \gr(i_N)^\top)\big)\circ R\circ (R_{q_M}\circ R_{i_M}^\top)^\top=R_{\mf{c}_1\cap\mf{c}_2,S_1\cap S_2}.$$
\end{proof}

\bibliography{basicbib}{}
\bibliographystyle{amsplain.bst}

\end{document}

%% file: TriangInc.pdf_tex
\begingroup%
  \makeatletter%
  \providecommand\color[2][]{%
    \errmessage{(Inkscape) Color is used for the text in Inkscape, but the package 'color.sty' is not loaded}%
    \renewcommand\color[2][]{}%
  }%
  \providecommand\transparent[1]{%
    \errmessage{(Inkscape) Transparency is used (non-zero) for the text in Inkscape, but the package 'transparent.sty' is not loaded}%
    \renewcommand\transparent[1]{}%
  }%
  \providecommand\rotatebox[2]{#2}%
  \ifx\svgwidth\undefined%
    \setlength{\unitlength}{139.76153564bp}%
    \ifx\svgscale\undefined%
      \relax%
    \else%
      \setlength{\unitlength}{\unitlength * \real{\svgscale}}%
    \fi%
  \else%
    \setlength{\unitlength}{\svgwidth}%
  \fi%
  \global\let\svgwidth\undefined%
  \global\let\svgscale\undefined%
  \makeatother%
  \begin{picture}(1,0.95633082)%
    \put(0,0){\includegraphics[width=\unitlength]{TriangInc.pdf}}%
    \put(0.71683457,0.67339215){\color[rgb]{0,0,0}\makebox(0,0)[lb]{\smash{$v_2$}}}%
    \put(0.48269536,0.54592522){\color[rgb]{0,0,0}\rotatebox{36.46421951}{\makebox(0,0)[b]{\smash{$e_2$}}}}%
    \put(0.76144799,0.47400266){\color[rgb]{0,0,0}\rotatebox{-78.06389987}{\makebox(0,0)[b]{\smash{$e_1$}}}}%
    \put(0.51388701,0.24965639){\color[rgb]{0,0,0}\rotatebox{-15.62903885}{\makebox(0,0)[b]{\smash{$e_3$}}}}%
    \put(0.30724605,0.39229507){\color[rgb]{0,0,0}\makebox(0,0)[rb]{\smash{$v_3$}}}%
    \put(0.81723272,0.23614011){\color[rgb]{0,0,0}\makebox(0,0)[lb]{\smash{$v_1$}}}%
    \put(0.5923668,0.42456631){\color[rgb]{0,0,0}\makebox(0,0)[b]{\smash{$t$}}}%
  \end{picture}%
\endgroup%

%% file: SurfaceWBoundInt.pdf_tex
\begingroup%
  \makeatletter%
  \providecommand\color[2][]{%
    \errmessage{(Inkscape) Color is used for the text in Inkscape, but the package 'color.sty' is not loaded}%
    \renewcommand\color[2][]{}%
  }%
  \providecommand\transparent[1]{%
    \errmessage{(Inkscape) Transparency is used (non-zero) for the text in Inkscape, but the package 'transparent.sty' is not loaded}%
    \renewcommand\transparent[1]{}%
  }%
  \providecommand\rotatebox[2]{#2}%
  \ifx\svgwidth\undefined%
    \setlength{\unitlength}{350.71945801bp}%
    \ifx\svgscale\undefined%
      \relax%
    \else%
      \setlength{\unitlength}{\unitlength * \real{\svgscale}}%
    \fi%
  \else%
    \setlength{\unitlength}{\svgwidth}%
  \fi%
  \global\let\svgwidth\undefined%
  \global\let\svgscale\undefined%
  \makeatother%
  \begin{picture}(1,0.44585401)%
    \put(0,0){\includegraphics[width=\unitlength]{SurfaceWBoundInt.pdf}}%
    \put(0.6555282,0.19912677){\color[rgb]{0,0,0}\makebox(0,0)[b]{\smash{$v_2$}}}%
    \put(0.45304076,0.18103397){\color[rgb]{0,0,0}\makebox(0,0)[b]{\smash{$v_3$}}}%
    \put(0.8520024,0.2733146){\color[rgb]{0,0,0}\makebox(0,0)[rb]{\smash{$v_1$}}}%
    \put(0.12319319,0.26909963){\color[rgb]{0,0,0}\makebox(0,0)[rb]{\smash{$\Sigma=$}}}%
    \put(0.91008438,0.21821438){\color[rgb]{0,0,0}\makebox(0,0)[lb]{\smash{$e_1$}}}%
    \put(0.54851039,0.19806687){\color[rgb]{0,0,0}\makebox(0,0)[b]{\smash{$e_2$}}}%
    \put(0.54966187,0.04606762){\color[rgb]{0,0,0}\makebox(0,0)[b]{\smash{$e_3$}}}%
  \end{picture}%
\endgroup%

%% file: twoTriang.pdf_tex
\begingroup%
  \makeatletter%
  \providecommand\color[2][]{%
    \errmessage{(Inkscape) Color is used for the text in Inkscape, but the package 'color.sty' is not loaded}%
    \renewcommand\color[2][]{}%
  }%
  \providecommand\transparent[1]{%
    \errmessage{(Inkscape) Transparency is used (non-zero) for the text in Inkscape, but the package 'transparent.sty' is not loaded}%
    \renewcommand\transparent[1]{}%
  }%
  \providecommand\rotatebox[2]{#2}%
  \ifx\svgwidth\undefined%
    \setlength{\unitlength}{119.3527832bp}%
    \ifx\svgscale\undefined%
      \relax%
    \else%
      \setlength{\unitlength}{\unitlength * \real{\svgscale}}%
    \fi%
  \else%
    \setlength{\unitlength}{\svgwidth}%
  \fi%
  \global\let\svgwidth\undefined%
  \global\let\svgscale\undefined%
  \makeatother%
  \begin{picture}(1,1.49233148)%
    \put(0,0){\includegraphics[width=\unitlength]{twoTriang.pdf}}%
    \put(0.39773979,1.21897215){\color[rgb]{0,0,0}\rotatebox{1.43928041}{\makebox(0,0)[b]{\smash{$G$}}}}%
    \put(0.53903506,1.08153094){\color[rgb]{0,0,0}\makebox(0,0)[b]{\smash{$t_1$}}}%
    \put(0.34731884,0.47793117){\color[rgb]{0,0,0}\rotatebox{-0.91143266}{\makebox(0,0)[b]{\smash{$t_2$}}}}%
    \put(0.21730555,0.9388681){\color[rgb]{0,0,0}\rotatebox{-0.43259527}{\makebox(0,0)[b]{\smash{$\mathfrak g$}}}}%
    \put(0.77093569,1.10339685){\color[rgb]{0,0,0}\rotatebox{1.43928041}{\makebox(0,0)[b]{\smash{$G$}}}}%
    \put(0.45887667,0.86087265){\color[rgb]{0,0,0}\rotatebox{1.43928041}{\makebox(0,0)[b]{\smash{$G$}}}}%
    \put(0.39951764,0.68788344){\color[rgb]{0,0,0}\rotatebox{1.43928041}{\makebox(0,0)[b]{\smash{$G$}}}}%
    \put(0.4860123,0.29272162){\color[rgb]{0,0,0}\rotatebox{1.43928041}{\makebox(0,0)[b]{\smash{$G$}}}}%
    \put(0.08576263,0.39448008){\color[rgb]{0,0,0}\rotatebox{1.43928041}{\makebox(0,0)[b]{\smash{$G$}}}}%
    \put(0.67013045,0.60483338){\color[rgb]{0,0,0}\rotatebox{-0.43259527}{\makebox(0,0)[b]{\smash{$\mathfrak g$}}}}%
    \put(0.2919284,0.13165685){\color[rgb]{0,0,0}\rotatebox{-0.43259527}{\makebox(0,0)[b]{\smash{$\mathfrak g$}}}}%
    \put(0.04431631,0.62857702){\color[rgb]{0,0,0}\rotatebox{-0.43259527}{\makebox(0,0)[b]{\smash{$\mathfrak g$}}}}%
    \put(0.81259214,0.89993275){\color[rgb]{0,0,0}\rotatebox{-0.43259527}{\makebox(0,0)[b]{\smash{$\mathfrak g$}}}}%
    \put(0.56328413,1.35784552){\color[rgb]{0,0,0}\rotatebox{-0.43259527}{\makebox(0,0)[b]{\smash{$\mathfrak g$}}}}%
    \put(0.68200216,0.80834029){\color[rgb]{0,0,0}\rotatebox{-0.43259527}{\makebox(0,0)[b]{\smash{$\bar{\mathfrak g}$}}}}%
    \put(0.73288139,1.30017252){\color[rgb]{0,0,0}\rotatebox{-0.43259527}{\makebox(0,0)[b]{\smash{$\bar{\mathfrak g}$}}}}%
    \put(0.19525789,1.08647991){\color[rgb]{0,0,0}\rotatebox{-0.43259527}{\makebox(0,0)[b]{\smash{$\bar{\mathfrak g}$}}}}%
    \put(0.18169014,0.74558931){\color[rgb]{0,0,0}\rotatebox{-0.43259527}{\makebox(0,0)[b]{\smash{$\bar{\mathfrak g}$}}}}%
    \put(0.12741898,0.19948597){\color[rgb]{0,0,0}\rotatebox{-0.43259527}{\makebox(0,0)[b]{\smash{$\bar{\mathfrak g}$}}}}%
    \put(0.68200219,0.46744968){\color[rgb]{0,0,0}\rotatebox{-0.43259527}{\makebox(0,0)[b]{\smash{$\bar{\mathfrak g}$}}}}%
  \end{picture}%
\endgroup%

%% file: twoTriangH.pdf_tex
\begingroup%
  \makeatletter%
  \providecommand\color[2][]{%
    \errmessage{(Inkscape) Color is used for the text in Inkscape, but the package 'color.sty' is not loaded}%
    \renewcommand\color[2][]{}%
  }%
  \providecommand\transparent[1]{%
    \errmessage{(Inkscape) Transparency is used (non-zero) for the text in Inkscape, but the package 'transparent.sty' is not loaded}%
    \renewcommand\transparent[1]{}%
  }%
  \providecommand\rotatebox[2]{#2}%
  \ifx\svgwidth\undefined%
    \setlength{\unitlength}{119.3527832bp}%
    \ifx\svgscale\undefined%
      \relax%
    \else%
      \setlength{\unitlength}{\unitlength * \real{\svgscale}}%
    \fi%
  \else%
    \setlength{\unitlength}{\svgwidth}%
  \fi%
  \global\let\svgwidth\undefined%
  \global\let\svgscale\undefined%
  \makeatother%
  \begin{picture}(1,1.49233148)%
    \put(0,0){\includegraphics[width=\unitlength]{twoTriangH.pdf}}%
    \put(0.53903506,1.08153094){\color[rgb]{0,0,0}\makebox(0,0)[b]{\smash{$t_1$}}}%
    \put(0.34731884,0.47793117){\color[rgb]{0,0,0}\rotatebox{-0.91143266}{\makebox(0,0)[b]{\smash{$t_2$}}}}%
    \put(0.21730555,0.9388681){\color[rgb]{0,0,0}\rotatebox{-0.43259527}{\makebox(0,0)[b]{\smash{$\xi_1'$}}}}%
    \put(0.68353602,0.61823903){\color[rgb]{0,0,0}\rotatebox{-0.43259527}{\makebox(0,0)[b]{\smash{$\xi_2''$}}}}%
    \put(0.2919284,0.13165685){\color[rgb]{0,0,0}\rotatebox{-0.43259527}{\makebox(0,0)[b]{\smash{$\xi_2'$}}}}%
    \put(0.04431631,0.62857702){\color[rgb]{0,0,0}\rotatebox{-0.43259527}{\makebox(0,0)[b]{\smash{$\xi_2$}}}}%
    \put(0.82599775,0.89993274){\color[rgb]{0,0,0}\rotatebox{-0.43259527}{\makebox(0,0)[b]{\smash{$\xi_1''$}}}}%
    \put(0.56328413,1.35784552){\color[rgb]{0,0,0}\rotatebox{-0.43259527}{\makebox(0,0)[b]{\smash{$\xi_1$}}}}%
    \put(0.66859659,0.79493464){\color[rgb]{0,0,0}\rotatebox{-0.43259527}{\makebox(0,0)[b]{\smash{$\xi_1''$}}}}%
    \put(0.73288139,1.30017252){\color[rgb]{0,0,0}\rotatebox{-0.43259527}{\makebox(0,0)[b]{\smash{$\xi_1$}}}}%
    \put(0.1952579,1.09988552){\color[rgb]{0,0,0}\rotatebox{-0.43259527}{\makebox(0,0)[b]{\smash{$\xi_1'$}}}}%
    \put(0.18169015,0.75899492){\color[rgb]{0,0,0}\rotatebox{-0.43259527}{\makebox(0,0)[b]{\smash{$\xi_2$}}}}%
    \put(0.12741898,0.19948597){\color[rgb]{0,0,0}\rotatebox{-0.43259527}{\makebox(0,0)[b]{\smash{$\xi_2'$}}}}%
    \put(0.69540784,0.45404411){\color[rgb]{0,0,0}\rotatebox{-0.43259527}{\makebox(0,0)[b]{\smash{$\xi_2''$}}}}%
  \end{picture}%
\endgroup%

%% file: twoTriangSl.pdf_tex
\begingroup%
  \makeatletter%
  \providecommand\color[2][]{%
    \errmessage{(Inkscape) Color is used for the text in Inkscape, but the package 'color.sty' is not loaded}%
    \renewcommand\color[2][]{}%
  }%
  \providecommand\transparent[1]{%
    \errmessage{(Inkscape) Transparency is used (non-zero) for the text in Inkscape, but the package 'transparent.sty' is not loaded}%
    \renewcommand\transparent[1]{}%
  }%
  \providecommand\rotatebox[2]{#2}%
  \ifx\svgwidth\undefined%
    \setlength{\unitlength}{119.3527832bp}%
    \ifx\svgscale\undefined%
      \relax%
    \else%
      \setlength{\unitlength}{\unitlength * \real{\svgscale}}%
    \fi%
  \else%
    \setlength{\unitlength}{\svgwidth}%
  \fi%
  \global\let\svgwidth\undefined%
  \global\let\svgscale\undefined%
  \makeatother%
  \begin{picture}(1,1.49233148)%
    \put(0,0){\includegraphics[width=\unitlength]{twoTriangSl.pdf}}%
    \put(0.53903506,1.08153094){\color[rgb]{0,0,0}\makebox(0,0)[b]{\smash{$t_1$}}}%
    \put(0.34731884,0.47793117){\color[rgb]{0,0,0}\rotatebox{-0.91143266}{\makebox(0,0)[b]{\smash{$t_2$}}}}%
    \put(0.21730555,0.9388681){\color[rgb]{0,0,0}\rotatebox{-0.43259527}{\makebox(0,0)[b]{\smash{$\xi$}}}}%
    \put(0.45887669,0.84746698){\color[rgb]{0,0,0}\rotatebox{1.43928041}{\makebox(0,0)[b]{\smash{$g^{-1}$}}}}%
    \put(0.39951764,0.68788344){\color[rgb]{0,0,0}\rotatebox{1.43928041}{\makebox(0,0)[b]{\smash{$g$}}}}%
    \put(0.67013045,0.60483338){\color[rgb]{0,0,0}\rotatebox{-0.43259527}{\makebox(0,0)[b]{\smash{$\eta$}}}}%
    \put(0.68200216,0.80834029){\color[rgb]{0,0,0}\rotatebox{-0.43259527}{\makebox(0,0)[b]{\smash{$\eta$}}}}%
    \put(0.18169014,0.74558931){\color[rgb]{0,0,0}\rotatebox{-0.43259527}{\makebox(0,0)[b]{\smash{$\xi$}}}}%
  \end{picture}%
\endgroup%

%% file: DoubleSymplGroupoid.pdf_tex
\begingroup%
  \makeatletter%
  \providecommand\color[2][]{%
    \errmessage{(Inkscape) Color is used for the text in Inkscape, but the package 'color.sty' is not loaded}%
    \renewcommand\color[2][]{}%
  }%
  \providecommand\transparent[1]{%
    \errmessage{(Inkscape) Transparency is used (non-zero) for the text in Inkscape, but the package 'transparent.sty' is not loaded}%
    \renewcommand\transparent[1]{}%
  }%
  \providecommand\rotatebox[2]{#2}%
  \ifx\svgwidth\undefined%
    \setlength{\unitlength}{192.55683594bp}%
    \ifx\svgscale\undefined%
      \relax%
    \else%
      \setlength{\unitlength}{\unitlength * \real{\svgscale}}%
    \fi%
  \else%
    \setlength{\unitlength}{\svgwidth}%
  \fi%
  \global\let\svgwidth\undefined%
  \global\let\svgscale\undefined%
  \makeatother%
  \begin{picture}(1,0.4129177)%
    \put(0,0){\includegraphics[width=\unitlength]{DoubleSymplGroupoid.pdf}}%
    \put(0.70443998,0.23322644){\color[rgb]{0.7254902,0,0}\makebox(0,0)[lb]{\smash{$\mf{e}$}}}%
    \put(0.70443998,0.16121309){\color[rgb]{0,0,0.7254902}\makebox(0,0)[lb]{\smash{$\mf{f}$}}}%
    \put(0.34351117,0.36556967){\color[rgb]{0,0,0.7254902}\makebox(0,0)[b]{\smash{$f_1$}}}%
    \put(0.34351117,0.01469737){\color[rgb]{0,0,0.7254902}\makebox(0,0)[b]{\smash{$f_2$}}}%
    \put(0.49340451,0.20331189){\color[rgb]{0.7254902,0,0}\makebox(0,0)[lb]{\smash{$e_2$}}}%
    \put(0.18336713,0.20331189){\color[rgb]{0.7254902,0,0}\makebox(0,0)[rb]{\smash{$e_1$}}}%
  \end{picture}%
\endgroup%

%% file: DoubleSymplGroupDbl.pdf_tex
\begingroup%
  \makeatletter%
  \providecommand\color[2][]{%
    \errmessage{(Inkscape) Color is used for the text in Inkscape, but the package 'color.sty' is not loaded}%
    \renewcommand\color[2][]{}%
  }%
  \providecommand\transparent[1]{%
    \errmessage{(Inkscape) Transparency is used (non-zero) for the text in Inkscape, but the package 'transparent.sty' is not loaded}%
    \renewcommand\transparent[1]{}%
  }%
  \providecommand\rotatebox[2]{#2}%
  \ifx\svgwidth\undefined%
    \setlength{\unitlength}{368.70725098bp}%
    \ifx\svgscale\undefined%
      \relax%
    \else%
      \setlength{\unitlength}{\unitlength * \real{\svgscale}}%
    \fi%
  \else%
    \setlength{\unitlength}{\svgwidth}%
  \fi%
  \global\let\svgwidth\undefined%
  \global\let\svgscale\undefined%
  \makeatother%
  \begin{picture}(1,0.1918643)%
    \put(0,0){\includegraphics[width=\unitlength]{DoubleSymplGroupDbl.pdf}}%
    \put(0.84564421,0.10193676){\color[rgb]{0.7254902,0,0}\makebox(0,0)[lb]{\smash{$\mf{e}$}}}%
    \put(0.84564421,0.06432789){\color[rgb]{0,0,0.7254902}\makebox(0,0)[lb]{\smash{$\mf{f}$}}}%
    \put(0.17197634,0.02947709){\color[rgb]{0,0,0.7254902}\makebox(0,0)[b]{\smash{$f_1$}}}%
    \put(0.1698066,0.11777521){\color[rgb]{0.7254902,0,0}\makebox(0,0)[b]{\smash{$e_2$}}}%
    \put(0.1698066,0.16713686){\color[rgb]{0.7254902,0,0}\makebox(0,0)[b]{\smash{$e_1$}}}%
    \put(0.06495782,0.08970661){\color[rgb]{0,0,0}\makebox(0,0)[b]{\smash{$g$}}}%
    \put(0.16926416,0.07459997){\color[rgb]{0,0,0.7254902}\makebox(0,0)[b]{\smash{$f_2$}}}%
    \put(0.48086993,0.08818043){\color[rgb]{0,0,0}\makebox(0,0)[rb]{\smash{$g$}}}%
    \put(0.74751803,0.08818043){\color[rgb]{0,0,0}\makebox(0,0)[lb]{\smash{$g$}}}%
    \put(0.67543837,0.16049438){\color[rgb]{0,0,0.7254902}\makebox(0,0)[b]{\smash{$f_1$}}}%
    \put(0.67543837,0.00077764){\color[rgb]{0,0,0.7254902}\makebox(0,0)[b]{\smash{$f_2$}}}%
    \put(0.55793362,0.09080556){\color[rgb]{0,0,0}\rotatebox{-45}{\makebox(0,0)[b]{\smash{$e_2g$}}}}%
    \put(0.68877564,0.08646609){\color[rgb]{0,0,0}\rotatebox{-45}{\makebox(0,0)[b]{\smash{$f_1g^{-1}$}}}}%
    \put(0.54959328,0.16049438){\color[rgb]{0.7254902,0,0}\makebox(0,0)[b]{\smash{$e_1$}}}%
    \put(0.54959328,0.00077764){\color[rgb]{0.7254902,0,0}\makebox(0,0)[b]{\smash{$e_2$}}}%
  \end{picture}%
\endgroup%

%% file: LuYakimov.pdf_tex
\begingroup%
  \makeatletter%
  \providecommand\color[2][]{%
    \errmessage{(Inkscape) Color is used for the text in Inkscape, but the package 'color.sty' is not loaded}%
    \renewcommand\color[2][]{}%
  }%
  \providecommand\transparent[1]{%
    \errmessage{(Inkscape) Transparency is used (non-zero) for the text in Inkscape, but the package 'transparent.sty' is not loaded}%
    \renewcommand\transparent[1]{}%
  }%
  \providecommand\rotatebox[2]{#2}%
  \ifx\svgwidth\undefined%
    \setlength{\unitlength}{364.41606445bp}%
    \ifx\svgscale\undefined%
      \relax%
    \else%
      \setlength{\unitlength}{\unitlength * \real{\svgscale}}%
    \fi%
  \else%
    \setlength{\unitlength}{\svgwidth}%
  \fi%
  \global\let\svgwidth\undefined%
  \global\let\svgscale\undefined%
  \makeatother%
  \begin{picture}(1,0.2742858)%
    \put(0,0){\includegraphics[width=\unitlength]{LuYakimov.pdf}}%
    \put(0.82187363,0.16346223){\color[rgb]{0.7254902,0,0}\makebox(0,0)[lb]{\smash{$\mf{e}$}}}%
    \put(0.82187363,0.12748094){\color[rgb]{0,0,0.7254902}\makebox(0,0)[lb]{\smash{$\mf{f}$}}}%
    \put(0.82187363,0.09149979){\color[rgb]{0,0.68627451,0}\makebox(0,0)[lb]{\smash{$\mf{c}$}}}%
    \put(0.15713267,0.23565379){\color[rgb]{0.7254902,0,0}\makebox(0,0)[b]{\smash{$e$}}}%
    \put(0.15713267,0.02517941){\color[rgb]{0,0,0.7254902}\makebox(0,0)[b]{\smash{$f$}}}%
    \put(0.08746447,0.14206111){\color[rgb]{0,0,0}\makebox(0,0)[b]{\smash{$g$}}}%
    \put(0.20228839,0.1358216){\color[rgb]{0,0.68627451,0}\makebox(0,0)[lb]{\smash{$c$}}}%
    \put(0.5107592,0.20829141){\color[rgb]{0.7254902,0,0}\makebox(0,0)[b]{\smash{$e$}}}%
    \put(0.65447136,0.20829141){\color[rgb]{0,0,0.7254902}\makebox(0,0)[b]{\smash{$f$}}}%
    \put(0.5806551,0.03592809){\color[rgb]{0,0.7254902,0}\makebox(0,0)[b]{\smash{$c$}}}%
    \put(0.7196407,0.12648888){\color[rgb]{0,0,0}\makebox(0,0)[lb]{\smash{$g$}}}%
    \put(0.43479438,0.12648888){\color[rgb]{0,0,0}\makebox(0,0)[rb]{\smash{$g$}}}%
    \put(0.50106862,0.12846279){\color[rgb]{0,0,0}\rotatebox{45}{\makebox(0,0)[b]{\smash{$g^{-1}e$}}}}%
    \put(0.72300546,0.04074314){\color[rgb]{0,0,0}\makebox(0,0)[b]{\smash{$v_0$}}}%
    \put(0.43322679,0.04074316){\color[rgb]{0,0,0}\makebox(0,0)[b]{\smash{$v_0$}}}%
  \end{picture}%
\endgroup%

%% file: ColoredSurface2.pdf_tex
\begingroup%
  \makeatletter%
  \providecommand\color[2][]{%
    \errmessage{(Inkscape) Color is used for the text in Inkscape, but the package 'color.sty' is not loaded}%
    \renewcommand\color[2][]{}%
  }%
  \providecommand\transparent[1]{%
    \errmessage{(Inkscape) Transparency is used (non-zero) for the text in Inkscape, but the package 'transparent.sty' is not loaded}%
    \renewcommand\transparent[1]{}%
  }%
  \providecommand\rotatebox[2]{#2}%
  \ifx\svgwidth\undefined%
    \setlength{\unitlength}{197.95bp}%
    \ifx\svgscale\undefined%
      \relax%
    \else%
      \setlength{\unitlength}{\unitlength * \real{\svgscale}}%
    \fi%
  \else%
    \setlength{\unitlength}{\svgwidth}%
  \fi%
  \global\let\svgwidth\undefined%
  \global\let\svgscale\undefined%
  \makeatother%
  \begin{picture}(1,0.37877172)%
    \put(0,0){\includegraphics[width=\unitlength]{ColoredSurface2.pdf}}%
    \put(0.39531748,0.17912048){\color[rgb]{0,0,0}\makebox(0,0)[b]{\smash{$\mf{g}_1$}}}%
    \put(0.56277877,0.22877471){\color[rgb]{0,0,0}\makebox(0,0)[b]{\smash{$\mf{g}_2$}}}%
    \put(0.6190086,0.1158357){\color[rgb]{0,0,0}\makebox(0,0)[b]{\smash{$\mf{g}_3$}}}%
    \put(0.28568503,0.39263549){\color[rgb]{0,0,0}\makebox(0,0)[lb]{\smash{$\mf{c}_{1,2}\subset\mf{g}_1\oplus\ol{\mf{g}_2}$}}}%
    \put(0.81435637,0.05157731){\color[rgb]{0,0,0}\makebox(0,0)[lb]{\smash{$\mf{c}_{2,3}\subset\mf{g}_2\oplus\ol{\mf{g}_3}$}}}%
    \put(0.98473847,0.27093398){\color[rgb]{0,0,0}\makebox(0,0)[lb]{\smash{$\mf{c}_2\subset\mf{g}_2$}}}%
    \put(-0.01866985,0.09246901){\color[rgb]{0,0,0}\makebox(0,0)[lb]{\smash{$\mf{c}_1\subset\mf{g}_1$}}}%
  \end{picture}%
\endgroup%

%% file: PBoalchDrawCut.pdf_tex
\begingroup%
  \makeatletter%
  \providecommand\color[2][]{%
    \errmessage{(Inkscape) Color is used for the text in Inkscape, but the package 'color.sty' is not loaded}%
    \renewcommand\color[2][]{}%
  }%
  \providecommand\transparent[1]{%
    \errmessage{(Inkscape) Transparency is used (non-zero) for the text in Inkscape, but the package 'transparent.sty' is not loaded}%
    \renewcommand\transparent[1]{}%
  }%
  \providecommand\rotatebox[2]{#2}%
  \ifx\svgwidth\undefined%
    \setlength{\unitlength}{400.46674805bp}%
    \ifx\svgscale\undefined%
      \relax%
    \else%
      \setlength{\unitlength}{\unitlength * \real{\svgscale}}%
    \fi%
  \else%
    \setlength{\unitlength}{\svgwidth}%
  \fi%
  \global\let\svgwidth\undefined%
  \global\let\svgscale\undefined%
  \makeatother%
  \begin{picture}(1,0.43901284)%
    \put(0,0){\includegraphics[width=\unitlength]{PBoalchDrawCut.pdf}}%
    \put(0.97077051,0.27238128){\color[rgb]{0,0,0}\rotatebox{90}{\makebox(0,0)[b]{\smash{$C_0$}}}}%
    \put(0.9045633,0.27000167){\color[rgb]{0,0,0}\rotatebox{67.12569648}{\makebox(0,0)[b]{\smash{$C_1$}}}}%
    \put(0.85561496,0.26841056){\color[rgb]{0,0,0}\rotatebox{49.56876216}{\makebox(0,0)[b]{\smash{$C_2$}}}}%
    \put(0.59726803,0.24272625){\color[rgb]{0,0,0}\rotatebox{21.87670524}{\makebox(0,0)[b]{\smash{$hC_{2r}$}}}}%
    \put(0.66752893,0.24285279){\color[rgb]{0,0,0}\rotatebox{24.8581464}{\makebox(0,0)[b]{\smash{$C_{2r-1}$}}}}%
    \put(0.74170126,0.24307219){\color[rgb]{0,0,0}\rotatebox{30.214021}{\makebox(0,0)[b]{\smash{$C_{2r-2}$}}}}%
    \put(0.25690022,0.25896829){\color[rgb]{0,0,0}\makebox(0,0)[b]{\smash{$\mf{h}$}}}%
    \put(0.27212902,0.34065521){\color[rgb]{0,0,0}\makebox(0,0)[b]{\smash{$\mf{g}$}}}%
    \put(0.50515272,0.27238128){\color[rgb]{0,0,0}\rotatebox{90}{\makebox(0,0)[b]{\smash{$C_0$}}}}%
    \put(0.50499178,0.16173475){\color[rgb]{0,0,0}\rotatebox{90}{\makebox(0,0)[b]{\smash{$\Ad_hh_0$}}}}%
    \put(0.61979604,0.16403784){\color[rgb]{0,0,0}\rotatebox{-21.85426076}{\makebox(0,0)[b]{\smash{$hh_0$}}}}%
    \put(0.687472,0.15615431){\color[rgb]{0,0,0}\rotatebox{-24.41246982}{\makebox(0,0)[b]{\smash{$h_{2r-1}$}}}}%
    \put(0.74283998,0.15645648){\color[rgb]{0,0,0}\rotatebox{-30.55405644}{\makebox(0,0)[b]{\smash{$h_{2r-2}$}}}}%
    \put(0.84109975,0.15790038){\color[rgb]{0,0,0}\rotatebox{-51.31545983}{\makebox(0,0)[b]{\smash{$h_2$}}}}%
    \put(0.8921767,0.15940121){\color[rgb]{0,0,0}\rotatebox{-67.68393333}{\makebox(0,0)[b]{\smash{$h_1$}}}}%
    \put(0.96014326,0.161725){\color[rgb]{0,0,0}\rotatebox{-90}{\makebox(0,0)[b]{\smash{$\Ad_hh_0$}}}}%
    \put(0.7345784,0.02258804){\color[rgb]{0,0,0}\makebox(0,0)[b]{\smash{$h$}}}%
    \put(0.19084801,0.3455303){\color[rgb]{0,0,0}\makebox(0,0)[b]{\smash{$v_6$}}}%
    \put(0.07675666,0.28246117){\color[rgb]{0,0,0}\makebox(0,0)[b]{\smash{$v_1$}}}%
    \put(0.07108752,0.12337104){\color[rgb]{0,0,0}\makebox(0,0)[b]{\smash{$v_2$}}}%
    \put(0.2057295,0.05073525){\color[rgb]{0,0,0}\makebox(0,0)[b]{\smash{$v_3$}}}%
    \put(0.33541099,0.12230809){\color[rgb]{0,0,0}\makebox(0,0)[b]{\smash{$v_4$}}}%
    \put(0.334348,0.27998092){\color[rgb]{0,0,0}\makebox(0,0)[b]{\smash{$v_5$}}}%
    \put(0.20524967,0.4138469){\color[rgb]{0,0,0}\makebox(0,0)[b]{\smash{$x_G$}}}%
    \put(0.2040889,0.22277009){\color[rgb]{0,0,0}\makebox(0,0)[b]{\smash{$x_H$}}}%
    \put(0.96036853,0.40186088){\color[rgb]{0,0,0}\makebox(0,0)[b]{\smash{$x_G$}}}%
    \put(0.95521243,0.01900783){\color[rgb]{0,0,0}\makebox(0,0)[b]{\smash{$x_H$}}}%
  \end{picture}%
\endgroup%

%% file: ColoredBranched.pdf_tex
\begingroup%
  \makeatletter%
  \providecommand\color[2][]{%
    \errmessage{(Inkscape) Color is used for the text in Inkscape, but the package 'color.sty' is not loaded}%
    \renewcommand\color[2][]{}%
  }%
  \providecommand\transparent[1]{%
    \errmessage{(Inkscape) Transparency is used (non-zero) for the text in Inkscape, but the package 'transparent.sty' is not loaded}%
    \renewcommand\transparent[1]{}%
  }%
  \providecommand\rotatebox[2]{#2}%
  \ifx\svgwidth\undefined%
    \setlength{\unitlength}{288.49775391bp}%
    \ifx\svgscale\undefined%
      \relax%
    \else%
      \setlength{\unitlength}{\unitlength * \real{\svgscale}}%
    \fi%
  \else%
    \setlength{\unitlength}{\svgwidth}%
  \fi%
  \global\let\svgwidth\undefined%
  \global\let\svgscale\undefined%
  \makeatother%
  \begin{picture}(1,0.27999128)%
    \put(0,0){\includegraphics[width=\unitlength]{ColoredBranched.pdf}}%
    \put(0.13478189,0.20808431){\color[rgb]{0,0,0}\makebox(0,0)[b]{\smash{$\mf{g}_2$}}}%
    \put(0.72593051,0.15369864){\color[rgb]{0,0,0}\makebox(0,0)[b]{\smash{$\mf{g}_1$}}}%
    \put(0.1655216,0.08748998){\color[rgb]{0,0,0}\makebox(0,0)[b]{\smash{$\mf{g}_3$}}}%
    \put(0.55567969,0.25892309){\color[rgb]{0,0,0}\makebox(0,0)[b]{\smash{$\mf{c}$}}}%
  \end{picture}%
\endgroup%

%% file: BoalchFission.pdf_tex
\begingroup%
  \makeatletter%
  \providecommand\color[2][]{%
    \errmessage{(Inkscape) Color is used for the text in Inkscape, but the package 'color.sty' is not loaded}%
    \renewcommand\color[2][]{}%
  }%
  \providecommand\transparent[1]{%
    \errmessage{(Inkscape) Transparency is used (non-zero) for the text in Inkscape, but the package 'transparent.sty' is not loaded}%
    \renewcommand\transparent[1]{}%
  }%
  \providecommand\rotatebox[2]{#2}%
  \ifx\svgwidth\undefined%
    \setlength{\unitlength}{234.75227051bp}%
    \ifx\svgscale\undefined%
      \relax%
    \else%
      \setlength{\unitlength}{\unitlength * \real{\svgscale}}%
    \fi%
  \else%
    \setlength{\unitlength}{\svgwidth}%
  \fi%
  \global\let\svgwidth\undefined%
  \global\let\svgscale\undefined%
  \makeatother%
  \begin{picture}(1,0.88562182)%
    \put(0,0){\includegraphics[width=\unitlength]{BoalchFission.pdf}}%
    \put(0.55682215,0.02204808){\color[rgb]{0,0,0}\makebox(0,0)[b]{\smash{$G$}}}%
    \put(0.0680904,0.67369041){\color[rgb]{0,0,0}\makebox(0,0)[b]{\smash{$H_1$}}}%
    \put(0.10463109,0.80615041){\color[rgb]{0,0,0}\makebox(0,0)[b]{\smash{$H_2$}}}%
    \put(0.32387524,0.84269111){\color[rgb]{0,0,0}\makebox(0,0)[b]{\smash{$H_n$}}}%
  \end{picture}%
\endgroup%

%% file: Branched3.pdf_tex
\begingroup%
  \makeatletter%
  \providecommand\color[2][]{%
    \errmessage{(Inkscape) Color is used for the text in Inkscape, but the package 'color.sty' is not loaded}%
    \renewcommand\color[2][]{}%
  }%
  \providecommand\transparent[1]{%
    \errmessage{(Inkscape) Transparency is used (non-zero) for the text in Inkscape, but the package 'transparent.sty' is not loaded}%
    \renewcommand\transparent[1]{}%
  }%
  \providecommand\rotatebox[2]{#2}%
  \ifx\svgwidth\undefined%
    \setlength{\unitlength}{288bp}%
    \ifx\svgscale\undefined%
      \relax%
    \else%
      \setlength{\unitlength}{\unitlength * \real{\svgscale}}%
    \fi%
  \else%
    \setlength{\unitlength}{\svgwidth}%
  \fi%
  \global\let\svgwidth\undefined%
  \global\let\svgscale\undefined%
  \makeatother%
  \begin{picture}(1,0.55039656)%
    \put(0,0){\includegraphics[width=\unitlength]{Branched3.pdf}}%
    \put(0.69267546,0.07258886){\color[rgb]{0,0,0}\makebox(0,0)[b]{\smash{$=$}}}%
    \put(0.31871442,0.07258886){\color[rgb]{0,0,0}\makebox(0,0)[b]{\smash{$=$}}}%
    \put(0.4964922,0.25592217){\color[rgb]{0,0,0}\makebox(0,0)[b]{\smash{$=$}}}%
  \end{picture}%
\endgroup%

%% file: TwoSheetDict.pdf_tex
\begingroup%
  \makeatletter%
  \providecommand\color[2][]{%
    \errmessage{(Inkscape) Color is used for the text in Inkscape, but the package 'color.sty' is not loaded}%
    \renewcommand\color[2][]{}%
  }%
  \providecommand\transparent[1]{%
    \errmessage{(Inkscape) Transparency is used (non-zero) for the text in Inkscape, but the package 'transparent.sty' is not loaded}%
    \renewcommand\transparent[1]{}%
  }%
  \providecommand\rotatebox[2]{#2}%
  \ifx\svgwidth\undefined%
    \setlength{\unitlength}{417.99521484bp}%
    \ifx\svgscale\undefined%
      \relax%
    \else%
      \setlength{\unitlength}{\unitlength * \real{\svgscale}}%
    \fi%
  \else%
    \setlength{\unitlength}{\svgwidth}%
  \fi%
  \global\let\svgwidth\undefined%
  \global\let\svgscale\undefined%
  \makeatother%
  \begin{picture}(1,0.4967227)%
    \put(0,0){\includegraphics[width=\unitlength]{TwoSheetDict.pdf}}%
    \put(0.96991032,0.45066494){\color[rgb]{0,0,0}\makebox(0,0)[b]{\smash{$+$}}}%
    \put(0.97072078,0.39392914){\color[rgb]{0,0,0}\makebox(0,0)[b]{\smash{$-$}}}%
    \put(0.96991032,0.29755314){\color[rgb]{0,0,0}\makebox(0,0)[b]{\smash{$+$}}}%
    \put(0.97837637,0.24081735){\color[rgb]{0,0,0}\makebox(0,0)[b]{\smash{$-$}}}%
    \put(0.61775317,0.29755314){\color[rgb]{0,0,0}\makebox(0,0)[b]{\smash{$+$}}}%
    \put(0.61856375,0.24081735){\color[rgb]{0,0,0}\makebox(0,0)[b]{\smash{$-$}}}%
    \put(0.96991032,0.14444132){\color[rgb]{0,0,0}\makebox(0,0)[b]{\smash{$+$}}}%
    \put(0.97072078,0.08770554){\color[rgb]{0,0,0}\makebox(0,0)[b]{\smash{$-$}}}%
    \put(0.62158097,0.02195188){\color[rgb]{0,0,0}\makebox(0,0)[b]{\smash{$+$}}}%
    \put(0.62239155,0.05325539){\color[rgb]{0,0,0}\makebox(0,0)[b]{\smash{$-$}}}%
    \put(0.61775317,0.14444132){\color[rgb]{0,0,0}\makebox(0,0)[b]{\smash{$+$}}}%
    \put(0.61856375,0.17574483){\color[rgb]{0,0,0}\makebox(0,0)[b]{\smash{$-$}}}%
  \end{picture}%
\endgroup%

%% file: PBoalchDblSymp2.pdf_tex
\begingroup%
  \makeatletter%
  \providecommand\color[2][]{%
    \errmessage{(Inkscape) Color is used for the text in Inkscape, but the package 'color.sty' is not loaded}%
    \renewcommand\color[2][]{}%
  }%
  \providecommand\transparent[1]{%
    \errmessage{(Inkscape) Transparency is used (non-zero) for the text in Inkscape, but the package 'transparent.sty' is not loaded}%
    \renewcommand\transparent[1]{}%
  }%
  \providecommand\rotatebox[2]{#2}%
  \ifx\svgwidth\undefined%
    \setlength{\unitlength}{431.37949219bp}%
    \ifx\svgscale\undefined%
      \relax%
    \else%
      \setlength{\unitlength}{\unitlength * \real{\svgscale}}%
    \fi%
  \else%
    \setlength{\unitlength}{\svgwidth}%
  \fi%
  \global\let\svgwidth\undefined%
  \global\let\svgscale\undefined%
  \makeatother%
  \begin{picture}(1,0.42265754)%
    \put(0,0){\includegraphics[width=\unitlength]{PBoalchDblSymp2.pdf}}%
    \put(0.1426888,0.2056003){\color[rgb]{0,0,0}\makebox(0,0)[b]{\smash{$\mf{d}$}}}%
    \put(0.24303575,0.18805835){\color[rgb]{0.58823529,0,0}\makebox(0,0)[b]{\smash{$\g_\Delta$}}}%
    \put(0.05179604,0.23124185){\color[rgb]{0.58823529,0,0}\makebox(0,0)[b]{\smash{$\g_\Delta$}}}%
    \put(0.18567856,0.26275412){\color[rgb]{0,0,0.58823529}\makebox(0,0)[b]{\smash{$\mf{h}$}}}%
    \put(0.096562,0.14954334){\color[rgb]{0,0,0.58823529}\makebox(0,0)[b]{\smash{$\mf{h}$}}}%
    \put(0.63004482,0.31161759){\color[rgb]{0.58823529,0,0}\makebox(0,0)[b]{\smash{$\g_\Delta$}}}%
    \put(0.43138708,0.38076431){\color[rgb]{0.58823529,0,0}\makebox(0,0)[b]{\smash{$\g_\Delta$}}}%
    \put(0.56526955,0.40856756){\color[rgb]{0,0,0.58823529}\makebox(0,0)[b]{\smash{$\mf{h}$}}}%
    \put(0.4687349,0.26568452){\color[rgb]{0,0,0.58823529}\makebox(0,0)[b]{\smash{$\mf{h}$}}}%
    \put(0.52711785,0.35809157){\color[rgb]{0,0,0}\makebox(0,0)[b]{\smash{$\g$}}}%
    \put(0.55853414,0.14046995){\color[rgb]{0,0,0.58823529}\makebox(0,0)[b]{\smash{$\mf{h}$}}}%
    \put(0.62810417,0.06858991){\color[rgb]{0.58823529,0,0}\makebox(0,0)[b]{\smash{$\g_\Delta$}}}%
    \put(0.42240405,0.11293157){\color[rgb]{0.58823529,0,0}\makebox(0,0)[b]{\smash{$\g_\Delta$}}}%
    \put(0.47424344,0.02938581){\color[rgb]{0,0,0.58823529}\makebox(0,0)[b]{\smash{$\mf{h}$}}}%
    \put(0.52340882,0.09475035){\color[rgb]{0,0,0}\makebox(0,0)[b]{\smash{$\g$}}}%
    \put(0.94292885,0.27431366){\color[rgb]{0,0,0.58823529}\makebox(0,0)[b]{\smash{$\mf{h}$}}}%
    \put(0.86802799,0.16557698){\color[rgb]{0,0,0.58823529}\makebox(0,0)[b]{\smash{$\mf{h}$}}}%
    \put(0.90760012,0.22447821){\color[rgb]{0,0,0}\makebox(0,0)[b]{\smash{$\g$}}}%
  \end{picture}%
\endgroup%

%% file: Signp.pdf_tex
\begingroup%
  \makeatletter%
  \providecommand\color[2][]{%
    \errmessage{(Inkscape) Color is used for the text in Inkscape, but the package 'color.sty' is not loaded}%
    \renewcommand\color[2][]{}%
  }%
  \providecommand\transparent[1]{%
    \errmessage{(Inkscape) Transparency is used (non-zero) for the text in Inkscape, but the package 'transparent.sty' is not loaded}%
    \renewcommand\transparent[1]{}%
  }%
  \providecommand\rotatebox[2]{#2}%
  \ifx\svgwidth\undefined%
    \setlength{\unitlength}{318.78994141bp}%
    \ifx\svgscale\undefined%
      \relax%
    \else%
      \setlength{\unitlength}{\unitlength * \real{\svgscale}}%
    \fi%
  \else%
    \setlength{\unitlength}{\svgwidth}%
  \fi%
  \global\let\svgwidth\undefined%
  \global\let\svgscale\undefined%
  \makeatother%
  \begin{picture}(1,0.44126713)%
    \put(0,0){\includegraphics[width=\unitlength]{Signp.pdf}}%
    \put(0.48399402,0.21698863){\color[rgb]{0,0,0}\makebox(0,0)[b]{\smash{$A$}}}%
    \put(0.53680105,0.39710599){\color[rgb]{0,0,0}\makebox(0,0)[b]{\smash{$\alpha$}}}%
    \put(0.37023269,0.39557432){\color[rgb]{0,0,0}\makebox(0,0)[b]{\smash{$\beta$}}}%
    \put(0.40119599,0.21622279){\color[rgb]{0,0,0}\makebox(0,0)[b]{\smash{$-$}}}%
    \put(0.1828553,0.21698861){\color[rgb]{0,0,0}\makebox(0,0)[b]{\smash{$A$}}}%
    \put(0.23476874,0.39557432){\color[rgb]{0,0,0}\makebox(0,0)[b]{\smash{$\beta$}}}%
    \put(0.07802944,0.39710599){\color[rgb]{0,0,0}\makebox(0,0)[b]{\smash{$\alpha$}}}%
    \put(0.10005726,0.21622277){\color[rgb]{0,0,0}\makebox(0,0)[b]{\smash{$+$}}}%
    \put(0.83532255,0.21698861){\color[rgb]{0,0,0}\makebox(0,0)[b]{\smash{$A$}}}%
    \put(0.89170373,0.39557432){\color[rgb]{0,0,0}\makebox(0,0)[b]{\smash{$\beta$}}}%
    \put(0.71977414,0.39710599){\color[rgb]{0,0,0}\makebox(0,0)[b]{\smash{$\alpha$}}}%
    \put(0.81091828,0.03101327){\color[rgb]{0.78431373,0,0}\makebox(0,0)[b]{\smash{$[\alpha^{-1}_A\beta_A]$}}}%
  \end{picture}%
\endgroup%

%% file: 2MrkDisk.pdf_tex
\begingroup%
  \makeatletter%
  \providecommand\color[2][]{%
    \errmessage{(Inkscape) Color is used for the text in Inkscape, but the package 'color.sty' is not loaded}%
    \renewcommand\color[2][]{}%
  }%
  \providecommand\transparent[1]{%
    \errmessage{(Inkscape) Transparency is used (non-zero) for the text in Inkscape, but the package 'transparent.sty' is not loaded}%
    \renewcommand\transparent[1]{}%
  }%
  \providecommand\rotatebox[2]{#2}%
  \ifx\svgwidth\undefined%
    \setlength{\unitlength}{141.95708008bp}%
    \ifx\svgscale\undefined%
      \relax%
    \else%
      \setlength{\unitlength}{\unitlength * \real{\svgscale}}%
    \fi%
  \else%
    \setlength{\unitlength}{\svgwidth}%
  \fi%
  \global\let\svgwidth\undefined%
  \global\let\svgscale\undefined%
  \makeatother%
  \begin{picture}(1,0.68453982)%
    \put(0,0){\includegraphics[width=\unitlength]{2MrkDisk.pdf}}%
    \put(0.64213985,0.32993684){\color[rgb]{0,0,0}\makebox(0,0)[lb]{\smash{$g_{e_1}$}}}%
    \put(0.35788766,0.33298933){\color[rgb]{0,0,0}\makebox(0,0)[rb]{\smash{$g_{e_2}$}}}%
  \end{picture}%
\endgroup%

%% file: DblPoissLieGrpDsk.pdf_tex
\begingroup%
  \makeatletter%
  \providecommand\color[2][]{%
    \errmessage{(Inkscape) Color is used for the text in Inkscape, but the package 'color.sty' is not loaded}%
    \renewcommand\color[2][]{}%
  }%
  \providecommand\transparent[1]{%
    \errmessage{(Inkscape) Transparency is used (non-zero) for the text in Inkscape, but the package 'transparent.sty' is not loaded}%
    \renewcommand\transparent[1]{}%
  }%
  \providecommand\rotatebox[2]{#2}%
  \ifx\svgwidth\undefined%
    \setlength{\unitlength}{279.76342773bp}%
    \ifx\svgscale\undefined%
      \relax%
    \else%
      \setlength{\unitlength}{\unitlength * \real{\svgscale}}%
    \fi%
  \else%
    \setlength{\unitlength}{\svgwidth}%
  \fi%
  \global\let\svgwidth\undefined%
  \global\let\svgscale\undefined%
  \makeatother%
  \begin{picture}(1,0.31234548)%
    \put(0,0){\includegraphics[width=\unitlength]{DblPoissLieGrpDsk.pdf}}%
    \put(0.4679576,0.24486658){\color[rgb]{0.78431373,0,0}\makebox(0,0)[rb]{\smash{$\mf{e}$}}}%
    \put(0.51715665,0.24486658){\color[rgb]{0,0,0.78431373}\makebox(0,0)[lb]{\smash{$\mf{f}$}}}%
    \put(0.10750476,0.14798218){\color[rgb]{0,0,0}\makebox(0,0)[rb]{\smash{$e_1$}}}%
    \put(0.16158671,0.2509263){\color[rgb]{0,0,0}\makebox(0,0)[b]{\smash{$v_1$}}}%
    \put(0.16158671,0.03931895){\color[rgb]{0,0,0}\makebox(0,0)[b]{\smash{$v_2$}}}%
    \put(0.216168,0.14798218){\color[rgb]{0,0,0}\makebox(0,0)[lb]{\smash{$e_2$}}}%
    \put(0.44017425,0.14798218){\color[rgb]{0,0,0}\makebox(0,0)[rb]{\smash{$G$}}}%
    \put(0.54883749,0.14798218){\color[rgb]{0,0,0}\makebox(0,0)[lb]{\smash{$G$}}}%
    \put(0.4679576,0.05041649){\color[rgb]{0.78431373,0,0}\makebox(0,0)[rb]{\smash{$\mf{e}$}}}%
    \put(0.51715665,0.05041649){\color[rgb]{0,0,0.78431373}\makebox(0,0)[lb]{\smash{$\mf{f}$}}}%
    \put(0.78332132,0.14798218){\color[rgb]{0,0,0}\makebox(0,0)[rb]{\smash{$g_1$}}}%
    \put(0.89198456,0.14798218){\color[rgb]{0,0,0}\makebox(0,0)[lb]{\smash{$g_2$}}}%
  \end{picture}%
\endgroup%

%% file: PoissLieGrpDsk.pdf_tex
\begingroup%
  \makeatletter%
  \providecommand\color[2][]{%
    \errmessage{(Inkscape) Color is used for the text in Inkscape, but the package 'color.sty' is not loaded}%
    \renewcommand\color[2][]{}%
  }%
  \providecommand\transparent[1]{%
    \errmessage{(Inkscape) Transparency is used (non-zero) for the text in Inkscape, but the package 'transparent.sty' is not loaded}%
    \renewcommand\transparent[1]{}%
  }%
  \providecommand\rotatebox[2]{#2}%
  \ifx\svgwidth\undefined%
    \setlength{\unitlength}{280.28305664bp}%
    \ifx\svgscale\undefined%
      \relax%
    \else%
      \setlength{\unitlength}{\unitlength * \real{\svgscale}}%
    \fi%
  \else%
    \setlength{\unitlength}{\svgwidth}%
  \fi%
  \global\let\svgwidth\undefined%
  \global\let\svgscale\undefined%
  \makeatother%
  \begin{picture}(1,0.29495096)%
    \put(0,0){\includegraphics[width=\unitlength]{PoissLieGrpDsk.pdf}}%
    \put(0.468944,0.24415183){\color[rgb]{0.78431373,0,0}\makebox(0,0)[rb]{\smash{$\mf{e}$}}}%
    \put(0.51805183,0.24415183){\color[rgb]{0,0,0.78431373}\makebox(0,0)[lb]{\smash{$\mf{f}$}}}%
    \put(0.44121216,0.14744705){\color[rgb]{0.78431373,0,0}\makebox(0,0)[rb]{\smash{$E$}}}%
    \put(0.54967394,0.14744705){\color[rgb]{0,0,0}\makebox(0,0)[lb]{\smash{$G$}}}%
    \put(0.468944,0.05006224){\color[rgb]{0.78431373,0,0}\makebox(0,0)[rb]{\smash{$\mf{e}$}}}%
    \put(0.51805183,0.05006224){\color[rgb]{0,0,0.78431373}\makebox(0,0)[lb]{\smash{$\mf{f}$}}}%
    \put(0.78372307,0.14744705){\color[rgb]{0.78431373,0,0}\makebox(0,0)[rb]{\smash{$e$}}}%
    \put(0.89218481,0.14744705){\color[rgb]{0,0,0}\makebox(0,0)[lb]{\smash{$g$}}}%
    \put(0.10730546,0.1433174){\color[rgb]{0,0,0}\makebox(0,0)[rb]{\smash{$e_1$}}}%
    \put(0.16128713,0.24607067){\color[rgb]{0,0,0}\makebox(0,0)[b]{\smash{$v_1$}}}%
    \put(0.16128713,0.03485561){\color[rgb]{0,0,0}\makebox(0,0)[b]{\smash{$v_2$}}}%
    \put(0.21576724,0.1433174){\color[rgb]{0,0,0}\makebox(0,0)[lb]{\smash{$e_2$}}}%
  \end{picture}%
\endgroup%

%% file: PoisHomDsk.pdf_tex
\begingroup%
  \makeatletter%
  \providecommand\color[2][]{%
    \errmessage{(Inkscape) Color is used for the text in Inkscape, but the package 'color.sty' is not loaded}%
    \renewcommand\color[2][]{}%
  }%
  \providecommand\transparent[1]{%
    \errmessage{(Inkscape) Transparency is used (non-zero) for the text in Inkscape, but the package 'transparent.sty' is not loaded}%
    \renewcommand\transparent[1]{}%
  }%
  \providecommand\rotatebox[2]{#2}%
  \ifx\svgwidth\undefined%
    \setlength{\unitlength}{280.28305664bp}%
    \ifx\svgscale\undefined%
      \relax%
    \else%
      \setlength{\unitlength}{\unitlength * \real{\svgscale}}%
    \fi%
  \else%
    \setlength{\unitlength}{\svgwidth}%
  \fi%
  \global\let\svgwidth\undefined%
  \global\let\svgscale\undefined%
  \makeatother%
  \begin{picture}(1,0.29495096)%
    \put(0,0){\includegraphics[width=\unitlength]{PoisHomDsk.pdf}}%
    \put(0.468944,0.24415183){\color[rgb]{0.78431373,0,0}\makebox(0,0)[rb]{\smash{$\mf{e}$}}}%
    \put(0.51805183,0.24415183){\color[rgb]{0,0.58823529,0}\makebox(0,0)[lb]{\smash{$\mf{h}$}}}%
    \put(0.44121216,0.14744705){\color[rgb]{0.78431373,0,0}\makebox(0,0)[rb]{\smash{$E$}}}%
    \put(0.54967394,0.14744705){\color[rgb]{0,0,0}\makebox(0,0)[lb]{\smash{$G$}}}%
    \put(0.468944,0.05006224){\color[rgb]{0.78431373,0,0}\makebox(0,0)[rb]{\smash{$\mf{e}$}}}%
    \put(0.51805183,0.05006224){\color[rgb]{0,0,0.78431373}\makebox(0,0)[lb]{\smash{$\mf{f}$}}}%
    \put(0.78372307,0.14744705){\color[rgb]{0.78431373,0,0}\makebox(0,0)[rb]{\smash{$e$}}}%
    \put(0.89218481,0.14744705){\color[rgb]{0,0,0}\makebox(0,0)[lb]{\smash{$g$}}}%
    \put(0.10730546,0.1433174){\color[rgb]{0,0,0}\makebox(0,0)[rb]{\smash{$e_1$}}}%
    \put(0.16128713,0.24607067){\color[rgb]{0,0,0}\makebox(0,0)[b]{\smash{$v_1$}}}%
    \put(0.16128713,0.03485561){\color[rgb]{0,0,0}\makebox(0,0)[b]{\smash{$v_2$}}}%
    \put(0.21576724,0.1433174){\color[rgb]{0,0,0}\makebox(0,0)[lb]{\smash{$e_2$}}}%
  \end{picture}%
\endgroup%

%% file: AffineGrpDsk.pdf_tex
\begingroup%
  \makeatletter%
  \providecommand\color[2][]{%
    \errmessage{(Inkscape) Color is used for the text in Inkscape, but the package 'color.sty' is not loaded}%
    \renewcommand\color[2][]{}%
  }%
  \providecommand\transparent[1]{%
    \errmessage{(Inkscape) Transparency is used (non-zero) for the text in Inkscape, but the package 'transparent.sty' is not loaded}%
    \renewcommand\transparent[1]{}%
  }%
  \providecommand\rotatebox[2]{#2}%
  \ifx\svgwidth\undefined%
    \setlength{\unitlength}{279.76342773bp}%
    \ifx\svgscale\undefined%
      \relax%
    \else%
      \setlength{\unitlength}{\unitlength * \real{\svgscale}}%
    \fi%
  \else%
    \setlength{\unitlength}{\svgwidth}%
  \fi%
  \global\let\svgwidth\undefined%
  \global\let\svgscale\undefined%
  \makeatother%
  \begin{picture}(1,0.31234548)%
    \put(0,0){\includegraphics[width=\unitlength]{AffineGrpDsk.pdf}}%
    \put(0.4679576,0.24486658){\color[rgb]{0.78431373,0,0}\makebox(0,0)[rb]{\smash{$\mf{e}$}}}%
    \put(0.51715665,0.24486658){\color[rgb]{0,0.58823529,0}\makebox(0,0)[lb]{\smash{$\mf{h}$}}}%
    \put(0.10750476,0.14798218){\color[rgb]{0,0,0}\makebox(0,0)[rb]{\smash{$e_1$}}}%
    \put(0.16158671,0.2509263){\color[rgb]{0,0,0}\makebox(0,0)[b]{\smash{$v_1$}}}%
    \put(0.16158671,0.03931895){\color[rgb]{0,0,0}\makebox(0,0)[b]{\smash{$v_2$}}}%
    \put(0.216168,0.14798218){\color[rgb]{0,0,0}\makebox(0,0)[lb]{\smash{$e_2$}}}%
    \put(0.44017425,0.14798218){\color[rgb]{0,0,0}\makebox(0,0)[rb]{\smash{$G$}}}%
    \put(0.54883749,0.14798218){\color[rgb]{0,0,0}\makebox(0,0)[lb]{\smash{$G$}}}%
    \put(0.4679576,0.05041649){\color[rgb]{0.78431373,0,0}\makebox(0,0)[rb]{\smash{$\mf{e}$}}}%
    \put(0.51715665,0.05041649){\color[rgb]{0,0,0.78431373}\makebox(0,0)[lb]{\smash{$\mf{f}$}}}%
    \put(0.78332132,0.14798218){\color[rgb]{0,0,0}\makebox(0,0)[rb]{\smash{$g_1$}}}%
    \put(0.89198456,0.14798218){\color[rgb]{0,0,0}\makebox(0,0)[lb]{\smash{$g_2$}}}%
  \end{picture}%
\endgroup%

%% file: LuYakimovDsk.pdf_tex
\begingroup%
  \makeatletter%
  \providecommand\color[2][]{%
    \errmessage{(Inkscape) Color is used for the text in Inkscape, but the package 'color.sty' is not loaded}%
    \renewcommand\color[2][]{}%
  }%
  \providecommand\transparent[1]{%
    \errmessage{(Inkscape) Transparency is used (non-zero) for the text in Inkscape, but the package 'transparent.sty' is not loaded}%
    \renewcommand\transparent[1]{}%
  }%
  \providecommand\rotatebox[2]{#2}%
  \ifx\svgwidth\undefined%
    \setlength{\unitlength}{279.76342773bp}%
    \ifx\svgscale\undefined%
      \relax%
    \else%
      \setlength{\unitlength}{\unitlength * \real{\svgscale}}%
    \fi%
  \else%
    \setlength{\unitlength}{\svgwidth}%
  \fi%
  \global\let\svgwidth\undefined%
  \global\let\svgscale\undefined%
  \makeatother%
  \begin{picture}(1,0.31234548)%
    \put(0,0){\includegraphics[width=\unitlength]{LuYakimovDsk.pdf}}%
    \put(0.4908341,0.2505857){\color[rgb]{0,0.58823529,0}\makebox(0,0)[b]{\smash{$\mf{l}_{\mf{c}}$}}}%
    \put(0.10750476,0.14798218){\color[rgb]{0,0,0}\makebox(0,0)[rb]{\smash{$e_1$}}}%
    \put(0.16158671,0.2509263){\color[rgb]{0,0,0}\makebox(0,0)[b]{\smash{$v_1$}}}%
    \put(0.16158671,0.03931895){\color[rgb]{0,0,0}\makebox(0,0)[b]{\smash{$v_2$}}}%
    \put(0.216168,0.14798218){\color[rgb]{0,0,0}\makebox(0,0)[lb]{\smash{$e_2$}}}%
    \put(0.44017425,0.14798218){\color[rgb]{0,0,0}\makebox(0,0)[rb]{\smash{$G$}}}%
    \put(0.54883749,0.14798218){\color[rgb]{0,0,0}\makebox(0,0)[lb]{\smash{$G$}}}%
    \put(0.4679576,0.05041649){\color[rgb]{0.78431373,0,0}\makebox(0,0)[rb]{\smash{$\mf{e}$}}}%
    \put(0.51715665,0.05041649){\color[rgb]{0,0,0.78431373}\makebox(0,0)[lb]{\smash{$\mf{f}$}}}%
    \put(0.78332132,0.14798218){\color[rgb]{0,0,0}\makebox(0,0)[rb]{\smash{$g_1$}}}%
    \put(0.89198456,0.14798218){\color[rgb]{0,0,0}\makebox(0,0)[lb]{\smash{$g_2$}}}%
  \end{picture}%
\endgroup%

%% file: PoissLieGrpBoalch.pdf_tex
\begingroup%
  \makeatletter%
  \providecommand\color[2][]{%
    \errmessage{(Inkscape) Color is used for the text in Inkscape, but the package 'color.sty' is not loaded}%
    \renewcommand\color[2][]{}%
  }%
  \providecommand\transparent[1]{%
    \errmessage{(Inkscape) Transparency is used (non-zero) for the text in Inkscape, but the package 'transparent.sty' is not loaded}%
    \renewcommand\transparent[1]{}%
  }%
  \providecommand\rotatebox[2]{#2}%
  \ifx\svgwidth\undefined%
    \setlength{\unitlength}{413.484375bp}%
    \ifx\svgscale\undefined%
      \relax%
    \else%
      \setlength{\unitlength}{\unitlength * \real{\svgscale}}%
    \fi%
  \else%
    \setlength{\unitlength}{\svgwidth}%
  \fi%
  \global\let\svgwidth\undefined%
  \global\let\svgscale\undefined%
  \makeatother%
  \begin{picture}(1,0.19993442)%
    \put(0,0){\includegraphics[width=\unitlength]{PoissLieGrpBoalch.pdf}}%
    \put(0.24022677,0.16068126){\color[rgb]{0,0,0.78431373}\makebox(0,0)[rb]{\smash{$\mf{h}$}}}%
    \put(0.27351483,0.16068126){\color[rgb]{0.78431373,0,0}\makebox(0,0)[lb]{\smash{$\g_\Delta$}}}%
    \put(0.22142857,0.0951293){\color[rgb]{0,0,0.78431373}\makebox(0,0)[rb]{\smash{$H$}}}%
    \put(0.29495009,0.0951293){\color[rgb]{0,0,0}\makebox(0,0)[lb]{\smash{$D$}}}%
    \put(0.24022677,0.02911637){\color[rgb]{0,0,0.78431373}\makebox(0,0)[rb]{\smash{$\mf{h}$}}}%
    \put(0.27351483,0.02911637){\color[rgb]{0.78431373,0,0}\makebox(0,0)[lb]{\smash{$\g_\Delta$}}}%
    \put(0.53358701,0.09797216){\color[rgb]{0,0,0}\makebox(0,0)[lb]{\smash{$G$}}}%
    \put(0.38267442,0.09797216){\color[rgb]{0,0,0}\makebox(0,0)[rb]{\smash{$G$}}}%
    \put(0.45232638,0.09797216){\color[rgb]{0,0,0.78431373}\makebox(0,0)[rb]{\smash{$H$}}}%
    \put(0.46425873,0.134551){\color[rgb]{0.78431373,0,0}\makebox(0,0)[lb]{\smash{$\g_\Delta$}}}%
    \put(0.46425873,0.06102948){\color[rgb]{0.78431373,0,0}\makebox(0,0)[lb]{\smash{$\g_\Delta$}}}%
    \put(0.46393504,0.09797216){\color[rgb]{0,0,0.78431373}\makebox(0,0)[lb]{\smash{$\h$}}}%
    \put(0.74024563,0.09850919){\color[rgb]{0,0,0}\makebox(0,0)[lb]{\smash{$G$}}}%
    \put(0.658985,0.09850919){\color[rgb]{0,0,0.78431373}\makebox(0,0)[rb]{\smash{$H$}}}%
    \put(0.67059366,0.09850919){\color[rgb]{0,0,0.78431373}\makebox(0,0)[lb]{\smash{$\h$}}}%
    \put(0.05338993,0.09714864){\color[rgb]{0,0,0}\makebox(0,0)[rb]{\smash{$e_1$}}}%
    \put(0.08998176,0.16680061){\color[rgb]{0,0,0}\makebox(0,0)[b]{\smash{$v_1$}}}%
    \put(0.08998176,0.02362712){\color[rgb]{0,0,0}\makebox(0,0)[b]{\smash{$v_2$}}}%
    \put(0.12691145,0.09714864){\color[rgb]{0,0,0}\makebox(0,0)[lb]{\smash{$e_2$}}}%
    \put(0.85394076,0.09792945){\color[rgb]{0,0,0}\makebox(0,0)[rb]{\smash{$e_3$}}}%
    \put(0.88893218,0.09792945){\color[rgb]{0,0,0}\makebox(0,0)[lb]{\smash{$e_4$}}}%
    \put(0.87145708,0.1310628){\color[rgb]{0,0,0}\makebox(0,0)[b]{\smash{$v_3$}}}%
    \put(0.87145708,0.06116266){\color[rgb]{0,0,0}\makebox(0,0)[b]{\smash{$v_4$}}}%
  \end{picture}%
\endgroup%

%% file: twistedCartCour.pdf_tex
\begingroup%
  \makeatletter%
  \providecommand\color[2][]{%
    \errmessage{(Inkscape) Color is used for the text in Inkscape, but the package 'color.sty' is not loaded}%
    \renewcommand\color[2][]{}%
  }%
  \providecommand\transparent[1]{%
    \errmessage{(Inkscape) Transparency is used (non-zero) for the text in Inkscape, but the package 'transparent.sty' is not loaded}%
    \renewcommand\transparent[1]{}%
  }%
  \providecommand\rotatebox[2]{#2}%
  \ifx\svgwidth\undefined%
    \setlength{\unitlength}{191.425bp}%
    \ifx\svgscale\undefined%
      \relax%
    \else%
      \setlength{\unitlength}{\unitlength * \real{\svgscale}}%
    \fi%
  \else%
    \setlength{\unitlength}{\svgwidth}%
  \fi%
  \global\let\svgwidth\undefined%
  \global\let\svgscale\undefined%
  \makeatother%
  \begin{picture}(1,0.32792732)%
    \put(0,0){\includegraphics[width=\unitlength]{twistedCartCour.pdf}}%
    \put(0.66900957,0.16462626){\color[rgb]{0,0,0}\rotatebox{-12.24326978}{\makebox(0,0)[b]{\smash{$\ol{\mf{g}}\circlearrowleft G\circlearrowright\mf{g}$}}}}%
    \put(0.33483399,0.17230566){\color[rgb]{0,0,0}\rotatebox{12.02905569}{\makebox(0,0)[b]{\smash{$\ol{\mf{g}}\circlearrowleft G\circlearrowright\mf{g}$}}}}%
    \put(0.67178895,0.10482744){\color[rgb]{0,0,0}\rotatebox{-11.53324931}{\makebox(0,0)[b]{\smash{$e$}}}}%
    \put(0.34583553,0.10511675){\color[rgb]{0,0,0}\rotatebox{12.38214218}{\makebox(0,0)[b]{\smash{$e'$}}}}%
    \put(0.48588799,0.13576092){\color[rgb]{0,0,0}\makebox(0,0)[b]{\smash{$v$}}}%
  \end{picture}%
\endgroup%

%% file: twistedCartDir.pdf_tex
\begingroup%
  \makeatletter%
  \providecommand\color[2][]{%
    \errmessage{(Inkscape) Color is used for the text in Inkscape, but the package 'color.sty' is not loaded}%
    \renewcommand\color[2][]{}%
  }%
  \providecommand\transparent[1]{%
    \errmessage{(Inkscape) Transparency is used (non-zero) for the text in Inkscape, but the package 'transparent.sty' is not loaded}%
    \renewcommand\transparent[1]{}%
  }%
  \providecommand\rotatebox[2]{#2}%
  \ifx\svgwidth\undefined%
    \setlength{\unitlength}{191.425bp}%
    \ifx\svgscale\undefined%
      \relax%
    \else%
      \setlength{\unitlength}{\unitlength * \real{\svgscale}}%
    \fi%
  \else%
    \setlength{\unitlength}{\svgwidth}%
  \fi%
  \global\let\svgwidth\undefined%
  \global\let\svgscale\undefined%
  \makeatother%
  \begin{picture}(1,0.32792732)%
    \put(0,0){\includegraphics[width=\unitlength]{twistedCartDir.pdf}}%
    \put(0.66900954,0.17298464){\color[rgb]{0,0,0}\rotatebox{-12.24326978}{\makebox(0,0)[b]{\smash{$\xi_{v}\circlearrowleft G\cdots$}}}}%
    \put(0.33483397,0.18066403){\color[rgb]{0,0,0}\rotatebox{12.02905569}{\makebox(0,0)[b]{\smash{$\cdots G\circlearrowright\xi_{v}$}}}}%
    \put(0.3339668,0.10403293){\color[rgb]{0,0,0}\rotatebox{12.20818656}{\makebox(0,0)[b]{\smash{$e'$}}}}%
    \put(0.65748251,0.10489783){\color[rgb]{0,0,0}\rotatebox{-13.02021519}{\makebox(0,0)[b]{\smash{$e$}}}}%
    \put(0.48936035,0.13576092){\color[rgb]{0,0,0}\makebox(0,0)[b]{\smash{$v$}}}%
  \end{picture}%
\endgroup%

%% file: SewingEdges.pdf_tex
\begingroup%
  \makeatletter%
  \providecommand\color[2][]{%
    \errmessage{(Inkscape) Color is used for the text in Inkscape, but the package 'color.sty' is not loaded}%
    \renewcommand\color[2][]{}%
  }%
  \providecommand\transparent[1]{%
    \errmessage{(Inkscape) Transparency is used (non-zero) for the text in Inkscape, but the package 'transparent.sty' is not loaded}%
    \renewcommand\transparent[1]{}%
  }%
  \providecommand\rotatebox[2]{#2}%
  \ifx\svgwidth\undefined%
    \setlength{\unitlength}{116.38927002bp}%
    \ifx\svgscale\undefined%
      \relax%
    \else%
      \setlength{\unitlength}{\unitlength * \real{\svgscale}}%
    \fi%
  \else%
    \setlength{\unitlength}{\svgwidth}%
  \fi%
  \global\let\svgwidth\undefined%
  \global\let\svgscale\undefined%
  \makeatother%
  \begin{picture}(1,1.43250499)%
    \put(0,0){\includegraphics[width=\unitlength]{SewingEdges.pdf}}%
  \end{picture}%
\endgroup%

%% file: SewingEdges2.pdf_tex
\begingroup%
  \makeatletter%
  \providecommand\color[2][]{%
    \errmessage{(Inkscape) Color is used for the text in Inkscape, but the package 'color.sty' is not loaded}%
    \renewcommand\color[2][]{}%
  }%
  \providecommand\transparent[1]{%
    \errmessage{(Inkscape) Transparency is used (non-zero) for the text in Inkscape, but the package 'transparent.sty' is not loaded}%
    \renewcommand\transparent[1]{}%
  }%
  \providecommand\rotatebox[2]{#2}%
  \ifx\svgwidth\undefined%
    \setlength{\unitlength}{116.38927002bp}%
    \ifx\svgscale\undefined%
      \relax%
    \else%
      \setlength{\unitlength}{\unitlength * \real{\svgscale}}%
    \fi%
  \else%
    \setlength{\unitlength}{\svgwidth}%
  \fi%
  \global\let\svgwidth\undefined%
  \global\let\svgscale\undefined%
  \makeatother%
  \begin{picture}(1,1.35827133)%
    \put(0,0){\includegraphics[width=\unitlength]{SewingEdges2.pdf}}%
  \end{picture}%
\endgroup%

%% file: SewingEdges3.pdf_tex
\begingroup%
  \makeatletter%
  \providecommand\color[2][]{%
    \errmessage{(Inkscape) Color is used for the text in Inkscape, but the package 'color.sty' is not loaded}%
    \renewcommand\color[2][]{}%
  }%
  \providecommand\transparent[1]{%
    \errmessage{(Inkscape) Transparency is used (non-zero) for the text in Inkscape, but the package 'transparent.sty' is not loaded}%
    \renewcommand\transparent[1]{}%
  }%
  \providecommand\rotatebox[2]{#2}%
  \ifx\svgwidth\undefined%
    \setlength{\unitlength}{116.3532959bp}%
    \ifx\svgscale\undefined%
      \relax%
    \else%
      \setlength{\unitlength}{\unitlength * \real{\svgscale}}%
    \fi%
  \else%
    \setlength{\unitlength}{\svgwidth}%
  \fi%
  \global\let\svgwidth\undefined%
  \global\let\svgscale\undefined%
  \makeatother%
  \begin{picture}(1,1.35869128)%
    \put(0,0){\includegraphics[width=\unitlength]{SewingEdges3.pdf}}%
  \end{picture}%
\endgroup%

%% file: GraphEdId1.pdf_tex
\begingroup%
  \makeatletter%
  \providecommand\color[2][]{%
    \errmessage{(Inkscape) Color is used for the text in Inkscape, but the package 'color.sty' is not loaded}%
    \renewcommand\color[2][]{}%
  }%
  \providecommand\transparent[1]{%
    \errmessage{(Inkscape) Transparency is used (non-zero) for the text in Inkscape, but the package 'transparent.sty' is not loaded}%
    \renewcommand\transparent[1]{}%
  }%
  \providecommand\rotatebox[2]{#2}%
  \ifx\svgwidth\undefined%
    \setlength{\unitlength}{166.30184326bp}%
    \ifx\svgscale\undefined%
      \relax%
    \else%
      \setlength{\unitlength}{\unitlength * \real{\svgscale}}%
    \fi%
  \else%
    \setlength{\unitlength}{\svgwidth}%
  \fi%
  \global\let\svgwidth\undefined%
  \global\let\svgscale\undefined%
  \makeatother%
  \begin{picture}(1,0.61273223)%
    \put(0,0){\includegraphics[width=\unitlength]{GraphEdId1.pdf}}%
    \put(0.73638556,0.42724698){\color[rgb]{0,0,0}\makebox(0,0)[b]{\smash{$e_1$}}}%
    \put(0.65922418,0.23280032){\color[rgb]{0,0,0}\makebox(0,0)[b]{\smash{$e_2$}}}%
  \end{picture}%
\endgroup%

%% file: GraphEdId2.pdf_tex
\begingroup%
  \makeatletter%
  \providecommand\color[2][]{%
    \errmessage{(Inkscape) Color is used for the text in Inkscape, but the package 'color.sty' is not loaded}%
    \renewcommand\color[2][]{}%
  }%
  \providecommand\transparent[1]{%
    \errmessage{(Inkscape) Transparency is used (non-zero) for the text in Inkscape, but the package 'transparent.sty' is not loaded}%
    \renewcommand\transparent[1]{}%
  }%
  \providecommand\rotatebox[2]{#2}%
  \ifx\svgwidth\undefined%
    \setlength{\unitlength}{166.30184326bp}%
    \ifx\svgscale\undefined%
      \relax%
    \else%
      \setlength{\unitlength}{\unitlength * \real{\svgscale}}%
    \fi%
  \else%
    \setlength{\unitlength}{\svgwidth}%
  \fi%
  \global\let\svgwidth\undefined%
  \global\let\svgscale\undefined%
  \makeatother%
  \begin{picture}(1,0.61273223)%
    \put(0,0){\includegraphics[width=\unitlength]{GraphEdId2.pdf}}%
  \end{picture}%
\endgroup%

%% file: GraphEdId3.pdf_tex
\begingroup%
  \makeatletter%
  \providecommand\color[2][]{%
    \errmessage{(Inkscape) Color is used for the text in Inkscape, but the package 'color.sty' is not loaded}%
    \renewcommand\color[2][]{}%
  }%
  \providecommand\transparent[1]{%
    \errmessage{(Inkscape) Transparency is used (non-zero) for the text in Inkscape, but the package 'transparent.sty' is not loaded}%
    \renewcommand\transparent[1]{}%
  }%
  \providecommand\rotatebox[2]{#2}%
  \ifx\svgwidth\undefined%
    \setlength{\unitlength}{166.30184326bp}%
    \ifx\svgscale\undefined%
      \relax%
    \else%
      \setlength{\unitlength}{\unitlength * \real{\svgscale}}%
    \fi%
  \else%
    \setlength{\unitlength}{\svgwidth}%
  \fi%
  \global\let\svgwidth\undefined%
  \global\let\svgscale\undefined%
  \makeatother%
  \begin{picture}(1,0.61273223)%
    \put(0,0){\includegraphics[width=\unitlength]{GraphEdId3.pdf}}%
  \end{picture}%
\endgroup%

%% file: 3GonTriang.pdf_tex
\begingroup%
  \makeatletter%
  \providecommand\color[2][]{%
    \errmessage{(Inkscape) Color is used for the text in Inkscape, but the package 'color.sty' is not loaded}%
    \renewcommand\color[2][]{}%
  }%
  \providecommand\transparent[1]{%
    \errmessage{(Inkscape) Transparency is used (non-zero) for the text in Inkscape, but the package 'transparent.sty' is not loaded}%
    \renewcommand\transparent[1]{}%
  }%
  \providecommand\rotatebox[2]{#2}%
  \ifx\svgwidth\undefined%
    \setlength{\unitlength}{167.44372559bp}%
    \ifx\svgscale\undefined%
      \relax%
    \else%
      \setlength{\unitlength}{\unitlength * \real{\svgscale}}%
    \fi%
  \else%
    \setlength{\unitlength}{\svgwidth}%
  \fi%
  \global\let\svgwidth\undefined%
  \global\let\svgscale\undefined%
  \makeatother%
  \begin{picture}(1,0.89503011)%
    \put(0,0){\includegraphics[width=\unitlength]{3GonTriang.pdf}}%
  \end{picture}%
\endgroup%

%% file: 4GonTriang.pdf_tex
\begingroup%
  \makeatletter%
  \providecommand\color[2][]{%
    \errmessage{(Inkscape) Color is used for the text in Inkscape, but the package 'color.sty' is not loaded}%
    \renewcommand\color[2][]{}%
  }%
  \providecommand\transparent[1]{%
    \errmessage{(Inkscape) Transparency is used (non-zero) for the text in Inkscape, but the package 'transparent.sty' is not loaded}%
    \renewcommand\transparent[1]{}%
  }%
  \providecommand\rotatebox[2]{#2}%
  \ifx\svgwidth\undefined%
    \setlength{\unitlength}{149.8bp}%
    \ifx\svgscale\undefined%
      \relax%
    \else%
      \setlength{\unitlength}{\unitlength * \real{\svgscale}}%
    \fi%
  \else%
    \setlength{\unitlength}{\svgwidth}%
  \fi%
  \global\let\svgwidth\undefined%
  \global\let\svgscale\undefined%
  \makeatother%
  \begin{picture}(1,1.00066756)%
    \put(0,0){\includegraphics[width=\unitlength]{4GonTriang.pdf}}%
  \end{picture}%
\endgroup%

%% file: twistedCartCour1.pdf_tex
\begingroup%
  \makeatletter%
  \providecommand\color[2][]{%
    \errmessage{(Inkscape) Color is used for the text in Inkscape, but the package 'color.sty' is not loaded}%
    \renewcommand\color[2][]{}%
  }%
  \providecommand\transparent[1]{%
    \errmessage{(Inkscape) Transparency is used (non-zero) for the text in Inkscape, but the package 'transparent.sty' is not loaded}%
    \renewcommand\transparent[1]{}%
  }%
  \providecommand\rotatebox[2]{#2}%
  \ifx\svgwidth\undefined%
    \setlength{\unitlength}{191.425bp}%
    \ifx\svgscale\undefined%
      \relax%
    \else%
      \setlength{\unitlength}{\unitlength * \real{\svgscale}}%
    \fi%
  \else%
    \setlength{\unitlength}{\svgwidth}%
  \fi%
  \global\let\svgwidth\undefined%
  \global\let\svgscale\undefined%
  \makeatother%
  \begin{picture}(1,0.32792732)%
    \put(0,0){\includegraphics[width=\unitlength]{twistedCartCour1.pdf}}%
    \put(0.66900957,0.16462626){\color[rgb]{0,0,0}\rotatebox{-12.24326978}{\makebox(0,0)[b]{\smash{$G$}}}}%
    \put(0.33483399,0.17230566){\color[rgb]{0,0,0}\rotatebox{12.02905569}{\makebox(0,0)[b]{\smash{$G$}}}}%
    \put(0.67178895,0.10482744){\color[rgb]{0,0,0}\rotatebox{-11.53324931}{\makebox(0,0)[b]{\smash{$e$}}}}%
    \put(0.34583553,0.10511675){\color[rgb]{0,0,0}\rotatebox{12.38214218}{\makebox(0,0)[b]{\smash{$e'$}}}}%
    \put(0.48588799,0.13576092){\color[rgb]{0,0,0}\makebox(0,0)[b]{\smash{$v$}}}%
    \put(0.48757,0.21128033){\color[rgb]{0,0,0}\makebox(0,0)[b]{\smash{$\mf{d}$}}}%
  \end{picture}%
\endgroup%

%% file: twistedCartDir1.pdf_tex
\begingroup%
  \makeatletter%
  \providecommand\color[2][]{%
    \errmessage{(Inkscape) Color is used for the text in Inkscape, but the package 'color.sty' is not loaded}%
    \renewcommand\color[2][]{}%
  }%
  \providecommand\transparent[1]{%
    \errmessage{(Inkscape) Transparency is used (non-zero) for the text in Inkscape, but the package 'transparent.sty' is not loaded}%
    \renewcommand\transparent[1]{}%
  }%
  \providecommand\rotatebox[2]{#2}%
  \ifx\svgwidth\undefined%
    \setlength{\unitlength}{191.425bp}%
    \ifx\svgscale\undefined%
      \relax%
    \else%
      \setlength{\unitlength}{\unitlength * \real{\svgscale}}%
    \fi%
  \else%
    \setlength{\unitlength}{\svgwidth}%
  \fi%
  \global\let\svgwidth\undefined%
  \global\let\svgscale\undefined%
  \makeatother%
  \begin{picture}(1,0.32792732)%
    \put(0,0){\includegraphics[width=\unitlength]{twistedCartDir1.pdf}}%
    \put(0.68572625,0.16462627){\color[rgb]{0,0,0}\rotatebox{-12.24326978}{\makebox(0,0)[b]{\smash{$\on{s}(\xi_v)\circlearrowleft G\cdots$}}}}%
    \put(0.31811722,0.17230566){\color[rgb]{0,0,0}\rotatebox{12.02905569}{\makebox(0,0)[b]{\smash{$\cdots G\circlearrowright\on{t}(\xi_v)$}}}}%
    \put(0.3339668,0.10403293){\color[rgb]{0,0,0}\rotatebox{12.20818656}{\makebox(0,0)[b]{\smash{$e'$}}}}%
    \put(0.65748251,0.10489783){\color[rgb]{0,0,0}\rotatebox{-13.02021519}{\makebox(0,0)[b]{\smash{$e$}}}}%
    \put(0.48936035,0.13576092){\color[rgb]{0,0,0}\makebox(0,0)[b]{\smash{$v$}}}%
  \end{picture}%
\endgroup%

%% file: GraphFusDel.pdf_tex
\begingroup%
  \makeatletter%
  \providecommand\color[2][]{%
    \errmessage{(Inkscape) Color is used for the text in Inkscape, but the package 'color.sty' is not loaded}%
    \renewcommand\color[2][]{}%
  }%
  \providecommand\transparent[1]{%
    \errmessage{(Inkscape) Transparency is used (non-zero) for the text in Inkscape, but the package 'transparent.sty' is not loaded}%
    \renewcommand\transparent[1]{}%
  }%
  \providecommand\rotatebox[2]{#2}%
  \ifx\svgwidth\undefined%
    \setlength{\unitlength}{377.51691895bp}%
    \ifx\svgscale\undefined%
      \relax%
    \else%
      \setlength{\unitlength}{\unitlength * \real{\svgscale}}%
    \fi%
  \else%
    \setlength{\unitlength}{\svgwidth}%
  \fi%
  \global\let\svgwidth\undefined%
  \global\let\svgscale\undefined%
  \makeatother%
  \begin{picture}(1,0.38446037)%
    \put(0,0){\includegraphics[width=\unitlength]{GraphFusDel.pdf}}%
    \put(0.14151558,0.18126467){\color[rgb]{0,0,0}\makebox(0,0)[lb]{\smash{$\xi_{P}$}}}%
    \put(0.29409153,0.18126467){\color[rgb]{0,0,0}\makebox(0,0)[rb]{\smash{$\xi_{Q}$}}}%
    \put(0.80848467,0.19523438){\color[rgb]{0,0,0}\makebox(0,0)[b]{\smash{$\xi_{Q}\circ\xi_{P}$}}}%
    \put(0.21546786,0.02430563){\color[rgb]{0,0,0}\makebox(0,0)[b]{\smash{$\Gamma$}}}%
    \put(0.81812304,0.02430563){\color[rgb]{0,0,0}\makebox(0,0)[b]{\smash{$\Gamma^*$}}}%
    \put(0.21761908,0.31166665){\color[rgb]{0,0,0}\makebox(0,0)[b]{\smash{$g_{{}_{\overleftarrow{Q}}e_{\overleftarrow{P}}}$}}}%
    \put(0.09471068,0.10399384){\color[rgb]{0,0,0}\makebox(0,0)[lb]{\smash{$g_{{}_{\overleftarrow{P}}e}$}}}%
    \put(0.35324214,0.10399384){\color[rgb]{0,0,0}\makebox(0,0)[rb]{\smash{$g_{e_{\overleftarrow{Q}}}$}}}%
    \put(0.74739663,0.10399384){\color[rgb]{0,0,0}\makebox(0,0)[lb]{\smash{$g_{{}_{\overleftarrow{P}}e}$}}}%
    \put(0.89149613,0.10399384){\color[rgb]{0,0,0}\makebox(0,0)[rb]{\smash{$g_{e_{\overleftarrow{Q}}}$}}}%
  \end{picture}%
\endgroup%

%% file: GraphFusComp.pdf_tex
\begingroup%
  \makeatletter%
  \providecommand\color[2][]{%
    \errmessage{(Inkscape) Color is used for the text in Inkscape, but the package 'color.sty' is not loaded}%
    \renewcommand\color[2][]{}%
  }%
  \providecommand\transparent[1]{%
    \errmessage{(Inkscape) Transparency is used (non-zero) for the text in Inkscape, but the package 'transparent.sty' is not loaded}%
    \renewcommand\transparent[1]{}%
  }%
  \providecommand\rotatebox[2]{#2}%
  \ifx\svgwidth\undefined%
    \setlength{\unitlength}{353.79794922bp}%
    \ifx\svgscale\undefined%
      \relax%
    \else%
      \setlength{\unitlength}{\unitlength * \real{\svgscale}}%
    \fi%
  \else%
    \setlength{\unitlength}{\svgwidth}%
  \fi%
  \global\let\svgwidth\undefined%
  \global\let\svgscale\undefined%
  \makeatother%
  \begin{picture}(1,0.34912319)%
    \put(0,0){\includegraphics[width=\unitlength]{GraphFusComp.pdf}}%
    \put(0.1255104,0.19341681){\color[rgb]{0,0,0}\makebox(0,0)[lb]{\smash{$\xi_{P}$}}}%
    \put(0.28831519,0.19341681){\color[rgb]{0,0,0}\makebox(0,0)[rb]{\smash{$\xi_{Q}$}}}%
    \put(0.83719376,0.20832306){\color[rgb]{0,0,0}\makebox(0,0)[b]{\smash{$\xi_{Q}\circ\xi_{P}$}}}%
    \put(0.20442052,0.02593507){\color[rgb]{0,0,0}\makebox(0,0)[b]{\smash{$\Gamma$}}}%
    \put(0.8474783,0.02593507){\color[rgb]{0,0,0}\makebox(0,0)[b]{\smash{$\Gamma^*$}}}%
    \put(0.07556765,0.27377045){\color[rgb]{0,0,0}\makebox(0,0)[lb]{\smash{$g_{e_{\overleftarrow{P}}}$}}}%
    \put(0.33786426,0.27377045){\color[rgb]{0,0,0}\makebox(0,0)[rb]{\smash{$g_{{}_{\overleftarrow{Q}}e}$}}}%
    \put(0.84329069,0.31064019){\color[rgb]{0,0,0}\makebox(0,0)[b]{\smash{$g_{e_{\overleftarrow{P}}}g_{{}_{\overleftarrow{Q}}e}$}}}%
    \put(0.07556765,0.11096566){\color[rgb]{0,0,0}\makebox(0,0)[lb]{\smash{$g_{{}_{\overleftarrow{P}}e}$}}}%
    \put(0.35143132,0.11096566){\color[rgb]{0,0,0}\makebox(0,0)[rb]{\smash{$g_{e_{\overleftarrow{Q}}}$}}}%
    \put(0.77201032,0.11096566){\color[rgb]{0,0,0}\makebox(0,0)[lb]{\smash{$g_{{}_{\overleftarrow{P}}e}$}}}%
    \put(0.9257704,0.11096566){\color[rgb]{0,0,0}\makebox(0,0)[rb]{\smash{$g_{e_{\overleftarrow{Q}}}$}}}%
  \end{picture}%
\endgroup%